\documentclass[10pt,a4paper]{article}
\usepackage[square,sort,comma,numbers]{natbib}
\usepackage{amsthm,amssymb,amsmath}
\usepackage[pdftex]{thumbpdf}       
\usepackage{float}
\usepackage{graphicx}
\usepackage{verbatim}
\usepackage[usenames,dvipsnames]{color}
\usepackage{dsfont}
\usepackage{fullpage}
\usepackage{ifpdf}
\usepackage{comment}
\numberwithin{equation}{section}

\ifpdf
\usepackage[pdftex]{hyperref} 
\usepackage{authblk} 

\hypersetup{%
pdftitle={Approximations for a single EV charging station},
pdfauthor={ A. Aveklouris},
pdfkeywords={EV, diffusion},
bookmarks=true,%
bookmarksnumbered=true,
pdfstartview={FitH},
pdfpagelayout={OneColumn},
colorlinks=true,
citecolor=NavyBlue,
linkcolor=RedOrange,
urlcolor=BrickRed,
bookmarksopen=true,%
bookmarksopenlevel=1,
unicode=true,%
breaklinks=true,%
}%
\else
  \newcommand\phantomsection\relax
  \newcommand{\url}[1]{#1}
  \newcommand{\href}[2]{#2}
  \usepackage{nohyperref}
\fi

\usepackage{url}
\mathchardef\ordinarycolon\mathcode`\: \mathcode`\:=\string"8000 \begingroup
\catcode`\:=\active \gdef:{\mathrel{\mathop\ordinarycolon}} \endgroup
 \theoremstyle{plain}              
\newtheorem{definition}{Definition}[section]
\newtheorem{remark}{Remark}[section]
\newtheorem{theorem}{Theorem}[section]
\newtheorem{proposition}[theorem]{Proposition}
\newtheorem{lemma}[theorem]{Lemma}
\newcommand{\ind}[1]{\mathds{1}_{\{#1\}}}   
\newcommand{\E}[1]{\mathbb{E}\left[#1\right]}  
\newcommand{\Prob}[1]{\mathbb{P}\left(#1\right)} 
\newcommand{\blt}{\boldsymbol}
\newcommand{\R}{\mathbb{R}}
\newcommand{\pe}{p_e}
\newcommand{\pz}{p_Z}
\newcommand{\Zf}{Z_{f}}
\newcommand{\Zfn}{Z_{f}^n}
\newcommand{\pf}{\pi_{f}}
\newcommand{\zf}{z_{f}^*}
\DeclareMathOperator{\Min}{\wedge}
\DeclareMathOperator{\Max}{\vee}
\graphicspath{{Figures/}}

\title{Bounds and Limit Theorems for a Layered Queueing Model in~Electric~Vehicle~Charging}
\author[1]{Angelos\ Aveklouris \thanks{a.aveklouris@tue.nl} }
\author[1]{Maria\ Vlasiou \thanks{m.vlasiou@tue.nl}}
\author[1,2]{Bert\ Zwart \thanks{Bert.Zwart@cwi.nl}}
\affil[1]{Eindhoven University of Technology}
\affil[2]{Centrum Wiskunde en Informatica}
\begin{document}

\maketitle
\begin{abstract}
The rise of electric vehicles (EVs) is unstoppable due to factors such as the decreasing cost of batteries and various policy decisions. These vehicles need to be charged and will therefore cause congestion in local distribution grids in the future. Motivated by this, we consider a charging station with finitely many parking spaces, in which electric vehicles arrive in order to get charged. An EV has a random parking time and a random charging time. Both the charging rate per vehicle and the charging rate possible for the station  are assumed to be limited. Thus, the charging rate of uncharged EVs depends on the number of cars charging simultaneously. This model leads to a layered queueing network in which parking spaces with EV chargers have a dual role, of a server (to cars) and customers (to the grid). We are interested in the performance of the aforementioned model, focusing on the fraction of vehicles that get fully charged. To do so, we develop several bounds and asymptotic  (fluid and diffusion) approximations for the vector process, which describes the total number of EVs and the number of not fully charged EVs in the charging station, and we compare these bounds and approximations with numerical outcomes.
\end{abstract}

\section{Introduction}\label{sec:Introduction}

The rise of electric vehicles (EVs) is unstoppable due to factors such as the decreasing cost of batteries and various policy decisions \cite{Energy17}. Currently, the bottlenecks are the ability to charge a battery at a fast rate and the number of charging stations, but this bottleneck is expected to move towards the current grid infrastructure.
This is illustrated in \cite{2025scenario}, the authors evaluate the impact of the energy transition on a real distribution grid in a field study, based on a scenario for the year 2025. The authors confront a local low-voltage grid with electrical vehicles and ovens and show that charging a small number of EVs is enough to burn a fuse.
Additional evidence of congestion is reported in \cite{clement2010impact}.
This paper proposes  to model and analyze such congestion by the use of the so-called \emph{layered queueing networks}.
Layered  networks are specific queueing networks where some entities in the system have a dual role; e.g., servers (in our context: parking spaces with EV chargers) become customers to a higher-layer (here: the power grid). The use of layered queueing networks allows us to analyze the interaction of two sources of congestion; first, the number of available spaces with charging stations (as not all cars find a space) and second,  the amount of available power that the power grid is able to feed to the charging station. \cite{2025scenario}

We consider a charging station (or parking lot) with
finitely many parking spaces. Each space has an EV charger connecting with the power grid.
EVs arrive at the charging station randomly in order to get charged. If an EV finds an available space, it enters the parking lot and charging starts immediately.
An EV has a random parking time and a random charging time. It leaves the parking lot only when its parking time expires; i.e., it remains at its space without consuming power until its parking time expires  if finishing its charge withing the parking time.
Both the charging rate per vehicle and the charging rate possible for the complete charging station  are assumed to be limited. Thus, the charging rate of uncharged EVs depends on the number of cars charging simultaneously. Last, we assume that all available power is shared at the same rate
to all cars that need charging. The available power that can delivered by the grid is assumed to be constant.

Using queueing terminology, our model can be described as a two-layered queueing network. An EV enters  the charging station and connects its battery to an EV charger. In our context, EVs play the role of customers, while EV chargers are the servers. Thus, the system of EVs and EV chargers can be viewed as the first layer.
Moreover, EV chargers are connected to the power grid.
Thus, at the second layer, active EV chargers
act as jobs that are served simultaneously by the power grid, which plays the role of a single server.

This paper focuses on the performance analysis of this system under Markovian assumptions. Specifically, we are interested in finding the fraction of fully charged EVs in the charging station, which is equivalent to
the probability that an EV leaves the charging station with a
fully charged battery. A mostly heuristic description of some partial
results in this paper has appeared in \cite{aveklouriselectric}.
We first start with the steady-state analysis of the original system, for which we can find explicit bounds for the fraction of fully charged EVs. To do so, we study three special cases of the original system: i) there is enough power for all EVs, ii) there are enough parking spaces for all EVs, and iii) the parking lot is full. In these cases, we are able to find the explicit joint distribution in steady-state of the total number of EVs and the number of not fully charged EVs in the charging station, which we call \textit{the vector process}.

In order to improve the bounds for the fraction of fully charged EVs, we next develop a fluid approximation for the number of uncharged EVs in the parking lot.
The mathematical results here are closely
related to results on processor-sharing queues with impatience \cite{gromoll2006impact}. However, the model here is more complicated as there is a limited number of spaces in the system and fully charged cars may not leave immediately as they are still parked.

We then move to diffusion approximations, working in three asymptotic regimes. First, we consider the
Halfin-Whitt regime, in which we prove a limit theorem for the vector process, showing that it converges to a two-dimensional reflected Ornstein--Uhlenbeck (OU) process with piecewise linear drift. Then, we consider an overloaded regime for the process describing the number of total EVs in the system. In this case, the limit reduces to a one-dimensional OU process  with piecewise linear drift. Last, we  approximate  the vector process
by a two-dimensional OU process when the parking times are sufficiently large. The mathematical results here are based on martingales arguments \cite{pang2007martingale}.

EVs can be charged in several ways. Our setup can be seen as an example of slow charging, in
which drivers typically park their EV and are not physically present during charging (but are
busy shopping, working, sleeping, etc). For queueing models focusing on fast charging, we refer
to \cite{bayram2013electric, yudovina2015socially}. Both papers consider
a gradient scheduler to control delays. Next,
\cite{Swapping17} presents a queueing model for battery swapping while \cite{Turitsyn2010} is an early
paper on a queueing analysis of EV charging, focusing on designing safe control rules (in term of voltage
drops) with minimal communication overhead.

Despite being a relatively new topic, the engineering literature on EV charging is huge. We can
only provide a small sample of the already vast but still
emerging literature on EV charging. The focus of \cite{su2012performance} is on a specific
parking lot and presents an algorithm for optimally managing a large number of plug-in EVs.
Algorithms to minimize the impact of plug-in EV charging on the distribution grid are proposed in
\cite{sortomme2011coordinated}.
In \cite{li2012modeling}, the overall charging demand of plug-in EVs is considered. Mathematical
models where vehicles communicate beforehand with the grid to convey information about their
charging status are studied in \cite{said2013queuing}. In \cite{kempker2016optimization}, EVs are the central object and a dynamic program is formulated that prescribes how EVs should charge
their battery using price signals.

In addition, layered queueing network have been successfully applied in analyzing interactive networks in communications networks and manufacturing systems.
These are queueing networks where some entities in the system have a dual role. In such systems,
the dynamics in layers are correlated and the service speeds vary over time. Layered queueing
networks can be characterized by separate layers (see \cite{rolia1995method} and
\cite{woodside1995stochastic}) or simultaneous layers (such as our model); \cite{aveklourisSSC}. In the first case, customers receive
service with some delay. An application where layered networks with separate layers appear is the
manufacturing systems e.g., \citep{dorsman2013marginal} and \citep{dorsman2015heavy}. On the
other hand, in layered networks with simultaneous layers, customers receive service from the
different layers simultaneously. Layered networks with simultaneous layers have applications in
communications networks, e.g.\ in web-based multi-tiered system architectures. In such
environments, different applications compete for access to shared infrastructure resources, both
at the software level
 and at the hardware level.
For background, see  \cite{van2001web}, and \cite{van2009dynamic}.

The paper is organized as follows. In Section~\ref{sec:Model}, we provide
a detailed model description -- in particular we introduce our stochastic model and we define the system dynamics. Next, in Section~\ref{sec:explisit results} we present some explicit bounds in steady-state for the fraction of fully charged EVs.
Section~\ref{sec:Asymptotic approximations} contains several asymptotic approximations. First, a fluid approximation is presented; we then derive diffusion limits and  approximations in three asymptotic regimes. Numerical validations are presented in
Section~\ref{sec:numerics}. Last, all proofs are gathered in Section~\ref{proofs}.

\section{Model}\label{sec:Model}
In this section, we provide a detailed formulation of our model and explain various notational
conventions that are used in the remainder of this work.

\subsection{Preliminaries} \label{sec:Preliminaries and model description}
We use the following notational conventions. All
vectors and matrices are denoted by bold letters. Further,  $\R$ is the set of real numbers,
$\R_+$ is the set of nonnegative real numbers, and $\mathbb{N}$ is the set of strictly positive
integers.
For real numbers $x$ and $y$, we define $x\Max y:=\max\{x,y\}$ and
$x\Min y:=\min\{x,y\}$.
Furthermore, $\blt I$ represents the identity matrix and $\blt e$ and $\blt e_0$ are vectors
consisting
of 1's and 0's, respectively, the dimensions of which are clear from the context.
Also, $\blt {e}_i$ is the vector whose $i^\text{th}$ element is 1 and the rest are all 0.

Let $(\Omega,\mathcal{F},\mathbb{P})$
be a probability space. For $T>0,$ let
$\mathcal{D}[0,T]^2:=\mathcal{D}[0,T]\times\mathcal{D}[0,T]$ be the two-dimensional Skorokhod
space; i.e.,
the space of two-dimensional real-valued functions on $[0,T]$ that are right continuous with
left limits endowed with the $J_1$ topology; cf.\ \cite{chen2001fundamentals}. Observe that as all candidate limit objects we consider are continuous, we only need to work with the uniform topology.
It is well-known that the space
$(\mathcal{D}[0,T]^2,J_1)$ is a complete and
separate metric space (i.e., a Polish metric space); \cite{billingsley1999convergence}.
We denote by $\mathcal{B}(\mathcal{D}[0,T]^2)$
the Borel $\sigma-$algebra of $\mathcal{D}[0,T]^2$. We assume that all the processes are defined
from $(\Omega,\mathcal{F},\mathbb{P})$ to $(\mathcal{B}(\mathcal{D}[0,T]^2),\mathcal{D}[0,T]^2).$
Further, we write $X(\cdot):=\{X(t), t\geq 0\}$ to represent a stochastic process and
$X(\infty)$ to represent a stochastic process in steady-state. Moreover, $\overset{d}=$ and
$\overset{d}\rightarrow$ denote  equality and convergence in distribution (weak convergence).
For two random variables $X, Y$, we write $X\leq_{st} Y$ (stochastic ordering) if $\Prob{X>a}\leq \Prob{Y>a}$ for any $a\in \R$.
Further, $\Phi(\cdot)$ and $\phi(\cdot)$ represent the cumulative probability function and the  probability density function (pdf) of the standard Normal distribution, respectively. Last,
let $C^2_b(G)$ denote the space of twice continuously differentiable functions on $G$ such
that their first- and second-order derivatives are bounded.

\subsection{Model description}

We consider a charging station with $K>0$ parking spaces. Each space has an EV charger which is connected to the power grid. EVs arrive
independently at the charging station according to a Poisson process with rate $\lambda$. They
have a random charging requirement  and a random parking time denoted by $B$ and $D$, respectively. The random variables $B$
and $D$ are assumed to be mutually independent and exponentially distributed with rates $\mu$ and
$\nu$, respectively.
If an EV finishes its charge, it remains at its space without consuming power until its parking time expires. We call these EVs \emph{fully charged EVs}. Thus, EVs leave the system only after their parking time expires, which implies that an EV may leave the system without its
battery being fully charged.
Furthermore, if all spaces are occupied, a newly arriving EV does
not enter the system but leaves immediately. As such, the total number of vehicles in the system
can be modeled by an Erlang loss system, though we need a more detailed description of the state space.

We denote by $Q(t)\in \{0,1,\ldots, K\}$ the total number of EVs (charged and uncharged) in the system at time $t\geq 0$,
where $Q(0)$ is the initial number of EVs. Further, we denote by $Z(t)\in\{0,1,\ldots, Q(t)\}$
the number of EVs without a fully charged battery at time $t$ and by $Z(0)$ the
number of such vehicles initially in the system. Thus, $C(t)=Q(t)-Z(t)$ represents the number of EVs
with a fully charged battery at time $t$.

The power consumed by the parking lot is limited and depends on the number of uncharged EVs at
time $t$. We let it be given by the \emph{power allocation function} $L:\R_{+} \rightarrow \R_{+} $,
\begin{equation*}\label{eq: allocation function}
L(Z(t)):= Z(t)\Min M.
\end{equation*}
We assume that the parameter $M$ is given and that $0<M\leq  K$. For example, the parameter $M$ can
depend on the contract between the power grid and the charging station.
Alternatively, $M$ can be thought as
the maximum number of EVs the charging station can charge at a maximum rate, where without loss of generality we can assume that the maximum rate is one. The model is illustrated in Figure~\ref{fig:generalmodel}.
\begin{figure}[!h]
\centering
\includegraphics[scale=0.5]{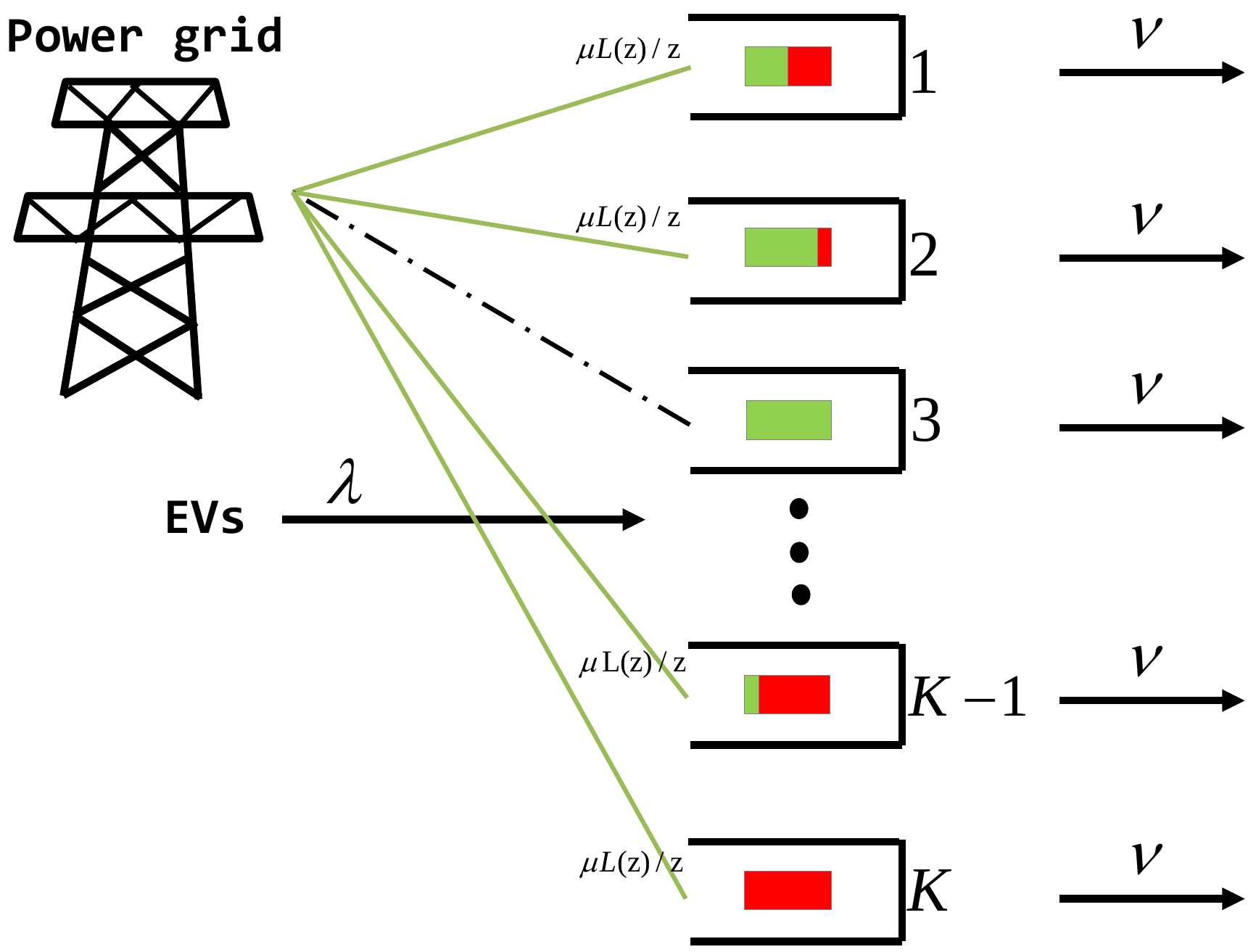}
\caption{A charging station with $K$ EV chargers.}
\label{fig:generalmodel}
\end{figure}

Last, note that the processes $Q(\cdot)$, $Z(\cdot)$, and $C(\cdot)$ depend on $K$ and $M$. We write
$Q^K_M(\cdot)$, $Z^K_M(\cdot)$, and $C^K_M(\cdot)$, when we wish to emphasize this.
It is clear from our context that the two-dimensional process
$\left\{(Q(t),Z(t)): t\geq 0\right\}$, is Markov. The transitions rates in the interior and on the boundary are shown in
Figure~\ref{fig:Transition}.

\begin{figure}[!h]
\centering
\includegraphics[scale=0.5]{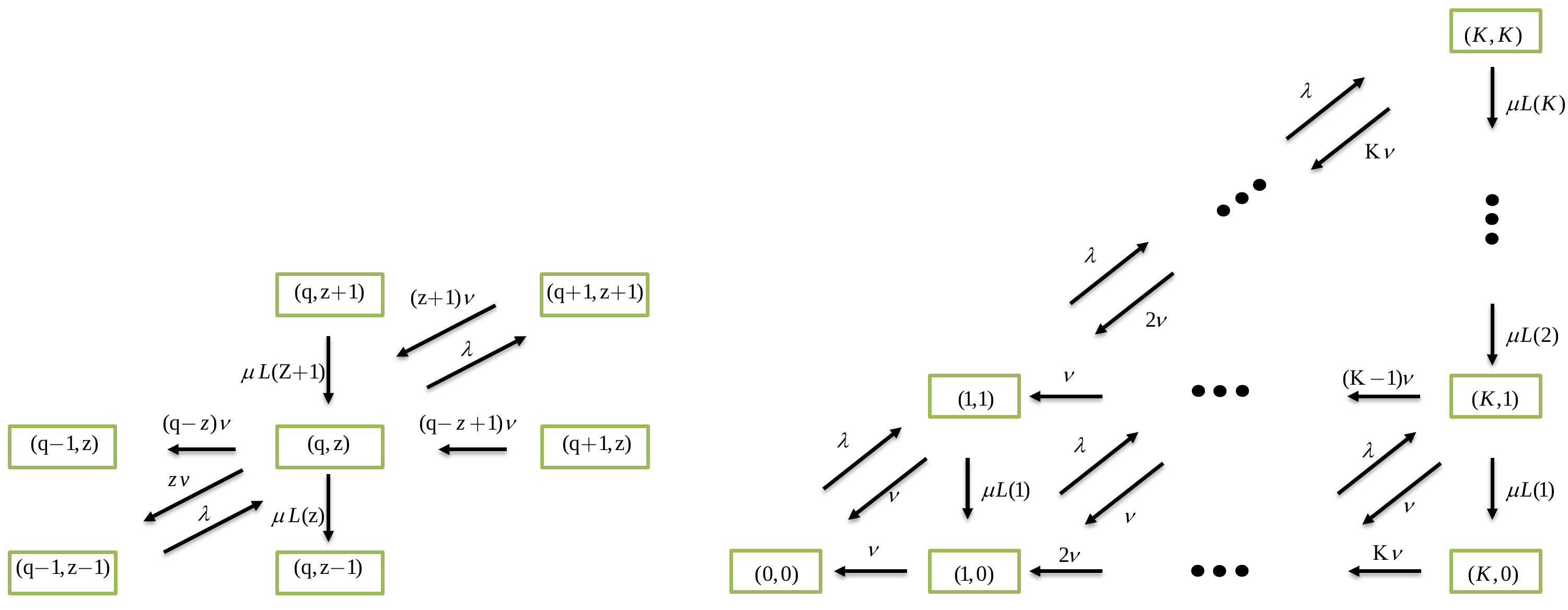}
\caption{Transition rates in the interior (left) and on the boundary (right)
of the process $\left\{(Q(t),Z(t)): t\geq 0\right\}$.}
\label{fig:Transition}
\end{figure}

\subsubsection{Alternative model description in case of infinitely many parking spaces}\label{sec:alternative}
Here, we give an alternative description of our model in case there are infinitely many parking
spaces; i.e., $K=\infty$. In this case, the model can be described as a tandem queue with impatient customers; see Figure~\ref{fig:ErlangAmodel}. EVs arrive at the charging station, which has $M$ servers, and charging starts immediately. There are two possible scenarios.
First, an EV gets fully charged during $D$ and  moves to the second queue, which has and infinite number of
servers. This happens with rate $\mu (Z(t)\Min M)$. In the second queue, EVs get served with rate $\nu C(t)$. In the second scenario, an EV abandons its charging because its parking time expired (and thus leaves the first queue impatiently); this happens with rate $\nu Z(t)$.
Note, that the total ``rate in'' in the system is $\lambda$ and the total ``rate out'' is $\nu(Z(t)+C(t))=\nu Q(t)$. In other words, $Q(t)$ describes the number of customers is an $M/M/ \infty$ queue, i.e., its steady-state distribution is a Poisson distribution with rate
$\lambda/\nu$.
\begin{figure}[!h]
\centering
\includegraphics[scale=0.5]{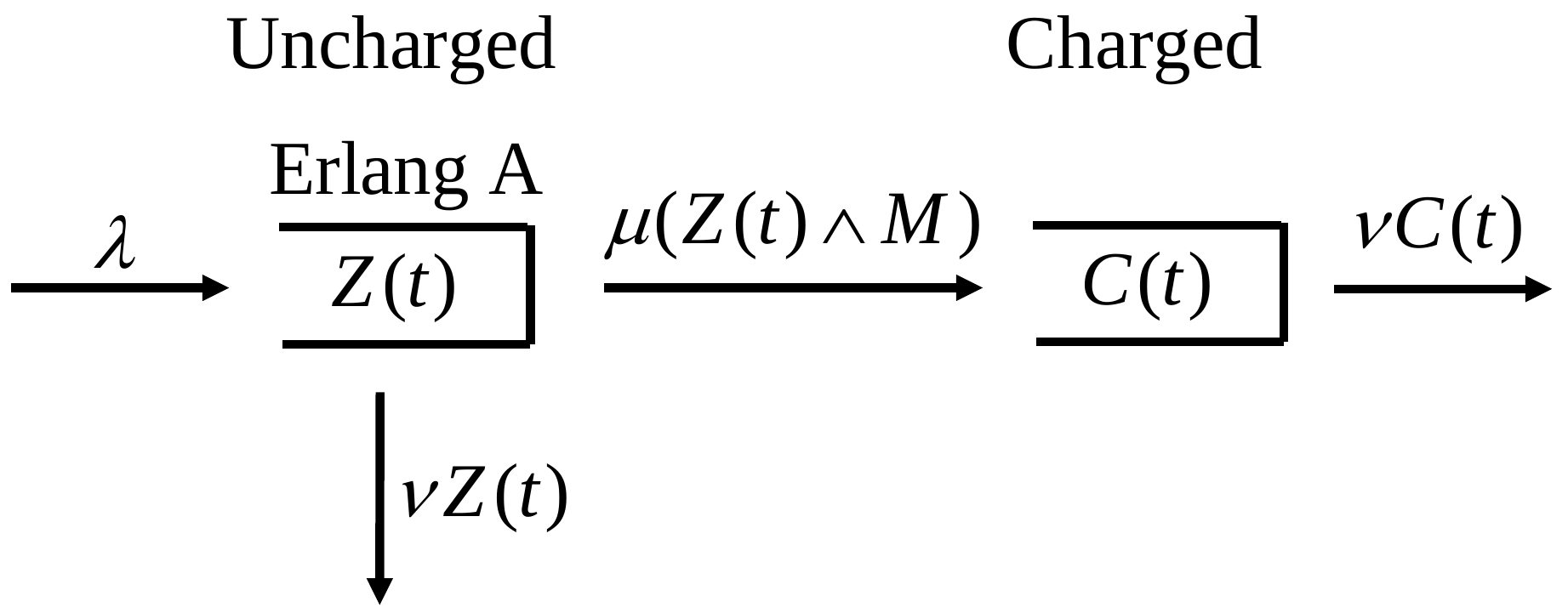}
\caption{Model description in case of infinitely many parking spaces.}
\label{fig:ErlangAmodel}
\end{figure}
As we will see in Proposition~\ref{Prop:erlang A}, the process describing the number of uncharged EVs in the system (i.e., $Z(\cdot)$)
behaves as a modified Erlang-A queue. The transition rates are shown in Figure~\ref{fig:ErlangA}.
\begin{figure}[!h]\label{fig:ErlangA}
\centering
\includegraphics[scale=0.5]{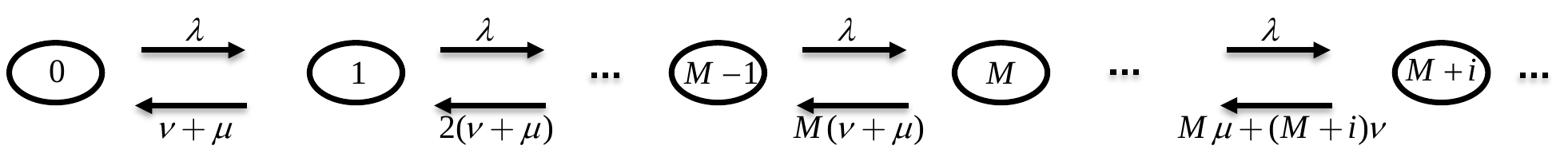}
\caption{Transition rates of the process $Z(\cdot)$ (Erlang-A).}
\label{fig:ErlangA}
\end{figure}

\subsection{System Dynamics}\label{sec:System dynamics}
In this section, we introduce the dynamics that describe the evolution of the system.
We avoid a rigorous sample-path construction of the stochastic processes and we refer to
\cite{chen2001fundamentals, pang2007martingale} for background.

For a constant $r$, let $N_{r}(\cdot)$ be a Poisson process with rate $r$. The total number of EVs in the system at
time $t\geq 0$, $Q(t)$, is given by
\begin{equation}\label{eq:numberQ}
\begin{split}
Q(t)=Q(0)+N_{\lambda} \left( \int_{0}^{t} \ind{Q(s)<K} ds\right)
-N_{\nu,1}\left(\int_{0}^{t}Z(s) ds\right)-N_{\nu,2}\left(\int_{0}^{t}C(s) ds\right),
\end{split}
\end{equation}
where $N_{\lambda}(\cdot)$, $N_{\nu,1}(\cdot)$, and $N_{\nu,2}(\cdot)$ are independent Poisson
processes. Here, the number of EVs that arrive at the charging station during the time
interval $[0,t]$, is given by the process $N_{\lambda}\left( \int_{0}^{t} \ind{Q(s)<K} ds\right)$,
$N_{\nu,1}\left(\int_{0}^{t}Z(s) ds\right)$ is the number of uncharged EVs that depart up to time $t$ and $N_{\nu,2}\left(\int_{0}^{t}C(s)
ds\right)$ counts the departures of fully charged EVs up to time $t$. Hence,
$N_{\nu,1}\left(\int_{0}^{t}Z(s) ds\right)+N_{\nu,2}\left(\int_{0}^{t}C(s) ds\right)$ is the
total number of departures until time $t$ (irrespective of whether the EVs are fully charged or not). In other words,
and by the properties of the Poisson process, we have
\begin{equation}\label{eq:mergeP}
N_{\nu,1}\left(\int_{0}^{t}Z(s) ds\right)+N_{\nu,2}\left(\int_{0}^{t}C(s) ds\right)
\overset{d} = N_{\nu}\left(\int_{0}^{t}Q(s) ds\right),
\end{equation}
and hence \eqref{eq:numberQ} describes the population in a well-known Erlang loss queue
\cite{kelly2014stochastic}.

Another important process is the number of uncharged EVs in
the system, $Z(t)$, which can be written in the following form:
\begin{align}\label{eq:numberU}
Z(t)=Z(0)+N_{\lambda} \left( \int_{0}^{t} \ind{Q(s)<K} ds\right)
-N_{\mu}\left(\int_{0}^{t} L(Z(s)) ds\right)
-N_{\nu,1}\left(\int_{0}^{t} Z(s) ds\right),
\end{align}
where $N_{\mu}\left(\int_{0}^{t} L(Z(s)) ds\right)$ is the number of EVs that get fully charged
during $[0,t]$ and is independent of the aforementioned Poisson processes.

Last, the process which describes the number of fully charged EVs is given by
\begin{equation*}\label{eq:numberC}
C(t)=Q(t)-Z(t)=C(0)
+N_{\mu}\left(\int_{0}^{t} L(Z(s)) ds\right)
-N_{\nu,2}\left(\int_{0}^{t} C(s) ds\right).
\end{equation*}
Observe that in case $K=\infty$, \eqref{eq:numberU} is reduced to the Erlang-A system; \cite{fleming1994heavy,zeltyn2005call}.
All the previous equations hold almost surely and are defined on the same probability space.

It is clear that the vector process
$(Q(\cdot), Z(\cdot))$ constitutes a two-dimensional Markov process. In the sequel, we are interested in finding the joint stationary distribution of
$(Q(\cdot), Z(\cdot))$ and in deriving the fraction of fully charged EVs. Although the computation of the exact joint distribution does not seem promising, we are able to obtain exact bounds for the fraction of fully charged EVs in the next section.

\section{Explicit bounds}\label{sec:explisit results}
The goal of this section is to  give explicit results on some performance measures. In an EV charging setting, one may be interested in finding the fraction of EVs that get fully charged. This is an important performance measure from the point of view of both drivers and of the manager of the charging station. Using arguments from queueing theory, it can be shown that the
fraction of EVs that get fully charged (in steady-state) equals
$P_{s}=1-\frac{\E{Z^K_M(\infty)}}{\E{Q^K_M(\infty)}}$. Note that $P_s$  gives the probability that a vehicle leaves the charging station with
fully charged battery.
Thus, it is clear that the computation of $P_s$ requires the computation of the (joint) distribution of the process $(Q(\cdot), Z(\cdot))$.

For the general model (i.e., for any $K< \infty$ and $M<\infty$) given in Section~\ref{sec:Model},
define the steady-state probabilities $p(q,z):=\lim_{t\rightarrow \infty}
\Prob{Q^K_M(t)=q,Z^K_M(t)=z}$. For simplicity, we use $p(q,z)$ instead of $p^K_M(q,z)$.
These steady-state probabilities are characterized by the following balance equations: for $(q,z)\in \{\R_{+}^2: z\leq q\}$, we have that
\begin{align}\label{eq:balanceEquations}
(q\nu \ind{q>0}+\lambda\ind{q<K}+\mu L(z) \ind{z>0})p(q,z) =
\lambda\ind{z>0} p(q-1,z-1)+(z+1)\nu \ind{q<K} p(q+1,z+1) \nonumber \\
       + \mu L(z+1) \ind{q\neq z}p(q,z+1)+(q-z+1)\nu\ind{q<K}p(q+1,z).
\end{align}
A closed form solution of the balance equations for any $K$ and any $M$ does not seem possible. However, we are able to obtain  explicit solutions in some special cases. Below we derive bounds for $P_s$ based on three different cases:
i) there is  enough power for everyone ($M=\infty$), ii)  there are enough parking spaces for everyone ($K=\infty$), iii) a full parking lot ($Q(t)\equiv K$). In the next proposition, we give upper and lower bounds for the fraction of EVs that get fully charged.

\begin{proposition}\label{prop:bounds}
Let $C^{K}_M(\infty)$ and $Q^{K}_M(\infty)$ be the number of fully charged EVs and the total
number of EVs in steady-state for any $K$ and any $M$. We have that
\begin{equation}\label{In:bounds}
\frac{\E {C^{\infty}_M(\infty)}}{\E {Q^{\infty}_M(\infty)}}
\leq
\frac{\E{ C^{K}_M(\infty)}}{\E {Q^{K}_M(\infty)}}
\leq\frac{\E {C^{K}_K(\infty)}}
{\E {Q^{K}_M(\infty)}}.
\end{equation}
Moreover, an additional lower bound is given by
\begin{equation}\label{In:lower2}
\frac{\E{ Q^{K}_M(\infty)}-\E {\Zf(\infty)}}
{\E{ Q^{K}_M(\infty)}}
\leq
\frac{\E{ C^{K}_M(\infty)}}{\E {Q^{K}_M(\infty)}},
\end{equation}
where $\Zf(\cdot)$ is defined in Section~\ref{sec:full parking lot}.
\end{proposition}
The proof of this proposition makes use of coupling arguments and stochastic ordering of random variables, and it is given in Section~\ref{sec:explisit results}.
We now briefly present the solution of the balance equations for the three special cases described above.

\subsection{Enough power for everyone}

Assume that $K$ is finite and that there is enough power for all EVs to be charged at a maximum rate, i.e., $M=K$.
In this case, the allocation function takes the form $L(Z(t))=Z(t)$, and the balance equations can
be solved explicitly and are given below.
\begin{proposition}\label{Pr: M=K}
Let $K<\infty$ and $M=K$; then the solution $\pe(\cdot,\cdot)$ to the balance equations
\eqref{eq:balanceEquations}
is given by the following (Binomial) distribution
\begin{equation}\label{eq:stationary dis}
\pe(q,z):=p_Q(q) \frac{q!}{z!(q-z)!}
\Big(\frac{\mu}{\nu+\mu}\Big)^{q-z}
\Big(\frac{\nu}{\nu+\mu}\Big)^{z},
\end{equation}
where
\begin{equation}\label{eq:Erlang loss formula}
p_Q(q):=\sum_{z=0}^{q}\pe(q,z)=\frac{1}{q!}\left(\frac{\lambda}{\nu}\right)^qp_Q(0).
\end{equation}
Moreover, the probability of an empty system is given by
\begin{equation*}\label{eq:empty system}
\pe(0,0)=p_Q(0)= \left(\sum_{i=0}^{K}\frac{1}{i!}
\left(\frac{\lambda}{\nu}\right)^i\right)^{-1}.
\end{equation*}
\end{proposition}

\subsection{Enough parking spaces for everyone}
In the second case,  we assume that there are infinitely many parking spaces; i.e., ($K=\infty$) and
$M<\infty$. In this case, all EVs can find a free position and  the process $Z(\cdot)$ can be modeled as a Markov process itself where
its transition rates are given in Figure~\ref{fig:ErlangA}. We see in the next proposition that
the process $Z(\cdot)$ behaves as a modified Erlang-A model with $M$ servers; \cite{zeltyn2005call}. The main difference here is that EVs can leave the system even if they are in service (i.e., are getting charged).
\begin{proposition}\label{Prop:erlang A}
For $z=0,1,\ldots$,  let $\pz(z):=\lim_{t \rightarrow \infty }\Prob{Z^{\infty}_M(t)=z}$ be the
stationary distribution of the Markov process $\{Z^{\infty}_M(t), t\geq 0\}$. It is given by
\begin{equation*}\label{eq:Stationary2}
\pz(z)= \begin{cases}
\frac{1}{z!}\left(\frac{\lambda}{\nu +\mu}\right)^z \pz(0), & \mbox{if }\ z\leq M, \\
\frac{1}{M!}\left(\frac{\lambda}{\nu +\mu}\right)^M\prod_{k=M+1}^{z} \frac{\lambda}{M\mu +k\nu}
\pz(0), & \mbox{if }\ z>M,
\end{cases}
\end{equation*}
where
\begin{equation*}\label{eq:empty system 2}
\pz(0)=\left(\sum_{j=0}^{M} \frac{1}{j!}\left(\frac{\lambda}{\nu +\mu}\right)^j +
\sum_{j=M+1}^{\infty} \frac{1}{M!}
\left(\frac{\lambda}{\nu +\mu}\right)^M\prod_{k=M+1}^{j}
\frac{\lambda}{M\mu +k\nu}
\right)^{-1}.
\end{equation*}
\end{proposition}

\subsection{A full parking lot}\label{sec:full parking lot}
Last, we consider the case where the parking lot is always full, i.e., the total number of EVs (uncharged and charged) is equal to the number of parking spaces. Roughly speaking we assume that the arrival rate is infinite, and that we replace (immediately) each departing EV by a newly arriving
EV, which we assume to be uncharged. Hence, the total number of EVs always remains constant and it is equal to $K$. In other words, the original two-dimensional stochastic model reduces to a one-dimensional model. For this model, we  find its steady-state distribution  below. This result yields an upper bound for the number of uncharged EVs in the original system and hence a lower bound for the fraction of EVs that get fully charged.
As we shall see later, the result in this section plays a crucial role in the study of the diffusion limit in the overloaded regime. Also, in the numerics, we see that a modification of the full parking lot case gives a very good approximation for the fraction of fully charged EVs.

On these assumptions, all newly
arriving EVs are uncharged and so it turns out that the process describing the number of
uncharged EVs in the system, $\{\Zf(t), t\geq 0\}$, is a birth-death process. In particular, the
birth rate is $\nu (K-\Zf(t))$ and the death rate is equal to $\mu (\Zf(t)\Min M)$.
The state-state distribution of the aforementioned birth-death process is given in the following proposition.
\begin{proposition}\label{Prop:overloaded}
The steady-state distribution of the Markov process $\{\Zf(t), t\geq 0\}$ is given by
\begin{equation*}
\pf(z)= \begin{cases}
\left(\frac{\mu}{\nu}\right)^{M-z} \frac{\prod_{j=0}^{M-z-1} (M-j)}{\prod_{j=1}^{M-z} (K-M-j)}
 \pf(M), & \mbox{if }\ 0 \leq z< M, \\
\frac{1}{M^{z-M}}
\left(\frac{\nu}{\mu}\right)^{z-M} \prod_{j=0}^{z-M-1} (K-M-j) \pf(M), & \mbox{if }\ M\leq z \leq
K,
\end{cases}
\end{equation*}
where
\begin{equation*}
\pf(M)=\left(\sum_{l=0}^{M-1}
\left(\frac{\mu}{\nu}\right)^{M-l} \frac{\prod_{j=0}^{M-l-1} (M-j)}{\prod_{j=1}^{M-l} (K-M-j)}
+\sum_{l=M}^{K} \frac{1}{M^{l-M}}
\left(\frac{\nu}{\mu}\right)^{l-M}\ \prod_{j=0}^{l-M-1} (K-M-j)
\right)^{-1}.
\end{equation*}
\end{proposition}

In Section~\ref{sec:numerics}, we validate these bounds in the three regimes; moderately,
critically, and over-loaded. As we will see, the bounds are not very close in general. For this reason, we move to asymptotic approximations.

\section{Asymptotic approximations} \label{sec:Asymptotic approximations}
In this section, we present asymptotic approximations. First, we focus on the fluid approximation and then we move to three diffusion approximations. Consider a family of systems indexed by $n\in\mathbb{N},$ where $n$ tends
to infinity, with the same basic structure as that of the system described in
Section~\ref{sec:Model}. To indicate the position in the sequence of systems,
a superscript $n$ will be appended to the system parameters and processes. In the remainder of this section, we assume that $\E{Q^n(0)}$ and $\E{Z^n(0)} $ are finite. Last, the proofs of the limit theorems are based on martingale arguments and are given in Sections~\ref{sec:HWregime}--\ref{sec:largeparking}. We give a rigorous proof for Theorem~\ref{Th:diffusion} in Section~\ref{sec:HWregime} and we omit the full details for the other proofs.

\subsection{Fluid approximation}\label{sec:Fluid approximation}
Here, we study a fluid model, which is a deterministic model that can be thought of as a formal law of large numbers approximation under appropriate scaling.
We develop a fluid approximation for finite $K$, following a similar approach as in \cite{gromoll2006impact}. The main differences here are the finitely many servers in the system and that the state space consists of two regions: $\{Z(t)>M\}$ and $\{Z(t)\leq  M\}$.

To obtain a non-trivial  fluid limit, we assume that the capacity of power in the $n^{\text{th}}$ system is given by $nM$, the arrival rate  by $n\lambda$, the number of parking spaces by $nK$, and we do not scale the time. The fluid scaling of the process describing the number of uncharged EVs in the charging station is given by $\frac{Z^n(\cdot)}{n}$. This scaling gives rise to the following definition of a fluid model.
\begin{definition}[Fluid model]\label{def:fluid model}
A continuous function $z(t): \R_+ \rightarrow [0,K]$ is a fluid-model solution if it satisfies the ODE
\begin{equation}\label{eq:fluid limit}
z'(t)= \lambda \Min \nu K- \nu z(t) - \mu (z(t)\Min M),
\end{equation}
for $t\in[0,t^*)$, where $t^*=\inf\{s\geq0: z(s)=0\}$ and $z(t) \equiv 0$ for $t\geq t^*$.
\end{definition}
Note that \eqref{eq:fluid limit} can be written as $z'(t)= R(z(t))$, where $R(\cdot)=\lambda \Min \nu K- \nu \cdot - \mu (\cdot\Min M)$. Further, the operator  $R(\cdot)$ is Lipschitz-continuous in $\R_+$, which guarantees that \eqref{eq:fluid limit} has a unique solution.
In the proof of Proposition~\ref{prop:fluid approximation} below, we shall see that if the initial state of the fluid model solution $z(0)\in [0,K]$, then $z(t)\leq K$ for any $t\geq 0$. The last statement ensures that the definition of our fluid model is well-defined.

Next, we see that the fluid model solution can arise as a limit of the fluid scaled process $\frac{Z^n(\cdot)}{n}$.
The  proof of the following proposition is based on martingale arguments and is given in Section~\ref{sec:Fluid ap}.

\begin{proposition}\label{prop:fluid limit}
If $\frac{Z^n(0)}{n} \overset{d}\rightarrow z(0)$ and $\frac{Q^n(0)}{n} \overset{d}\rightarrow K$, then we have that $\frac{Z^n(\cdot)}{n} \overset{d}\rightarrow z(\cdot)$, as $n \rightarrow \infty$. Moreover, the deterministic function $z(\cdot)$ satisfies \eqref{eq:fluid limit}.
\end{proposition}

Moreover, the next proposition states that the fluid model solution converges  to the unique invariant point as time goes to infinity.
\begin{proposition}\label{prop:fluid approximation}
Let $B$ and $D$ be exponential random variables with rates $\mu$ and $\nu$. We have that for any $z(0)\in [0,K]$, $z(t) \rightarrow z^*$ exponentially fast as $t \rightarrow \infty$. In addition, $z^*$ is given by the unique positive solution to the following fixed-point equation
\begin{equation}\label{eq:fluid proxy1}
z^* = (\lambda\Min \nu K) \E{ \min \{D , B \max\{ 1, \frac{ z^*}{M}\}\}}.
\end{equation}
\end{proposition}
In the proof of Proposition~\ref{prop:fluid approximation}, we shall see that if $z(0)=z^*$ then $z(t)=z^*$ for any $t\geq 0$, i.e., $z^*$ is the unique invariant point of \eqref{eq:fluid limit}.
The point $z^*$  can be view as an approximation of the expected number of uncharged EVs in the system for the original (stochastic) model. Observing that
$\E{ \min \{D , B \max\{ 1, \frac{ z^*}{M}\}\}}$
is the actual sojourn time of an EV in the system and that the quantity $(\lambda\Min \nu K)$ plays the role of the arrival rate,  \eqref{eq:fluid proxy1} can be seen as a version of Little's law. Further, if we allow a processor sharing discipline and infinity many servers (i.e., $L(\cdot) \equiv 1$ and $K=\infty$), then \eqref{eq:fluid proxy1} reduces to
\cite[Equation 4.1]{gromoll2006impact}.
\begin{remark}
We shall see in the proof of Proposition~\ref{prop:fluid approximation} that the invariant point $z^*$ has a simpler form than \eqref{eq:fluid proxy1} but the latter holds much more generally.
If the random variables $B$ and $D$ are generally distributed and possibly dependent with
$\E{B\Min D}<\infty$, then \eqref{eq:fluid proxy1} still holds. The mathematical analysis then requires the use of measure-valued processes, which is beyond of the scope of the current work; for a heuristic approach see \cite{aveklourisstochastic}. Thus, we present the proofs only under Markovian assumptions.
\end{remark}

To ensure that $z^*$ is indeed a fluid approximation, we show that we can interchange the fluid and the steady-state limits. First, note that $Z(\cdot)$ has a limiting distribution. To see this, observe that $Z(\cdot)$ is bounded almost surely from above by the queue length of an Erlang A queue with $M$ servers and infinite buffer. Alternatively, we can bound it by the queue length of an $M/G/\infty$ queue. Now, using the same arguments as in \cite{remerova14}, we conclude that $Z(\cdot)$ is a regenerative process and that there exists a stationary limit, $Z(\infty)$. The next proposition says that the stationary scaled sequence of random variables converges to the unique invariant point $z^*$.

\begin{proposition}\label{prop:interchange}
The stationary fluid scaled sequence of random variables $\frac{Z^n(\infty)}{n}$ is tight and
$\frac{Z^n(\infty)}{n} \overset{d}\rightarrow z^*$, as $n \rightarrow \infty$.
\end{proposition}

Note that the arrival rate in an Erlang loss queue is known and it is equal to
$\lambda (1-B(\lambda/\nu,K))$, where $B(\lambda/\nu,K)$ is the blocking probability in a loss system with $K$ servers and traffic intensity $\lambda/\nu$.
Furthermore, $\lambda (1-B(\lambda/\nu,K))$ is asymptotically exact for our fluid approximation in the sense that
$\lambda (1-B(n\lambda/\nu,nK))\rightarrow
(\lambda\Min \nu K)$, as $n \rightarrow \infty$.
To improve the fluid approximation, we replace $(\lambda \Min\nu K)$ by
 $\lambda (1-B(\lambda/\nu,K))$, leading to
\begin{equation}\label{eq:fluid proxy}
z^* = \lambda (1-B(\lambda/\nu,K)) \E{ \min \{D , B \max\{ 1, \frac{ z^*}{M}\}\}}.
\end{equation}
Heuristically, we assume that an EV sees the system in stationarity throughout its sojourn and we use Little's law and a version of the snapshot principle; \cite{Reiman82}.

Having found the fluid approximation for the number of uncharged EVs in the charging station, we derive the fluid approximation for the  fraction of EVs that get successfully charged.
Let $\overline{P}_s$ denote the probability that an EV leaves the parking lot with fully charged battery in the fluid model. It is given by
$\overline{P}_s=\Prob{D > B \max\{ 1, z^*/M\}}$, where $z^*$ is the unique solution of \eqref{eq:fluid proxy}. Under our assumptions, the explicit expression for this probability can be found. That is,
\begin{equation*}\label{eq:fluidFrac}
\overline{P}_s=
\begin{cases}
\frac{\mu}{\nu+\mu},
& \mbox{$z^*\leq M$}, \\
\frac{\mu M}{\lambda (1-B(\lambda/\nu,K))},
& \mbox{$z^*> M$}.
\end{cases}
\end{equation*}
Note that in region $\{z^*\leq M\}$ the fraction of fully charged EVs is nothing else that the probability of minimum of two exponential random variables.

We now focus on the fluid approximation for the number of uncharged EVs when the parking lot is full; Section~\ref{sec:full parking lot}. Analogous to Definition~\ref{def:fluid model}, we can define a fluid model, which we call $z_f(\cdot)$.

\begin{proposition}\label{prop:fluidfull}
Assume that the scaled parking spaces and  the scaled power capacity are given
by  $K^n=K n$ and $M^n=M n$, respectively. If $\frac{\Zfn(0)}{n} \overset{d}\rightarrow z_f(0)$, we have that
 $\frac{\Zfn(\cdot)}{n} \overset{d}\rightarrow z_f(\cdot)$, and $z_f(t)\rightarrow \zf$ as $n$ and $t$ go
to infinity. Further, the limits can be interchanged and  $\zf$ is given by the following formula
\begin{equation}\label{eq:fluid}
\zf=
  \begin{cases}
    \frac{\nu K}{\nu +\mu}, & \hbox{if } \zf\leq M, \\
    \frac{\nu K-\mu M}{\nu}, & \hbox{if } \zf>M.
  \end{cases}
\end{equation}
\end{proposition}
We give a heuristic approach how we can derive \eqref{eq:fluid}, skipping the proof, which can be done by using a similar procedure as in the general case.
The intuition behind \eqref{eq:fluid} is as follows. The actual sojourn time (in steady-state) of an EV in the
system is given by
\begin{equation*}
\E{B \frac{\Zfn(\infty)}{\Zfn(\infty)\Min M}}=\E{B \left(\frac{Z_{f}^n(\infty)}{M^n}\Max 1\right)}.
\end{equation*}
By Little's law, we have that
\begin{equation*}
\E{\Zfn(\infty)}=\nu \E{K^n-\Zfn(\infty)} \E{B \left(\frac{\Zfn(\infty)}{M^n}\Max 1\right)}.
\end{equation*}
Dividing the last equation by $n$, yields
\begin{equation*}
\E{\frac{\Zfn(\infty)}{n}}=\nu \E{K-\frac{\Zfn(\infty)}{n}}
 \E{B \left(\frac{\Zfn(\infty)}{M n}\Max 1\right)}.
\end{equation*}
Now, taking the limit as $n$ goes to infinity, leads to
\begin{equation*}
\zf=
 \frac{\nu (K-\zf) \left(\frac{\zf}{M }\Max 1\right)}{\mu}=
 \frac{\nu (K-\zf)\zf }{\mu \left(\zf\Min M \right)}.
\end{equation*}
Finally, $\zf$ is given by the following fixed-point equation
\begin{equation*}
\mu \left(\zf\Min M \right) = \nu (K-\zf),
\end{equation*}
and solving the last equation leads to \eqref{eq:fluid}.

We shall see in the numerical examples in Section~\ref{sec:numerics} that the fluid approximation  is a good approximation
of the fraction of fully charged EVs
in most cases. However, especially in the underloaded  regime  and for small number of EV chargers, the error becomes larger. In the next section, we  move to diffusion approximations.

\subsection{Diffusion Approximations} \label{sec:DiffAp}

In this section, we show  diffusion limit theorems and diffusion approximations for the process describing
the number of uncharged and the total number of EVs in the parking lot (\emph{vector process}). To do it, we follow
the strategy set up in \cite{pang2007martingale} using the martingale representation.

First, we work in the Halfin-Whitt regime; Section~\ref{sec:HWT}.
Using the ``square-root staffing rule'' to scale the system parameters, we extend \cite[Theorem~7.1]{pang2007martingale} and we obtain a limit which is a reflected two-dimensional OU process with piecewise linear drift. Then, we derive an equation which characterizes its steady-state distribution, the so-called Basic Adjoint Relation (BAR). However, it turns out that the  computation of the steady-state distribution is a hard problem, which is beyond the scope of this paper and it remains an open problem. To overcome this difficulty we consider more tractable asymptotic regimes.

The second asymptotic regime we consider here is an overloaded regime; Section~\ref{sec:Difover}. Assuming that the process describing the total number of EVs is in an overloaded regime and using the ``square-root staffing rule'' to scale the total power capacity in the system, we can show that the scaled vector process converges weakly to a one-dimensional limit. Thus, we can compute its steady-state distribution.

Last, in Section~\ref{sec:DifsmallparkingT}, we focus on the case that the parking times of the EVs are sufficiently large. We give a heavy traffic limit and a
two-dimensional approximation for the vector process.

\subsubsection{Diffusion approximation in Halfin-Whitt regime  }\label{sec:HWT}
The main goal in this section is to prove a two-dimensional diffusion limit for the vector process. For $-\infty < \beta, \kappa< \infty$, consider the following scaling:
\begin{enumerate}
\item  $\lambda^n=n(\nu+\mu)$,
\item $M^n=\frac{\lambda^n}{\nu+\mu}+\beta\sqrt{n}$,
\item
$K^n=\frac{\lambda^n}{\nu}+\kappa \sqrt{n}$.
\end{enumerate}
Define a sequence of diffusion-scaled processes
$\hat Q^n(\cdot):=\frac{Q^n(\cdot)-\frac{\lambda^n}{\nu}}{\sqrt{n}}$
and
$\hat Z^n(\cdot):=\frac{Z^n(\cdot)-\frac{\lambda^n}{\nu+\mu}}
{\sqrt{n}}$. The allocation function in the $n^{\text{th}}$ system is given by
$L^n(Z^n(\cdot)):= Z^n(\cdot)\Min M^n$. We can then prove the following theorem.

\begin{theorem}[Diffusion limit in Halfin-Whitt regime]\label{Th:diffusion}
	Supposing that $(\hat Z^n(0),\hat Q^n(0)) \overset{d} \rightarrow
	(\hat Z(0),\hat Q(0))$ as $n \rightarrow \infty$, then
	$(\hat Z^n(\cdot),\hat Q^n(\cdot))$
$\overset{d} \rightarrow (\hat Z(\cdot),\hat Q(\cdot))$.
The limit satisfies the following two-dimensional stochastic differential equation
\begin{align}\label{eq:SDE}
	\left(
      \begin{array}{cc}
	d\hat Z(t) \\
	d\hat Q(t)
	\end{array}\right)
	=
	\left(
      \begin{array}{cc}
	b_1(\hat Z(t)) \\
	b_2(\hat Q(t))
	\end{array}\right)
	dt
	+\left(
      \begin{array}{cc}
	\sqrt{2 (\nu+\mu)}& 0  \\
	0 & \sqrt{2 (\nu+\mu)}
	\end{array}\right)
	\left(
      \begin{array}{cc}
	dW_{\hat Z}(t) \\
	dW_{\hat Q}(t)
	\end{array}\right)
	-
	\left(
      \begin{array}{cc}
	d\hat Y(t) \\
	d\hat Y(t)
	\end{array}\right),
\end{align}
where $b_1(x)=-\mu (x \Min \beta)-\nu x$ and $b_2(x)=-\nu x$. Further, $W_{\hat Z}(\cdot)$ and
$W_{\hat Q}(\cdot)$ are driftless, univariate Brownian motions such that
$2(\nu+\mu)\E{W_{\hat Z}(t) W_{\hat Q}(t)}=
(2\nu+\mu) t$. In addition, $\hat Y(\cdot)$ is the unique nondecreasing
nonnegative process such that \eqref{eq:SDE} holds and
$\int_{0}^{\infty} \ind{{\hat Q}(t)<\kappa}d \hat Y(t)=0$.
\end{theorem}
Adapting \cite[Section~7.3]{pang2007martingale}, we can show that the last theorem also holds if we allow the arrival process to be a general stochastic process under the assumption that it is satisfies the functional central limit theorem.

The proof of Theorem~\ref{Th:diffusion} is given in Section~\ref{sec:HW} and is organized as
follows.
\begin{enumerate}
  \item We first establish a continuity result and show the existence and uniqueness
  of the candidate limit. (Proposition~\ref{Th: existence}.)
  \item We then rewrite the system dynamics using appropriate martingales and filtrations;
  see equations \eqref{Eq:difQ}, \eqref{Eq:difU}, and Proposition~\ref{Pr:Martingales}.
  \item Next, we show in Proposition~\ref{Pr:fluid limit} that the corresponding fluid-scaled
      processes converge weakly to deterministic functions.
  \item Last, the proof of Theorem~\ref{Th:diffusion} is done by applying the martingale
      central limit theorem in \cite{ethier1986markov} and Proposition~\ref{Th: existence}.
\end{enumerate}

Next, we focus on characterizing the joint steady-state distribution of the
limit given by \eqref{eq:SDE}.
Our approach is to find a functional equation which describes the joint steady-state distribution,
the so-called \emph{Basic Adjoint Relation}. The next step to use the BAR in order to obtain a key
relation for the moment generating function of the vector proces. The piecewise linear drift and the existence of the reflection in \eqref{eq:SDE} makes the key relation complicated and its analysis is beyond of the scope of this paper.

For any $t\geq 0$, we know that  $\hat Z(t)\in \R$ and $\hat Q(t)\leq \kappa$. It is more convenient to transform the previous
processes such that $\hat Z(t)\in \R$ and $\kappa-\hat Q(t)\geq
0$. To do so, we recall that $b_2(x)=-\nu x$. Thus,
the diffusion limit can be written in the following integral form -- see \eqref{eq:SDE}:
\begin{align*}
\hat Z(t) &= \hat Z(0)+ \int_{0}^{t} b_1(\hat Z(s))ds+
\sqrt{2(\nu+\mu)}W_{\hat Z}(t)-\hat Y(t),\\
 \hat Q(t) &= \hat Q(0)- \nu\int_{0}^{t} \hat Q(s)ds+
 \sqrt{2(\nu+\mu)}W_{\hat Q}(t)-\hat Y(t),
\end{align*}
where $\hat Y(\cdot)$ is defined in Theorem~\ref{Th:diffusion}.
Multiplying by $(-1)$, adding and subtracting the terms $\kappa$, $\nu \kappa t$ in the last equation,
we obtain
\begin{align*}
\kappa- \hat Q(t) =\kappa-  \hat Q(0)
+\nu\int_{0}^{t}( \kappa+\hat Q(s)-\kappa)ds
-\sqrt{2(\nu+\mu)}W_{\hat Q}(t)+\hat Y(t).
\end{align*}
Defining $\hat Q_{\kappa}(t):=\kappa-\hat Q(t)$ for $t \geq 0$, we have that
\begin{align*}
\hat Q_{\kappa}(t) &= \hat Q_{\kappa}(0)+
\int_{0}^{t} b_{\kappa}(\hat Q_{\kappa}(s))ds-
 \sqrt{2(\nu+\mu)}W_{\hat Q}(t)
 +\hat Y(t),
\end{align*}
where $b_{\kappa}(x)=\nu (\kappa-x)$. The process $\hat Q_{\kappa}(t)$ represents the number
of available spots in the parking lot  at time $t\geq 0$ (after scaling and after taking the limit as $n$ goes to infinity). Furthermore, $\hat Y(t)$ increases
if and only if
$\hat Q_{\kappa}(t)=0$.
Define
$\blt X(\cdot):=(\hat Z(\cdot),\hat Q_{\kappa}(\cdot) )$
and note that each component of $\blt X(\cdot)$
is a semimartingale.
Let $G=\{\blt{x}:=(x_1,x_2)\in \R^2: x_2>0\}$. The boundary and the closure of $G$ are given by
$\partial G=\{\blt{x}\in \R^2: x_2=0\}$ and $\bar{G}=G \bigcup \partial G$, respectively.
Now, observe that
$\blt {X}(\cdot)\in \bar{G}$ for any $t\geq 0$.
A geometrical representation of the space $G$ and its boundary is shown in the next figure.
\begin{figure}[h]
\centering
\includegraphics[scale=0.5]{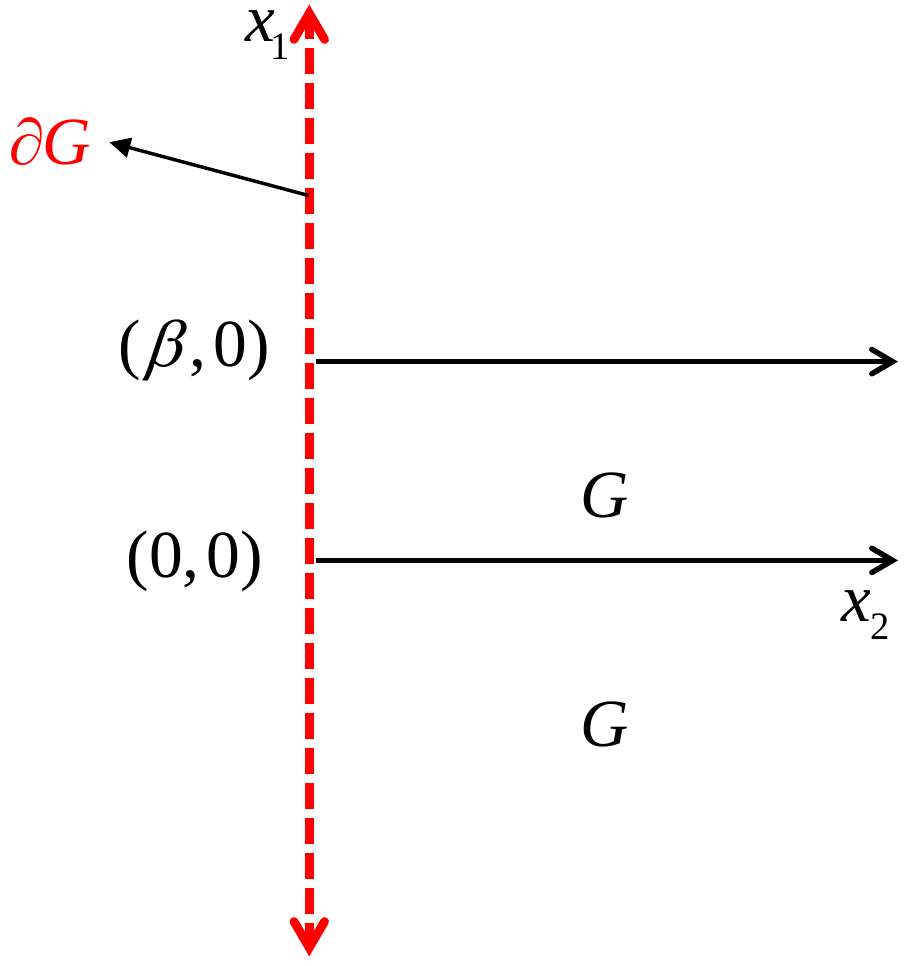}
\caption{The spaces $G$ and its boundary for $\beta >0$.}
\end{figure}

Before we continue the analysis of deriving the BAR, we note some properties for the
process $\hat{Y}(\cdot)$, which is known as \emph{regulator}. It is known that $(\hat{Q}(\cdot),
\hat{Y}(\cdot))$ satisfies a one-dimensional reflection mapping (or one-dimensional Skorokhod
problem).
The regulator $\hat{Y}(\cdot)$ is continuous, nondecreasing and has the property
$$\int_{0}^{\infty} \ind{{\hat Q}_{\kappa}(t)>0}d \hat Y(t)=0,$$
or equivalently for all $t\geq 0$,
$$\hat Y(t)=\int_{0}^{t} \ind{\hat {Q}_{\kappa}(t)=0}d \hat Y(s).$$
By \cite[Theorem 6.1]{chen2001fundamentals}, almost all the paths of the regulator are Lipschitz
continuous on the space
$\{x(\cdot) \in D(0,\infty), x(0)\geq 0\}$ under the uniform topology and hence absolutely
continuous.
From the latter, it follows that $\hat{Y}(\cdot)$ is of bounded variation.
Moreover, by \cite[Theorem 2.2]{xing2009} there exists a (positive) constant $w$ such that
\begin{equation}\label{eq:propertyR}
\hat Y(t)= w\int_{0}^{t} \ind{ \hat {Q}_{\kappa}(s)=0}ds.
\end{equation}
For more details we refer to \cite[Lemma 3.1]{bo2011} and \cite{karatzas1991}.

In the sequel, we focus on deriving a functional equation which characterizes the steady-state
distribution $\pi(\cdot,\cdot)$ of the process $\{\blt {X}(t), t\geq 0\}$, provided that it exists.
To handle the boundary of the space $G$, we define a measure $\sigma$ on $(\bar G,
\mathfrak{B}(\bar G))$ given by
\begin{equation}\label{eq:boundary measure}
\sigma(B)= \mathbb{ E}^{\pi}[\int_{0}^{1} \ind{\blt X(s) \in B} d \hat Y(s)], \ B \in
\mathfrak{B}(\bar G).
\end{equation}
Further, it follows by \eqref{eq:propertyR} that $\sigma(B)\leq w \mathbb{E}^{\pi}[ \hat
Y(1)]<\infty$, which yields that $\sigma$ is a finite measure. Moreover, we define it (for simplicity)
in $\bar G$ but as $\hat Y(\cdot)$ increases only on the boundary $\partial G$, the measure
concentrates on the boundary. In other words, $\sigma$ consists a finite boundary measure.

Using results from \cite{kang2014} and It\^{o} calculus, the BAR takes the following form
\begin{equation}\label{eq:BAR}
  \begin{split}
  \int_{\bar G}\mathcal{L}f(\blt x) \pi (d \blt x)
 -
 \int_{\bar G}\frac{\partial f}{\partial x_1}(\blt x) \sigma(d \blt x)
 +
 \int_{\bar G}\frac{\partial f}{\partial x_2}(\blt x) \sigma(d \blt x)=0,
\end{split}
\end{equation}
where the boundary measure $\sigma$ is defined in \eqref{eq:boundary measure} and $\mathcal{L}$
is the second order operator, i.e.,
\begin{align*}
  \mathcal{L}f(\blt x)= b_1(x_1) \frac{\partial f}{\partial x_1}(\blt x)
  +b_{\kappa}(x_2) \frac{\partial f}{\partial x_2}(\blt x)
  +(\nu+\mu)
  \frac{\partial^2 f}{\partial x_1 \partial x_1}(\blt x)
  +(\nu+\mu)
  \frac{\partial^2 f}{\partial x_2 \partial x_2}(\blt x)
  \\-(2\nu+\mu)
  \frac{\partial^2 f}{\partial x_1 \partial x_2}(\blt x).
\end{align*}

The next step is to derive a key relation between the moment generating functions of $\pi$ and $\sigma$.
Let us define the two-dimensional moment generating function (MGF) of
$\pi$,
\begin{equation*}\label{MGFp}
G^{\pi}(\blt \theta):=\mathbb{ E}^{\pi}[e^{\blt \theta \cdot \blt X(\infty) }]=
\int_{\bar G} e^{\blt \theta \cdot \blt x } \pi(d\blt{x}),
\end{equation*}
for $\blt{\theta}:=(\theta_1,\theta_2)\in \R^2$, and
$\blt \theta \cdot \blt x:=\theta_1x_1+\theta_2x_2$.
In the same way, we define the one-dimensional MGF of $\sigma$,
\begin{equation*}\label{MGFs}
G^{\sigma}(\theta_1):=
\int_{\bar G} e^{\theta_1 x_1 } \sigma(d \blt{x}).
\end{equation*}
Further, we assume that there exists a set $\Theta$ such that $\Theta=\{\blt{\theta}\in \R^2:
 G^{\pi}(\blt \theta)<\infty,\ G^{\sigma}(\theta_1)<\infty \}$. Assuming that
 $\blt{\theta}\in\Theta$ and adapting \cite[Lemma~4.1]{dai2011ST}, we derive the following key relation
\begin{align}\label{eq:GF}
 -\mu \beta\theta_1 \mathbb{E}^{\pi}
 [\ind{Z(\infty)>\beta} e^{\blt X  (\infty) \cdot \blt \theta}]
 -\mu \theta_1 \mathbb{E}^{\pi}
 [Z(\infty)\ind{Z(\infty)\leq \beta} e^{\blt X(\infty) \cdot \blt \theta}]
 -\nu \theta_1 G_{\theta_1}^\pi(\blt \theta)\nonumber\\
 -\nu \theta_2 G_{\theta_2}^\pi(\blt \theta)
+\gamma(\blt{\theta})
  G^\pi(\blt \theta)
  +(\theta_2-\theta_1)
  G^\sigma(\theta_1)=0,
\end{align}
where $\gamma(\blt{\theta})=\nu \kappa \theta_2
+(\nu+\mu) (\theta_1^2+\theta_2^2)
-(2\nu+\mu) \theta_1 \theta_2$ and $G_{\theta_i}^\pi(\cdot)$ denotes the derivative with respect
to $\theta_i$,\ $i=1,2$.

Equation \eqref{eq:GF} is rather complicated due to the piecewise linear term and the existence of the boundary measure.
Although the analysis of \eqref{eq:GF} is beyond of the scope of the current paper, we conjecture that the Wiener-Hopf method \cite{cohen1975} and boundary value techniques \cite{cohen2000} may be applied.

It turns out that \eqref{eq:GF} remains quite complicated even if we assume $K=\infty$, i.e., no boundary measure. Contrary to the one-dimensional case \cite{browne1995}, the steady-state distribution of $(Z(\cdot), Q(\cdot))$ cannot be written as a  linear combination of two distributions. To see this, define $\pi_{-}(\blt {x})$ to be a bivariate Normal distribution with mean vector
$\blt{\mu}_{-}= (0,0)$ and covariance matrix
$ \blt {\Sigma}_{-}=\left(
                     \begin{array}{cc}
                       1 & 1 \\
                        1 & \frac{\nu+\mu}{\nu} \\
                     \end{array}
                   \right)$.
In addition, let $\pi_{+}(\blt {x})$ be a bivariate Normal distribution with mean vector %
$\blt{\mu}_{+}=(- \mu \beta/\nu,0) $ and covariance matrix
$\blt {\Sigma}_{+}=\left(
                     \begin{array}{cc}
                       \frac{\nu+\mu}{\nu} & \frac{2\nu+\mu}{2\nu} \\
                        \frac{2\nu+\mu}{2\nu} & \frac{\nu+\mu}{\nu} \\
                     \end{array}
                   \right)$.
The distributions $\pi_{-}$ and $\pi_{+}$ correspond to the solution of the
Kolmogorov forward equations (or Fokker--Planck equation) of \eqref{eq:SDE} with drift function
$-(\nu+\mu)x$ and $-\mu \beta-\nu x $, respectively.
Adapting  \cite{browne1995}, we define
$\pi_{\infty}(\blt {x}):=
c_1 \pi_{-}(\blt {x})\ind{x_1\leq  \beta}+
c_2 \pi_{+}(\blt {x})\ind{x_1> \beta}$,
where the constants $c_1$, $c_2$ are given by \cite[Eqs. 3.9, 3.10]{fleming1994heavy}. Namely, we have that
\begin{equation*}
 c_1= \left(\Phi(\beta)+
             \sqrt{\frac{\nu+\mu}{\nu}} \exp\left
             \{\frac{\mu \beta^2}{2\nu}\right\} \left(
             1-\Phi\left(
             \sqrt{\frac{\nu+\mu}{\nu}}\beta\right)
              \right)
       \right)^{-1},
\end{equation*}
\begin{equation*}
 c_2= c_1\sqrt{\frac{\nu+\mu}{\nu}}\exp\left\{\frac{\mu \beta^2}{2\nu}\right \},
\end{equation*}
where $\Phi(\cdot)$ represents the cumulative probability function of the standard Normal distribution. It can be easily verified that $\pi_{\infty}$ is indeed a probability distribution but it does not satisfy the correct marginal distribution of $\hat{Q}(\infty)$,
 as it is shown in Figure~\ref{fig:pdfs}.
It is well-known that $\hat{Q}(\infty)$ follows
a Normal distribution with zero mean and variance $\frac{\nu+\mu}{\nu}$.
For a discussion on this topic see \cite{dai2011}.
\begin{figure}[!h]
\centering
\includegraphics[scale=0.8]{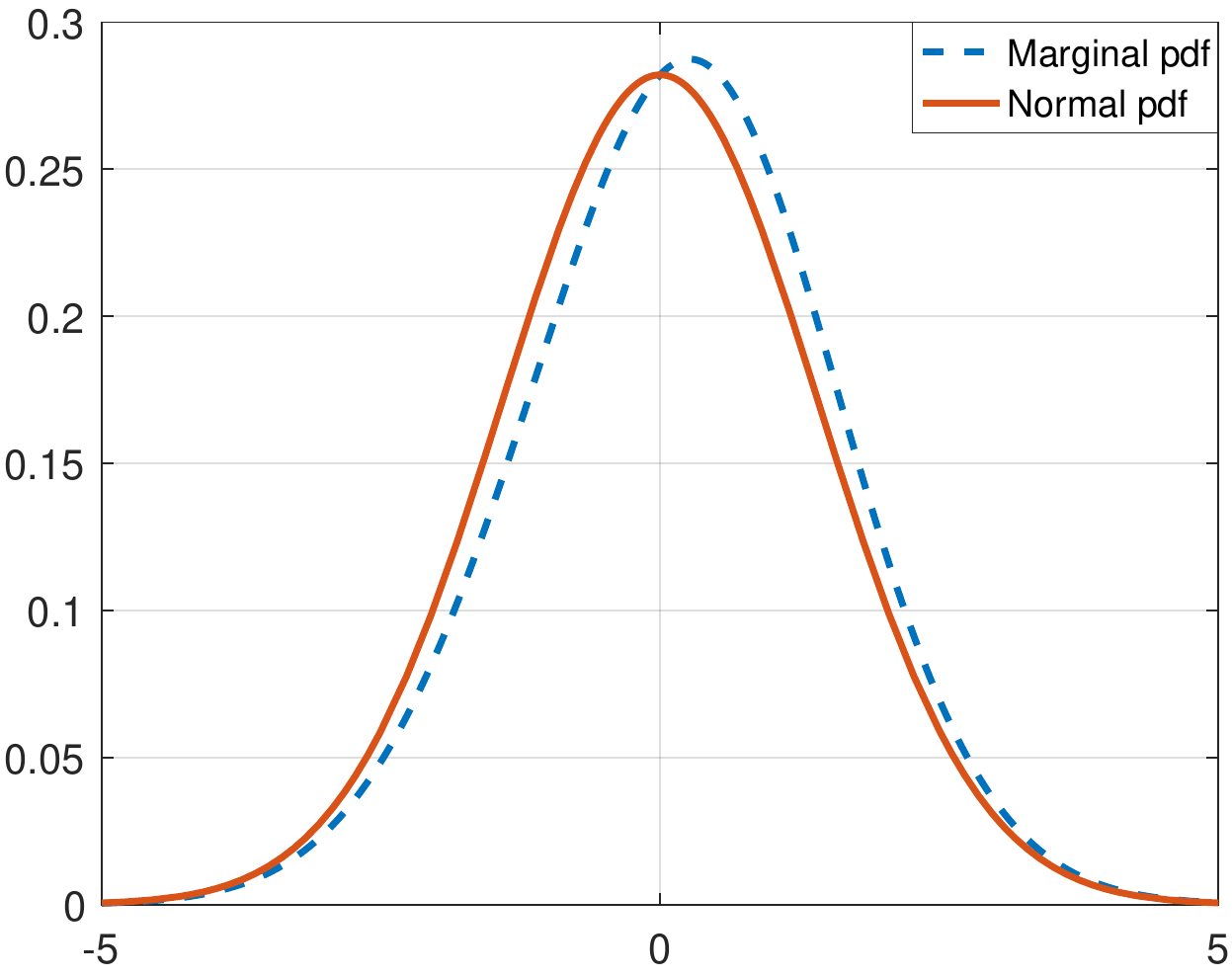}
\caption{Marginal pfd of $\hat{Q}(\infty)$ and Normal$(0,\frac{\nu+\mu}{\nu})$ pdf
for $\beta =0$ and $K=\infty$.}
\label{fig:pdfs}
\end{figure}
In the sequel, we move in different asymptotic regimes in order to overcome this difficulty.

\subsubsection{Diffusion approximation in an overloaded parking lot}
\label{sec:Difover}
In this section, we study an overloaded parking lot. First, we show a diffusion limit in the case that the parking lot is always full (see Section~\ref{sec:full parking lot}). We then show that the diffusion scaled vector process for the original system collapses to the one-dimensional limit in this case. Our motivation in this section comes from results in \cite{aldous1986}. Specifically, the authors there show that under an appropriate scaling (including parameter and time scaling), the number of empty spaces in an overloaded parking lot behaves like an $M/M/1$ queue. However, here we need a modification of this result by dropping the time scaling.

First, we define the dynamical equation that describes the evolution of the process of the number of uncharged EVs when the parking lot is always full. Let $N_{\nu}^f(\cdot)$ and $N_{\mu}^f(\cdot)$ denote two independent Poisson processed with rates
$\nu$ and $\mu$, respectively. For any $t\geq 0$, we have that
\begin{equation}\label{eq:newZ}
  \Zf(t)= \Zf(0)+N_{\nu}^f\left( \int_{0}^{t}(K-\Zf(s))ds\right)
-N_{\mu}^f\left( \int_{0}^{t}\Zf(s)\Min Mds\right),
\end{equation}
where $\Zf(0)\leq K$ almost surely.

Next, introduce our asymptotic regime. Take $K^n$ and $M^n$ such that $K^n=nK$ and
$M^n=\frac{\nu}{\nu+\mu}K^n+\sqrt{n}\beta$, where $K,\beta\geq 0$. The following proposition
gives a diffusion limit for the scaled process describing the number of uncharged EVs, i.e.,
$\Zf(\cdot)$.
\begin{proposition}\label{prop:hugear}
Let the scaled process $\hat{Z}^n_f(\cdot):=\frac{\Zfn(\cdot)-\frac{\nu}{\nu+\mu}K^n}{\sqrt{n}}$. If
$\hat{Z}^n_f(0)\overset{d}\rightarrow \hat{Z}_f(0)$, then
$\hat{Z}^n_f(\cdot)\overset{d}\rightarrow \hat{Z}_f(\cdot)$, where the limit satisfies the
following stochastic differential equation
\begin{equation*}
 d\hat{Z}_f(t)= v(\hat{Z}_f(t))dt+
\sqrt{\frac{2 \nu \mu K }{\nu+\mu}} dW(t).
\end{equation*}
Moreover, the drift function is given by
$v(x)=
\begin{cases}
  -(\nu+\mu)x, & \mbox{if } x\leq \beta, \\
  -\nu x-\mu \beta, & \mbox{if } x>\beta,
\end{cases}$
and $W(\cdot)$ is a standard Brownian motion.
\end{proposition}

Next, we give the steady-state distribution of the process $\hat{Z}_f(\cdot)$. This
can be done following \cite[Equation~3]{browne1995}.
Define the following truncated Normal probability density functions
\begin{equation*}
\pf^-(x)= \frac{\phi\left(\frac{x}{\sigma_1}\right)}
{\sigma_1 \Phi\left(\frac{\beta}{\sigma_1}\right)}, \mbox{ for } x\leq \beta,
\mbox{ and }
\pf^+(x)= \frac{\phi\left(\frac{x}{\sigma_2}\right)}
{\sigma_1\left(1- \Phi\left(\frac{\beta-\frac{\mu \beta}{\nu}}{\sigma_2}\right)\right)},
\mbox{ for } x> \beta,
\end{equation*}
 where $\sigma_1^2=\frac{1}{\nu+\mu}\frac{ \nu \mu K }{\nu+\mu}$ and
$\sigma_2^2=\frac{ \mu K }{\nu+\mu}$.
Now, the pdf of $\hat{Z}_f(\cdot)$ is given by
\begin{equation}\label{eq:SSfull}
\pf(x)= d_1 \pf^-(x) \ind{x\leq \beta}+d_2 \pf^+(x) \ind{x> \beta},
\end{equation}
for $x\in \R$. Moreover,  the constants are  $d_1=\frac{1}{1+r}$ and $d_2=1-d_1$ with
$r=\frac{\sigma_1^2}{\sigma_2^2}\frac{\pf^{-}(\beta)}{\pf^{+}(\beta)}$.

Having studied the system when it is always full, we now move to the original stochastic
model.
The first step is to find a relation between the process that gives the number of uncharged EVs and the
process that gives the empty parking spaces.
Recall that the total number of  EVs in the system is given by the following equation
\begin{equation*}
\begin{split}
Q(t)=Q(0)+N_{\lambda} \left( \int_{0}^{t} \ind{Q(s)<K} ds\right)
-N_{\nu,1}\left(\int_{0}^{t}Z(s) ds\right)-N_{\nu,2}\left(\int_{0}^{t}
\left(Q(s)-Z(s)\right) ds\right)
\end{split}
\end{equation*}
and the number of uncharged EVs is
\begin{align*}
Z(t)=Z(0)+N_{\lambda} \left( \int_{0}^{t} \ind{Q(s)<K} ds\right)
-N_{\mu}\left(\int_{0}^{t} Z(s)\Min M ds\right)
-N_{\nu,1}\left(\int_{0}^{t} Z(s) ds\right).
\end{align*}
Define the stochastic process that describes the number of empty parking spaces in the system,
$E(t):=K-Q(t)$, $t\geq 0$.
Using the definition of the process $E(\cdot)$, we have that the system dynamics  can be rewritten as follows
\begin{align}
E(t)&=E(0)+N_{\lambda} \left( \int_{0}^{t} \ind{E(s)>0} ds\right)
-N_{\nu,1}\left(\int_{0}^{t}Z(s) ds\right)
-N_{\nu,2}\left(\int_{0}^{t}
\left(K-E(s)-Z(s)\right) ds\right), \label{eq:emptyspaces}\\
Z(t)&=Z(0)+N_{\lambda} \left( \int_{0}^{t} \ind{E(s)>0} ds\right)
-N_{\mu}\left(\int_{0}^{t} Z(s)\Min M ds\right)
-N_{\nu,1}\left(\int_{0}^{t} Z(s) ds\right) \label{eq:alertZ}.
\end{align}
By \eqref{eq:emptyspaces}, it follows that
\begin{align*}
N_{\lambda} \left( \int_{0}^{t} \ind{E(s)>0} ds\right)=
E(t)-E(0)
+N_{\nu,1}\left(\int_{0}^{t}Z(s) ds\right)
+N_{\nu,2}\left(\int_{0}^{t}
\left(K-E(s)-Z(s)\right) ds\right). \\
\end{align*}
Applying the last equation in \eqref{eq:alertZ}, yields
\begin{align}
Z(t)&=Z(0)+E(0)-E(t) + N_{\nu,2}\left(\int_{0}^{t}
\left(K-E(s)-Z(s)\right) ds\right)
-N_{\mu}\left(\int_{0}^{t} Z(s)\Min M ds\right). \label{eq:keyrelation}
\end{align}
The last relation and an asymptotic bound for the process $E^n(\cdot)$ (see Proposition~\ref{Prop:boundEmpty}) are the core elements we use to prove the main result in this section.
\begin{theorem}\label{prop:fullparking}
Assume that $\lambda^n=\lambda n$, $K^n=  K n $ and $M^n=\frac{\nu}{\nu+\mu}K^n+\beta\sqrt{n}$.
Further, we assume $\nu K< \lambda$. If
$\frac{Z^n(0)-\frac{\nu}{\nu+\mu}K^n}{\sqrt{n}}\overset{d}\rightarrow \hat{Z}_f(0)$, then
\begin{equation*}
\frac{Z^n(\cdot)-\frac{\nu}{\nu+\mu}K^n}{\sqrt{n}}
\overset{d}\rightarrow \hat{Z}_f(\cdot), \mbox{ as }
n\rightarrow \infty,
\end{equation*}
where the process $\hat{Z}_f(\cdot)$ is given in Proposition~\ref{prop:hugear}.
\end{theorem}
A diffusion approximation for the expected number of the original system in an overloaded regime is now given by
\begin{equation}\label{eq:proxyover}
\E{\Zf(\infty)}\approx\sqrt{K}\E{\hat{Z}_{f}(\infty)}+\frac{\nu}{\nu+\mu}K,
\end{equation}
and by using \eqref{eq:SSfull}, we have that
$\E{\hat{Z}_{f}(\infty)}=d_1\int_{-\infty}^{\beta}  \pf^-(x)dx+d_2\int_{\beta}^{\infty}  \pf^+(x)dx$.

The asymptotic regime for an overloaded system leads to a one-dimensional approximation. In the next section, motivated by \cite{ward03}, we consider an asymptotic regime where we scale the parking times, which leads to a two-dimensional diffusion approximation.

\subsubsection{Diffusion approximation for small parking rates}
\label{sec:DifsmallparkingT}
In this section, we study a diffusion approximation in case the parking rate $\nu$ is
``small''. First, we focus on the system with infinitely many parking spaces and we show a heavy
traffic limit theorem; see Section~\ref{sec:alternative} for an
alternative model description when $K=\infty$.
In this case,  the limit is a two-dimensional OU process with reflection.
Then, making an overloaded assumption (for the uncharged EVs),  we derive a two-dimensional OU
limit process and we obtain the same limit if we assume a sufficiently large number of parking spaces.

Assume that $K=\infty$. Define the traffic intensity for this model as $\rho:=\frac{\lambda}{\mu M}$.
Let $\mu, M$ be fixed. Further, define $\nu^n=\frac{1}{n}$ and $\lambda^n=\mu
M(1-\frac{c}{\sqrt{n}})$ for some constant $c$. Note, that $\frac{1-\rho^n}{\sqrt{n}} \rightarrow
c$ as $n \rightarrow \infty$, which is our heavy traffic assumption. Moreover, define the diffusion scaled process as follows
\begin{align*}
\tilde{Z}^n(t):=
\frac{Z^n(nt)}{\sqrt{n}}
\quad  \text{and} \quad
\tilde{Q}^n(t):=
\frac{ Q^n(nt)-\mu M n}{\sqrt{n}}.
\end{align*}
The next proposition states a heavy traffic result for the two-dimensional scaled process.
\begin{proposition}[Heavy traffic]\label{Prop: heavy traffic}
Assume that $(\tilde Z^n(0),\tilde Q^n(0))$ $\overset{d} \rightarrow
(\tilde Z(0),\tilde Q(0))$ as $n \rightarrow \infty$.  We have that
$(\tilde Z^n(\cdot),\tilde Q^n(\cdot)) \overset{d} \rightarrow
(\tilde Z(\cdot),\tilde Q(\cdot))$, and that
the limit satisfies the following two-dimensional stochastic differential equation
\begin{align*}
	\left(
      \begin{array}{cc}
	d\tilde Z(t) \\
	d\tilde Q(t)
	\end{array}\right)
	=
	-\left(
      \begin{array}{cc}
	c\mu M+\tilde Z(t) \\
	c\mu M+\tilde Q(t)
	\end{array}\right)
	dt
+	
\left(
    \begin{array}{cc}
       \sqrt{2\mu M} & 0 \\
      0 &  \sqrt{2\mu M}\\
    \end{array}
  \right)
\left(
      \begin{array}{cccc}
	     dW_{\tilde Z}(t)  \\
		 d W_{\tilde Q}(t)
	\end{array}\right)
+
	\left(
      \begin{array}{cc}
	d\tilde Y(t) \\
	0
	\end{array}\right), \nonumber
\end{align*}
where
$\tilde{Y}(t)$ satisfies the relation $\int_{0}^{\infty} \ind{\tilde{Z}(t)>0}d\tilde{Y}(t)=0$.
Further, $W_{\tilde Z}(\cdot)$ and
$W_{\tilde Q}(\cdot)$ are driftless, univariate Brownian motions such that
$\E{W_{\tilde Z}(t) W_{\tilde Q}(t)}=
 t/2$.
\end{proposition}

Observe that the limit process in the last proposition depends on the reflection at zero. To
overcome this difficulty, we consider  an overloaded regime for the number of uncharged EVs. In this regime,
the fraction of time that the process $\tilde Z(\cdot)$ spends at state zero is negligible; \cite{ward03}. To this end, let $\lambda, \mu, M$ be fixed with $\lambda>\mu M$ and $\nu^n=1/n$. Modifying slightly the scaled processes, i.e.,
\begin{align*}
\tilde{Z}^n_o(t):=
\frac{ Z^n(nt)-(\lambda-\mu M)n}{\sqrt{n}}
\quad  \text{and} \quad
\tilde{Q}^n_o(t):=
\frac{ Q^n(nt)-\lambda n}{\sqrt{n}},
\end{align*}
we are able to show the following proposition.
\begin{proposition}\label{prop:Difover}
Let $\lambda>\mu M$. Supposing that $(\tilde Z^n_o(0),\tilde Q^n_o(0)) \overset{d} \rightarrow
(\tilde Z_o(0),\tilde Q_o(0))$ as $n \rightarrow \infty$, then
$(\tilde Z^n_o(\cdot),\tilde Q^n_o(\cdot)) \overset{d} \rightarrow
(\tilde Z_o(\cdot),\tilde Q_o(\cdot))$.
The diffusion limit satisfies the following two-dimensional stochastic differential equation
\begin{align*}
	\left(
      \begin{array}{cc}
	d\tilde Z_o(t) \\
	d\tilde Q_o(t)
	\end{array}\right)
	=
	-\left(
      \begin{array}{cc}
	\tilde Z_o(t) \\
	\tilde Q_o(t)
	\end{array}\right)
	dt
+	\left(
      \begin{array}{cccc}
	\sqrt{\lambda}& -\sqrt{\mu M}& -\sqrt{\lambda-\mu M} & 0   \\
		\sqrt{\lambda}& 0& -\sqrt{\lambda-\mu M} & -\sqrt{\mu M}
	\end{array}\right)
 d\blt{W}(t), \nonumber
\end{align*}
where $\blt {W}(\cdot)=(W_1(\cdot),W_2(\cdot),W_3(\cdot),W_4(\cdot))^T$, with $W_i(\cdot)$ independent standard Brownian
motions.
\end{proposition}

Note that we derive the same limit if we assume that
$K<\infty$ and we scale the number of parking spaces in the $n^{th}$ system $K^n$ such that
$\frac{K^n-\lambda n}{\sqrt{n}}\rightarrow \infty $,
as $n \rightarrow \infty$.  In this case,
the fraction of time that the scaled process $\tilde Q_o^n(\cdot)$ spends on the boundary is negligible. This is made rigorous in the following lemma.

\begin{lemma}\label{lem:prel}
For $T>0$, we have that for any
$\epsilon>0$ there exists $n_{\epsilon}$  such that
\begin{equation*}
\Prob{ \sup_{t\leq T} \tilde{Q}^n_o(t)< \frac{K^n-\lambda n}{\sqrt{n}}}> 1-\epsilon,
\end{equation*}
for all $n>n_{\epsilon} $.
\end{lemma}
\begin{remark}
 The sequence $\{\tilde{Q}^n_o(t), t\geq 0\}$
is stochastically bounded as it converges in distribution to $(\mathcal{D}[0,\infty),J_1)$ which
is a complete and separate metric space; \cite[Corollary~3.1]{whitt2007}. Then
Lemma~\ref{lem:prel} follows and holds true for any deviating sequence $R^n$ instead of
$\frac{K^n-\lambda n}{\sqrt{n}}$. Last, note that we only need the weak convergence of the
process $\tilde{Q}^n_o(\cdot)$ and the fact that the quantity $\frac{K^n-\lambda n}{\sqrt{n}}$ goes
to infinity. For the last convergence, it is enough to choose $K^n> \lambda n$.
\end{remark}

The joint steady-state distribution of
$(\tilde Z_o(\cdot),\tilde Q_o(\cdot))$, say $\pi_o(\cdot,\cdot)$, is given by a bivariate Normal distribution with mean
$\blt{\mu}=\left(0, 0\right)$ and covariance matrix
$\blt{\Sigma}=\left(
          \begin{array}{cc}
            \frac{\lambda}{\nu} & \frac{2\lambda-\mu M}{2\nu} \\
            \frac{2\lambda-\mu M}{2\nu} & \frac{\lambda}{\nu} \\
          \end{array}
        \right)
$.
Note, that $\blt{\Sigma}$ is indeed  a variance matrix as it is positive definite. To see
this, observe that
\begin{equation*}
\begin{split}
\det({\blt{\Sigma}})=\frac{1}{4\nu^2}(4\lambda^2-4\lambda^2-\mu^2 M^2
+4\lambda \mu M)>\frac{1}{4\nu^2}(-\mu^2 M^2+4\mu^2 M^2)>0,
\end{split}
\end{equation*}
where the first inequality holds by the assumption $\lambda>\mu M$.

Now, for parameters of the original system such that $\lambda>\mu M$, sufficiently ``small'' $\nu$,
and $K>\lambda/\nu$,
we suggest the following diffusion approximation
\begin{align*}
\E{Z(\infty)}& \approx \frac{\E{\tilde{Z}^n_o(\infty)}}
{\sqrt{\nu}}+
\frac{(\lambda-\mu M)}{\nu},\\
\E{Q(\infty)}&\approx
\frac{\E{\tilde{Q}^n_o(\infty)}}{\sqrt{\nu}}+
\frac{\lambda}{\nu }.
\end{align*}

\section{Numerical evaluation}\label{sec:numerics}
In this section, we
validate numerically the previous bounds and approximations for three cases: the moderately ($\lambda<\nu K$), critically ($\lambda=\nu K$), and
overloaded ($\lambda>\nu K$) systems.
We focus on the expected number of uncharged EVs in the system and the
probability that an EV leaves the charging station with
fully charged battery (success probability). In all the numerical examples, we solve the flow balance equations \eqref{eq:balanceEquations} using standard numerical methods and we let $\nu=$ $\mu=1$. Last, the relative error is calculated by the following formula,
$\mbox{RE}=\frac{|\E{Z(\infty)}-\E{Z^{ap}(\infty)}|}{\E{Z(\infty)}}100\%$, where
$\E{Z(\infty)}$ denotes the expected number of EVs in the original system by solving the
two-dimensional Markov process and $\E{Z^{ap}(\infty)}$ denotes the expected number of uncharged EVs for the
aforementioned approximations.

First, we evaluate the fluid approximation.
Table~\ref{table:FO} gives the relative error between the expected number of uncharged EVs for the original system and the fluid approximation given in \eqref{eq:fluid proxy1} for different values of the number of parking spaces $K$. For a given $K$, we give only the \emph{maximum} relative error for  $0<M\leq K$. As expected, the relative error decreases as $\lambda$ and $K$ increase. In table 2, we present the relative error between the expected number of uncharged EVs for the original system and the modified fluid approximation
given in equation \eqref{eq:fluid proxy}. Not surprisingly, the relative error is much smaller in this case,  as we can see in Table~\ref{table:FM}. For high values of $\lambda, K$ the relative error is approximately 2\% rather than 10--20\%. In addition, the modified fluid approximation seems to be reasonable also in the moderate regime.
\begin{table}[!h]
\begin{center}
\caption{Evaluation of the original fluid approximation}
\label{table:FO}
\begin{tabular}{|c||c|c|c|c|c|}
  \hline
                 & $K=10$ & $K=20$ & $K=30$ & $K=40$ & $K=50$ \\ \hline
  $\lambda=K$    & 39.6569 \% & 28.5579\% & 23.8308\% & 21.2686\% & 19.3942\% \\
  $\lambda=1.2K$ & 27.9191 \% & 18.0935\% & 14.3587\% & 12.0822 \% & 10.4540\% \\
  \hline
\end{tabular}\end{center}
\end{table}
\begin{table}[!h]\label{table:FM}
\begin{center}
\caption{Evaluation of the modified fluid approximation }
\label{table:FM}
\begin{tabular}{|c||c|c|c|c|c|}
  \hline
                 & $K=10$ & $K=20$ & $K=30$ & $K=40$ & $K=50$ \\ \hline
  $\lambda=0.8K$ & 8.8421\% & 7.2831\% & 6.5797\% & 6.1286\% & 5.7892\% \\
  $\lambda=K$    & 11.0904\% & 6.0576\% & 3.5069\% & 2.2082\% & 1.7918\% \\
  $\lambda=1.2K$ & 8.7045\% & 3.7961\% & 3.1936\% & 2.8380\% &  2.5959\% \\
  \hline
\end{tabular}\end{center}
\end{table}

Next, we evaluate the approximation in case of a full parking lot
(see Section~\eqref{sec:full parking lot}) and the diffusion approximation in an overloaded regime given by \eqref{eq:proxyover}. To improve the approximations,
we directly modify them by replacing the parameter $K$ by the expected number of the total EVs in the
original system, i.e., $\lambda(1-B(\lambda/\nu,K))$.
Table~\ref{table:M1BD} gives the relative error for $\E{\Zf(\infty)}$. As we expect, it decreases as $\lambda$ and $K$ increase. Furthermore, this approximation results to small relative errors in all regimes $(<3\%)$. The (prelimit) approximation
 $\E{\Zf(\infty)}$ is better than the modified diffusion approximation in \eqref{eq:proxyover}, as we see in Table~\ref{table:DM1BD}.
\begin{table}[!h]\begin{center}
\caption{Evaluation of the modified $\E{\Zf(\infty)}$}
\label{table:M1BD}
\begin{tabular}{|c||c|c|c|c|c|}
  \hline
                 & $K=10$ & $K=20$ & $K=30$ & $K=40$ & $K=50$ \\ \hline
  $\lambda=0.8K$ & 3.0083\% & 2.2710\% &  1.9940\% & 1.8354\% & 1.7248\% \\
  $\lambda=K$    & 1.9747 \% & 1.2064\% & 0.8632\% & 0.6492\% & 0.5425\% \\
  $\lambda=1.2K$ & 1.2803\% & 0.6708\% & 0.4649\% & 0.3557\% & 0.2873\% \\
  \hline
\end{tabular}\end{center}
\end{table}
\begin{table}[!h]\begin{center}
\caption{Evaluation of the modified diffusion approximation in an overloaded regime }
\label{table:DM1BD}
\begin{tabular}{|c||c|c|c|c|c|}
  \hline
                 & $K=10$ & $K=20$ & $K=30$ & $K=40$ & $K=50$ \\ \hline
  $\lambda=0.8K$ & 12.1493\% & 9.1953\% & 7.8522\% & 7.0228\% & 6.4357\% \\
  $\lambda=K$    & 11.7346\% & 8.1938\% & 6.2103\% & 4.7916\% & 3.7020\% \\
  $\lambda=1.2K$ & 10.7606\% & 6.7321\% & 5.1773\% & 4.5167\% & 4.0661\% \\
  \hline
\end{tabular}\end{center}
\end{table}

In the sequel, we depict the bounds in \eqref{In:bounds}, the modified bound in \eqref{In:lower2} -- where we replace $K$ by $\lambda(1-B(\lambda/\nu,K))$ -- (dotted line), the modified fluid approximation \eqref{eq:fluid} (dashed line), and the modified diffusion approximation in \eqref{eq:proxyover} (dash-dot line). In Figures \ref{fig:1}--\ref{fig:last}, the vertical axes
give the probability that an EV leaves the parking lot with
fully charged battery (success probability) and the horizontal axes give the ratio $M/K$. For each regime, we plot the success probability for $K=10,20,30,50$. In the moderated regime, the lower bound ($K=\infty$) is very close for high values of parking spaces. This is not surprising because the time that the process spends on the boundary is negligible in this case.
 The fluid approximation seems to be quite good in  most of the cases and the diffusion approximation does not improve the fluid one. Last,  note that the modified bound \eqref{In:lower2} does not give a lower bound of the original system. However, it seems to be the best approximation
 for all the cases; even in the moderated regime.
\begin{figure}[!h]
\centering
\begin{minipage}{.5\textwidth}
\includegraphics[scale=0.5]{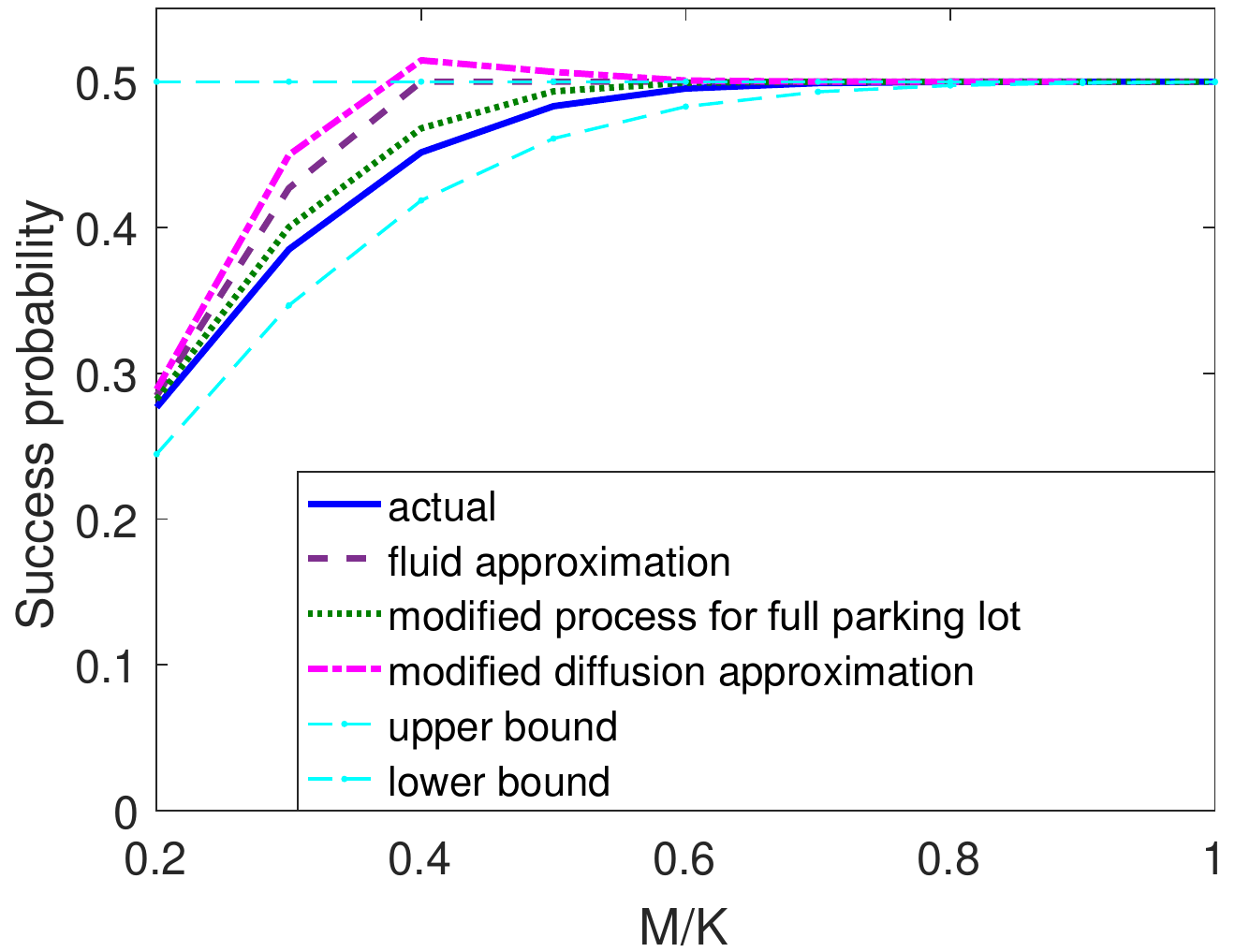}
\caption{$K=10$ and $\lambda=0.8K$.}
\label{fig:1}
\end{minipage}%
\begin{minipage}{.5\textwidth}
  \centering
  \includegraphics[scale=0.5]{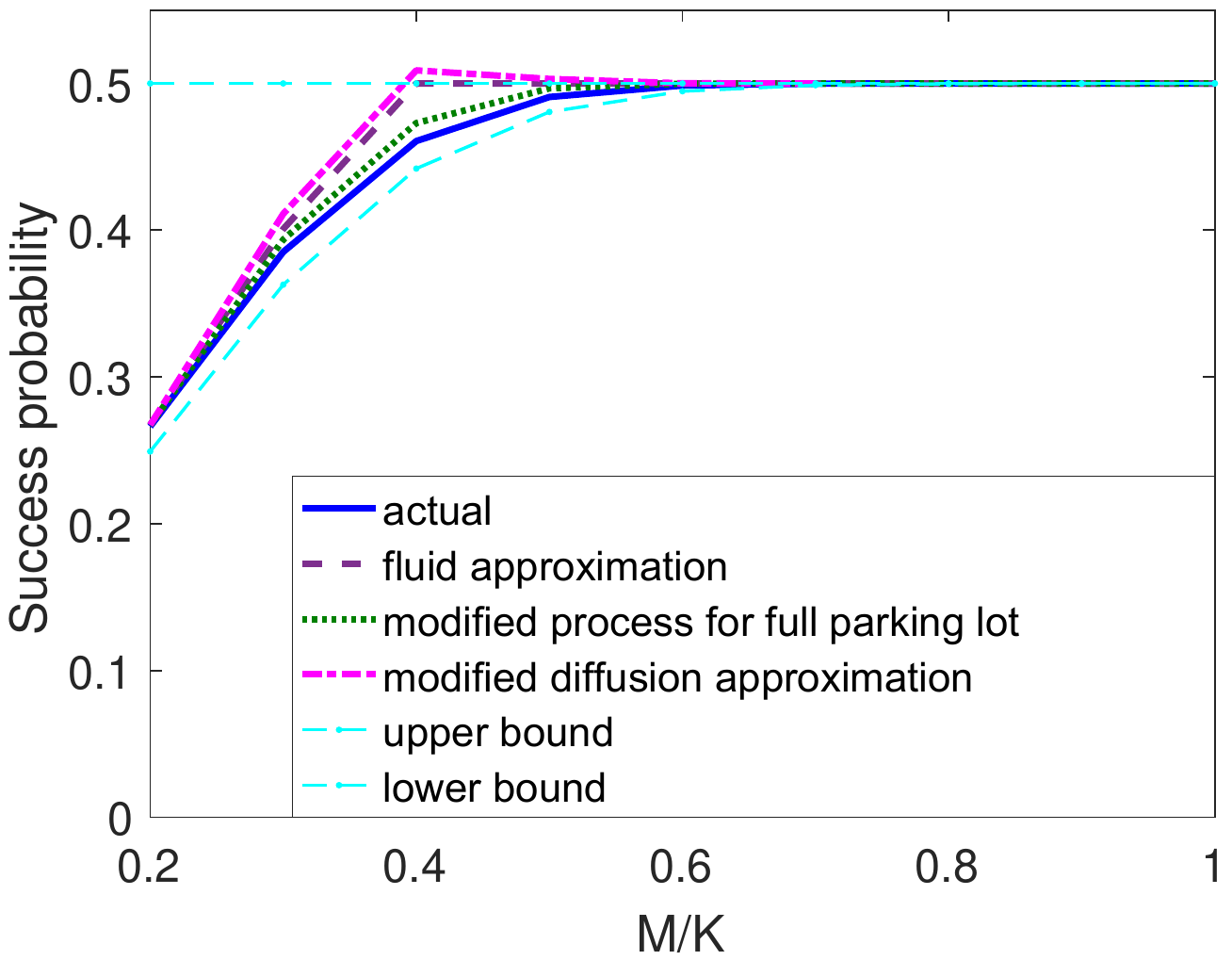}
\caption{$K=20$ and $\lambda=0.8K$.}
\end{minipage}
\end{figure}
\begin{figure}[!h]
\centering
\begin{minipage}{.5\textwidth}
\includegraphics[scale=0.5]{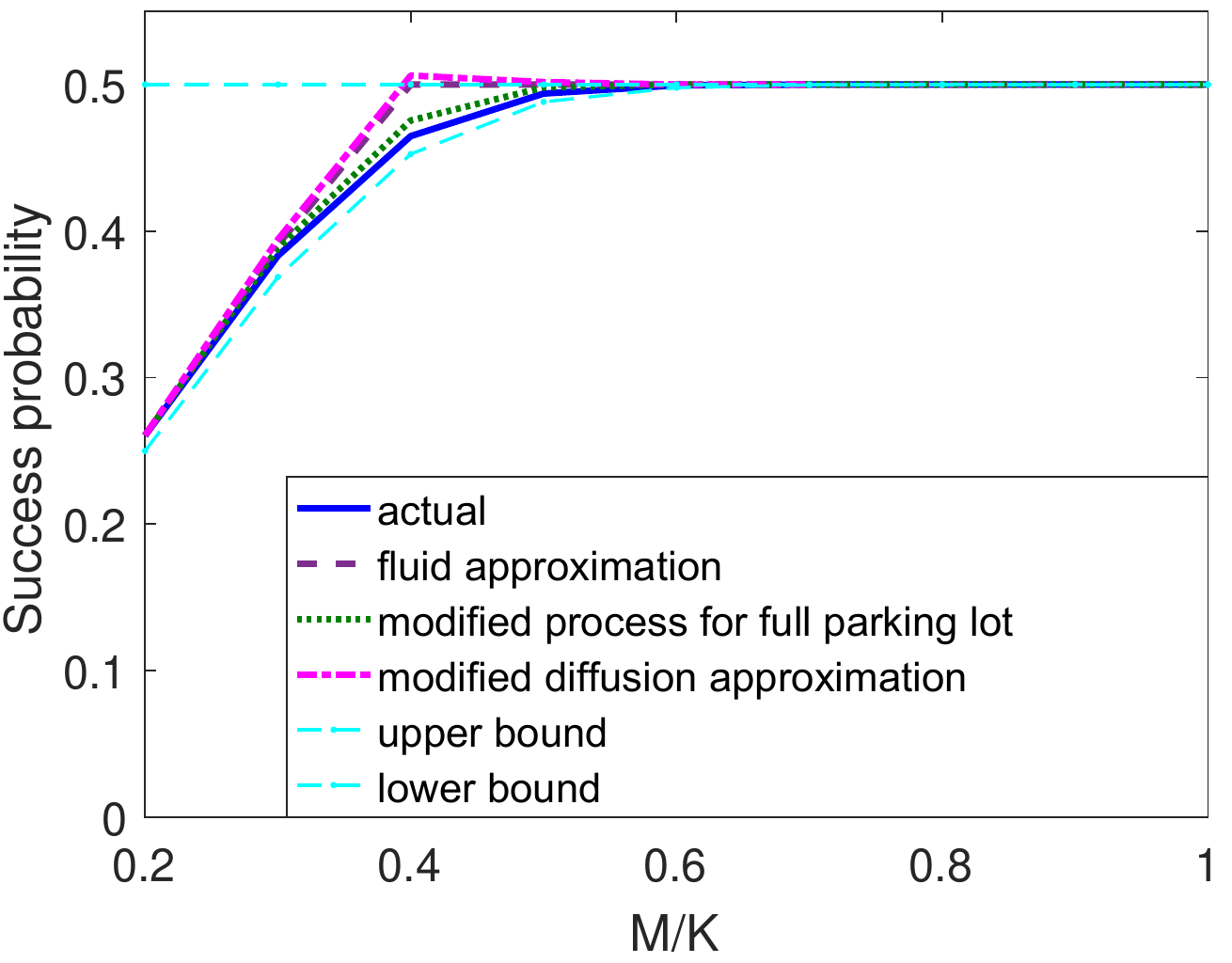}
\caption{$K=30$ and $\lambda=0.8K$.}
\end{minipage}%
\begin{minipage}{.5\textwidth}
  \centering
  \includegraphics[scale=0.5]{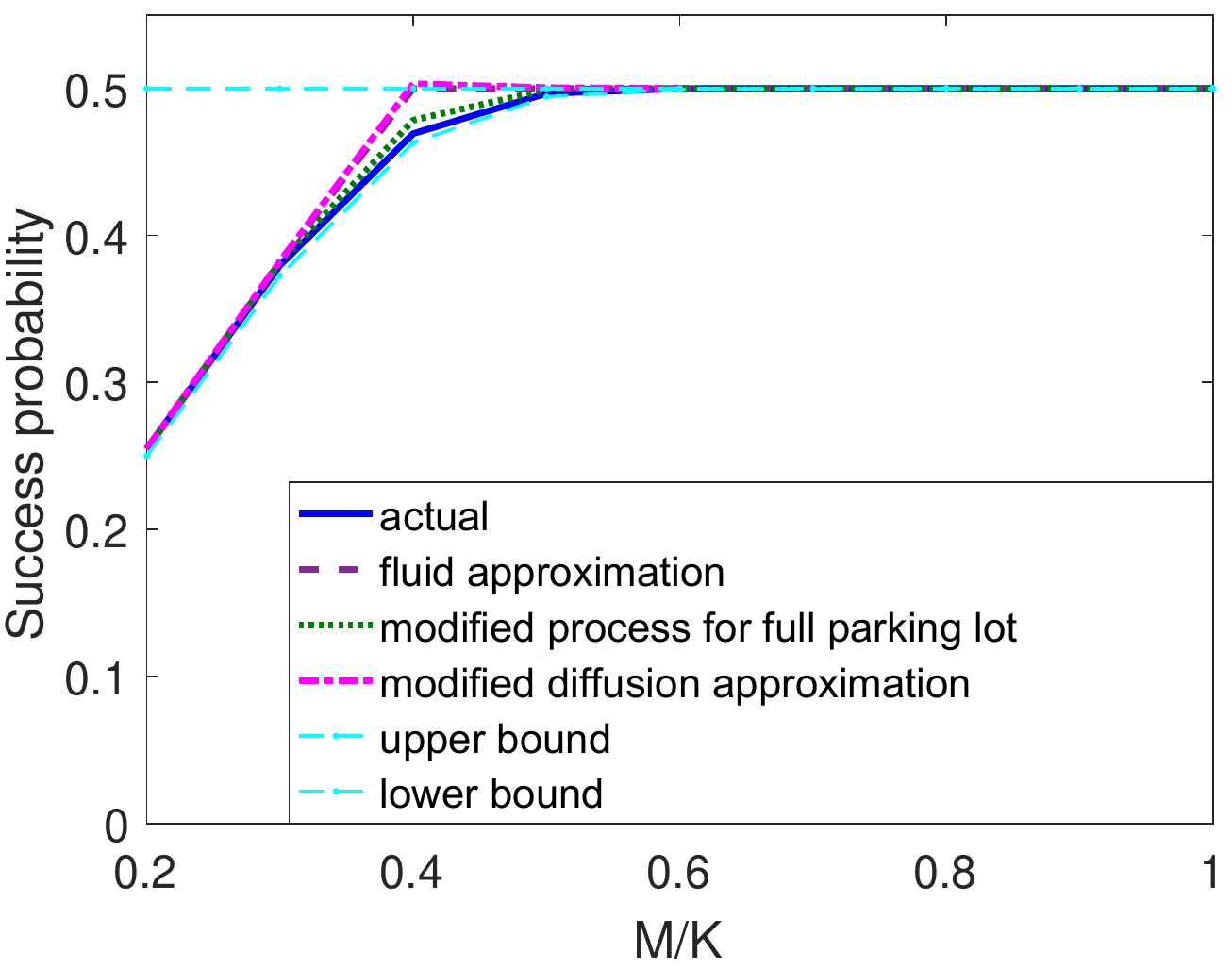}
\caption{$K=50$ and $\lambda=0.8K$.}
\end{minipage}
\end{figure}
\begin{figure}[!h]
\centering
\begin{minipage}{.5\textwidth}
\includegraphics[scale=0.5]{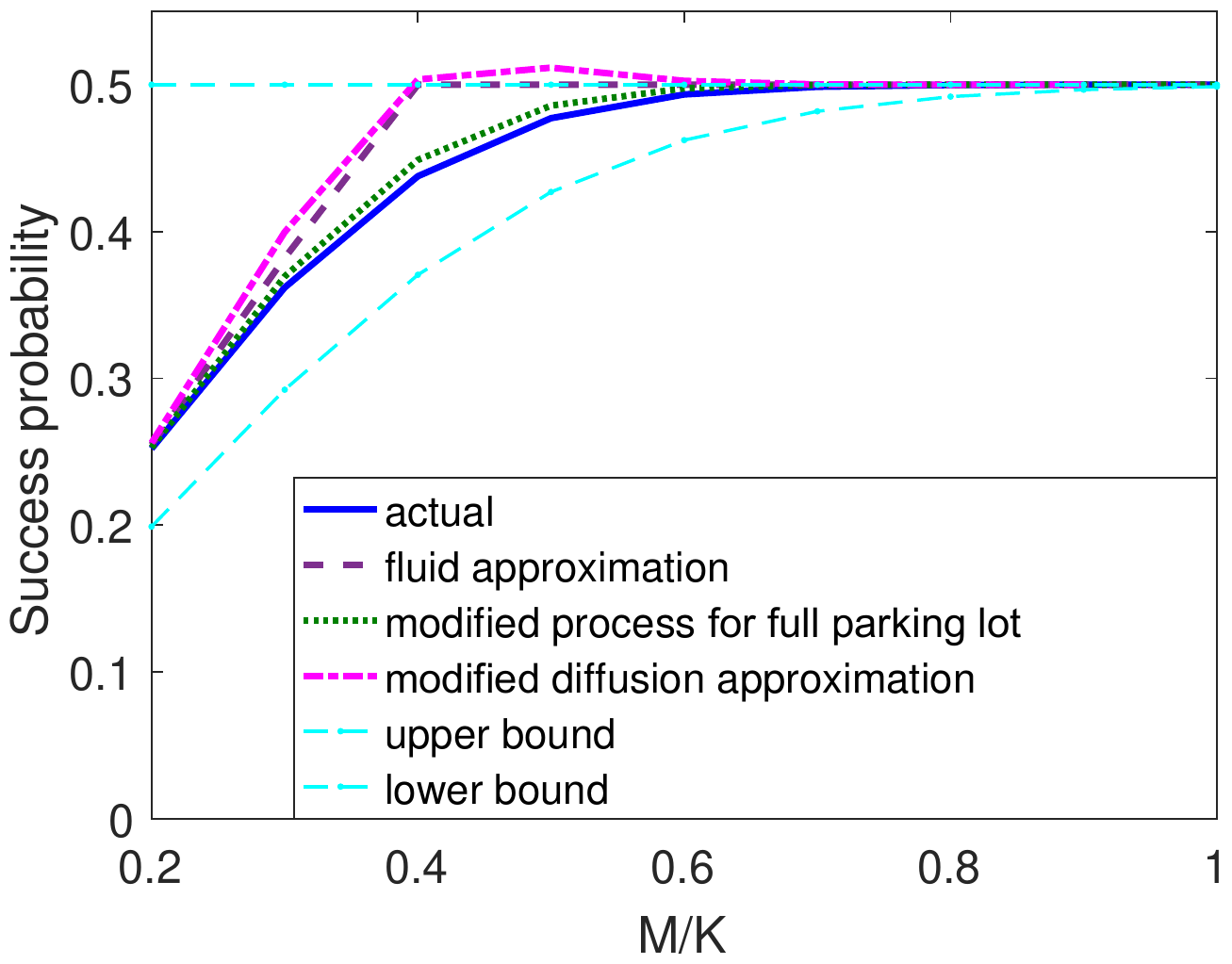}
\caption{$K=10$ and $\lambda=K$.}
\end{minipage}%
\begin{minipage}{.5\textwidth}
  \centering
  \includegraphics[scale=0.5]{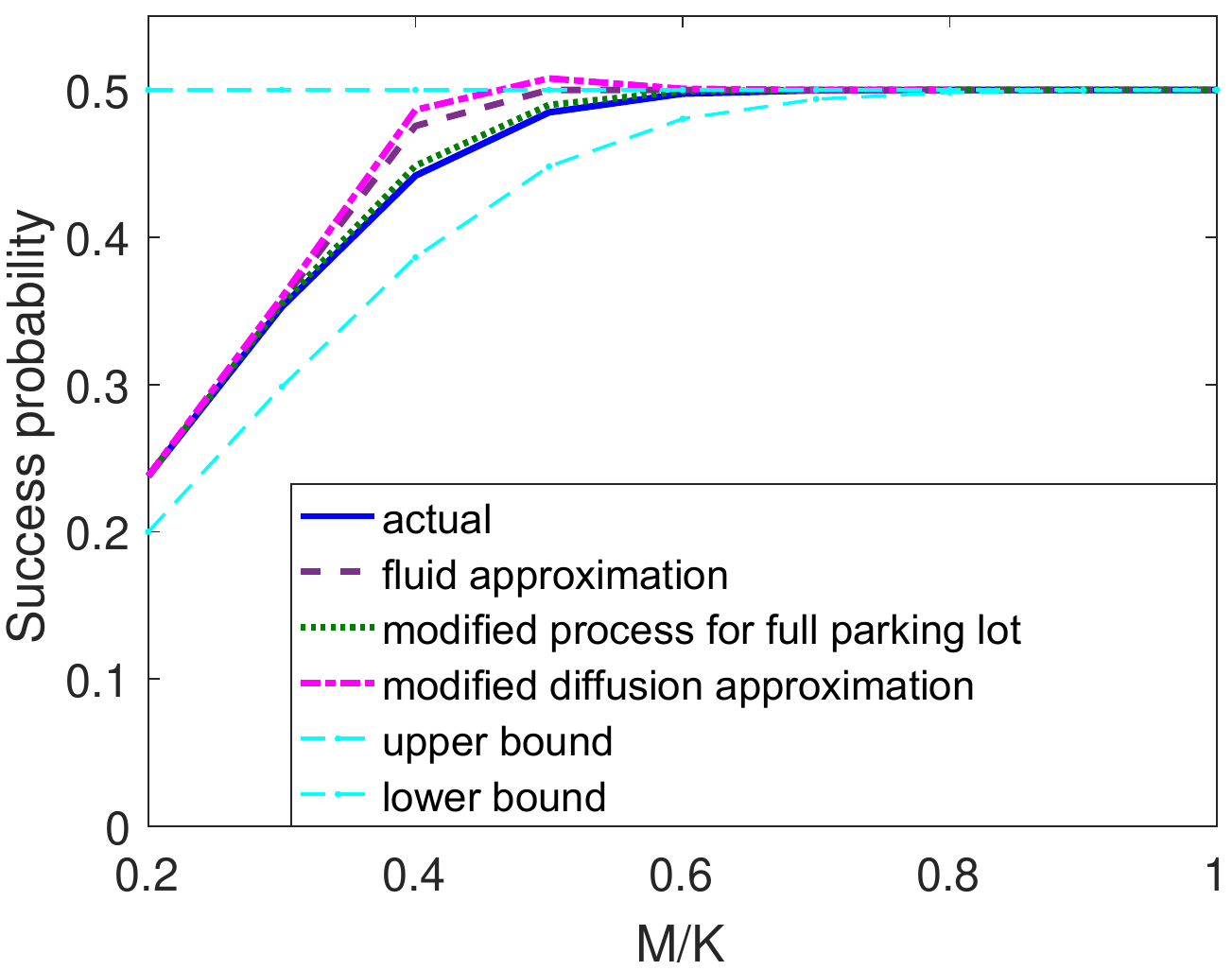}
\caption{$K=20$ and $\lambda=K$.}
\end{minipage}
\end{figure}
\begin{figure}[!h]
\centering
\begin{minipage}{.5\textwidth}
\includegraphics[scale=0.5]{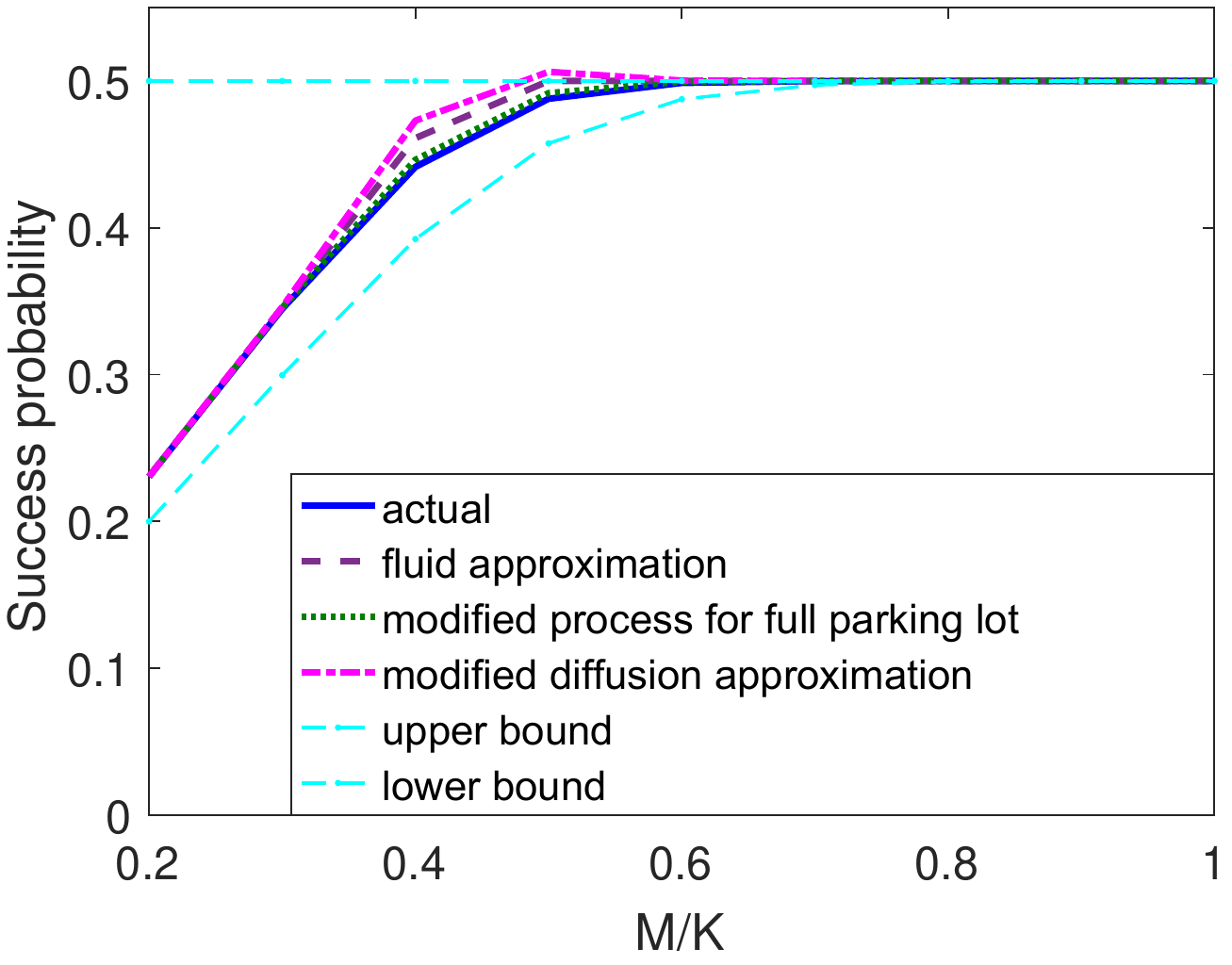}
\caption{$K=30$ and $\lambda=K$.}
\end{minipage}%
\begin{minipage}{.5\textwidth}
  \centering
  \includegraphics[scale=0.5]{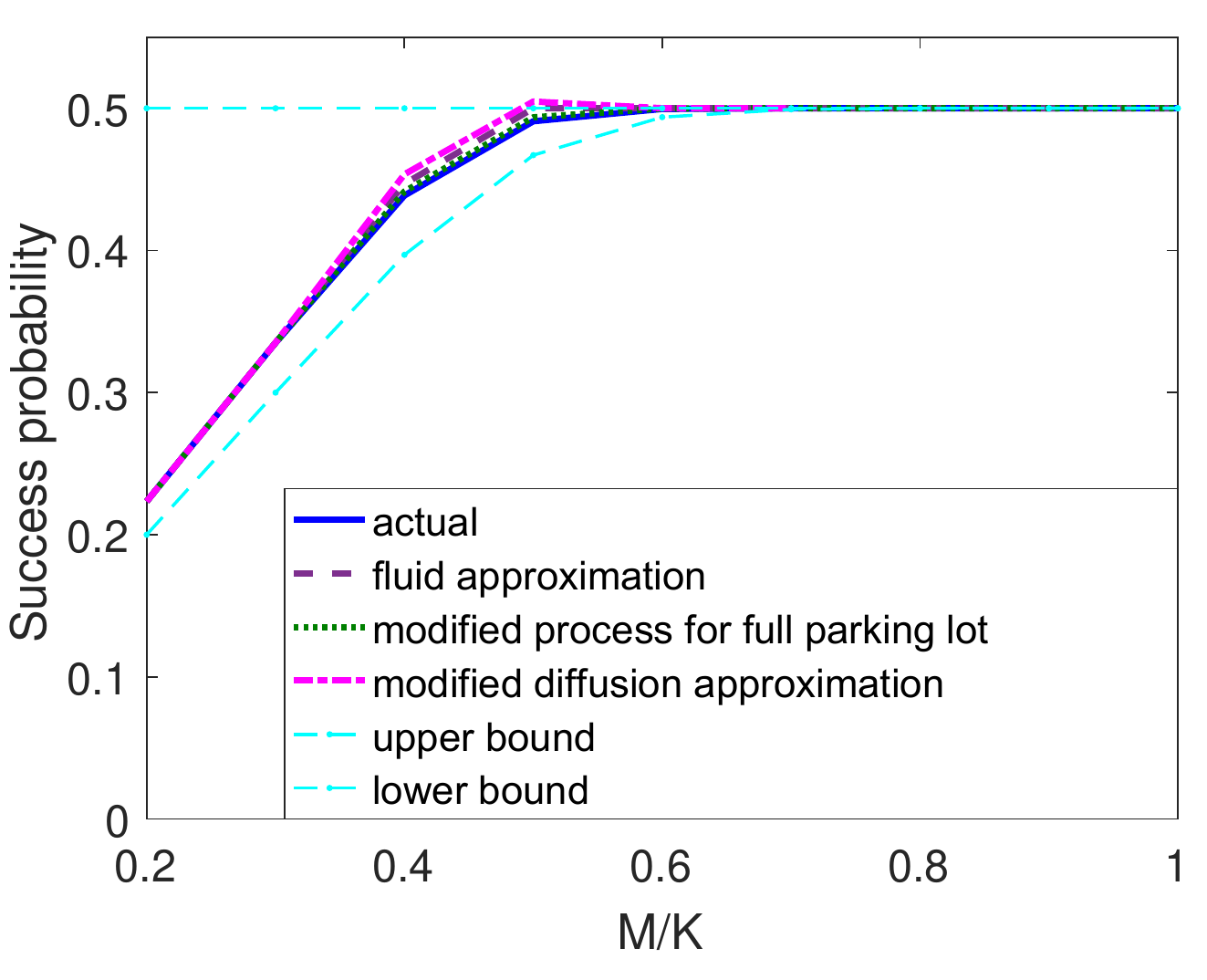}
\caption{$K=50$ and $\lambda=K$.}
\end{minipage}
\end{figure}
\begin{figure}[!h]
\centering
\begin{minipage}{.5\textwidth}
\includegraphics[scale=0.5]{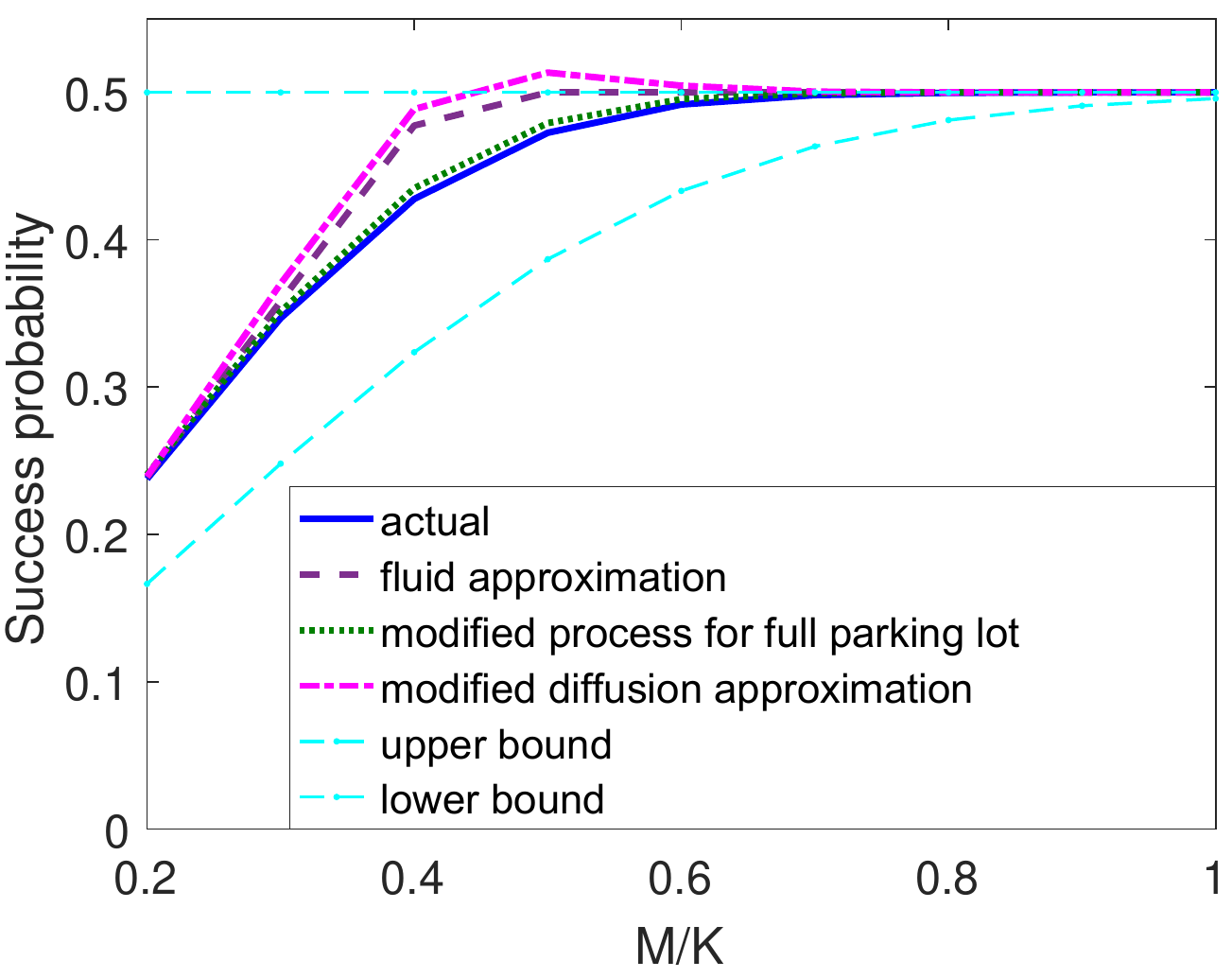}
\caption{$K=10$ and $\lambda=1.2K$.}
\end{minipage}%
\begin{minipage}{.5\textwidth}
  \centering
  \includegraphics[scale=0.5]{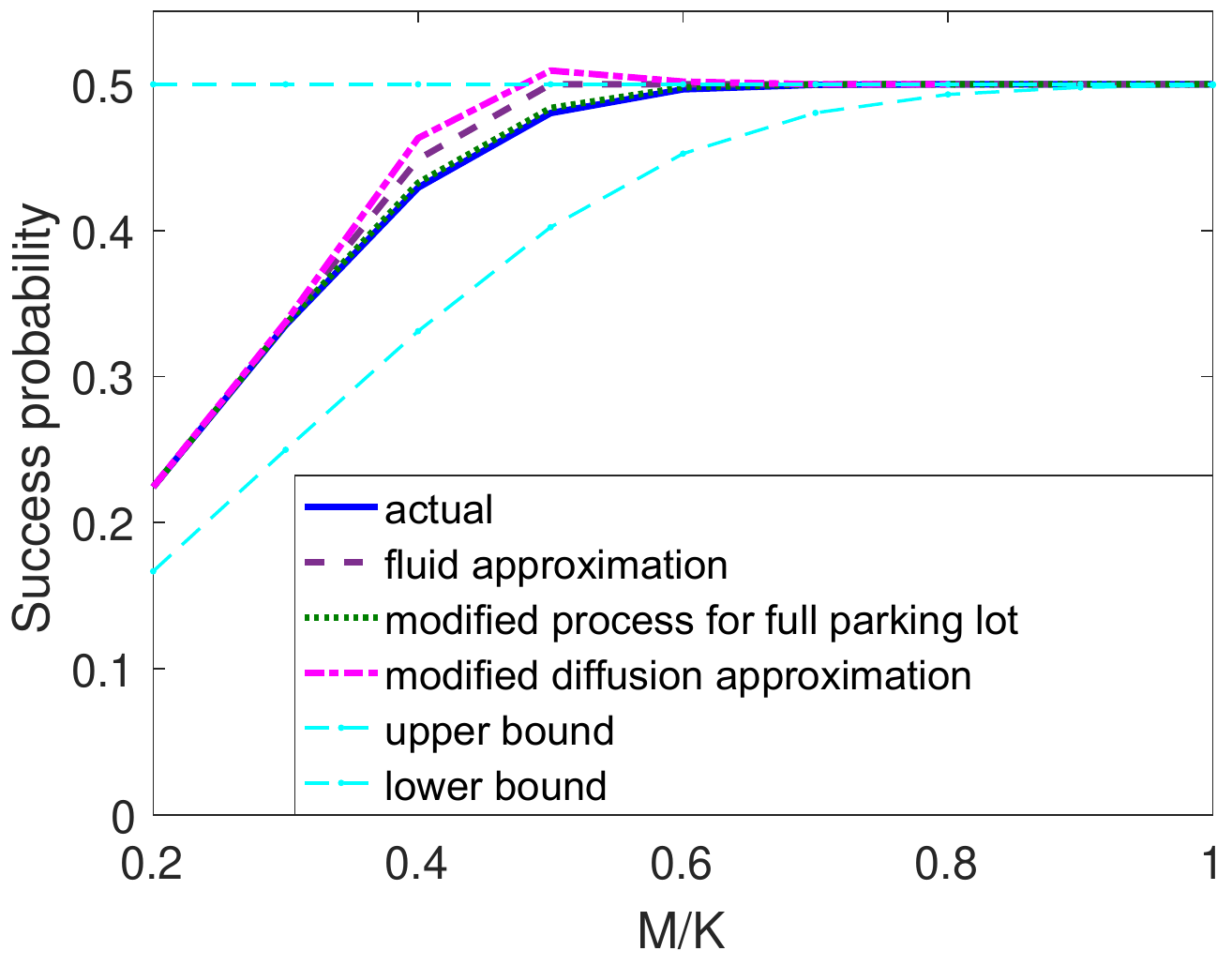}
\caption{$K=20$ and $\lambda=1.2K$.}
\end{minipage}
\end{figure}
\begin{figure}[!h]
\centering
\begin{minipage}{.5\textwidth}
\includegraphics[scale=0.5]{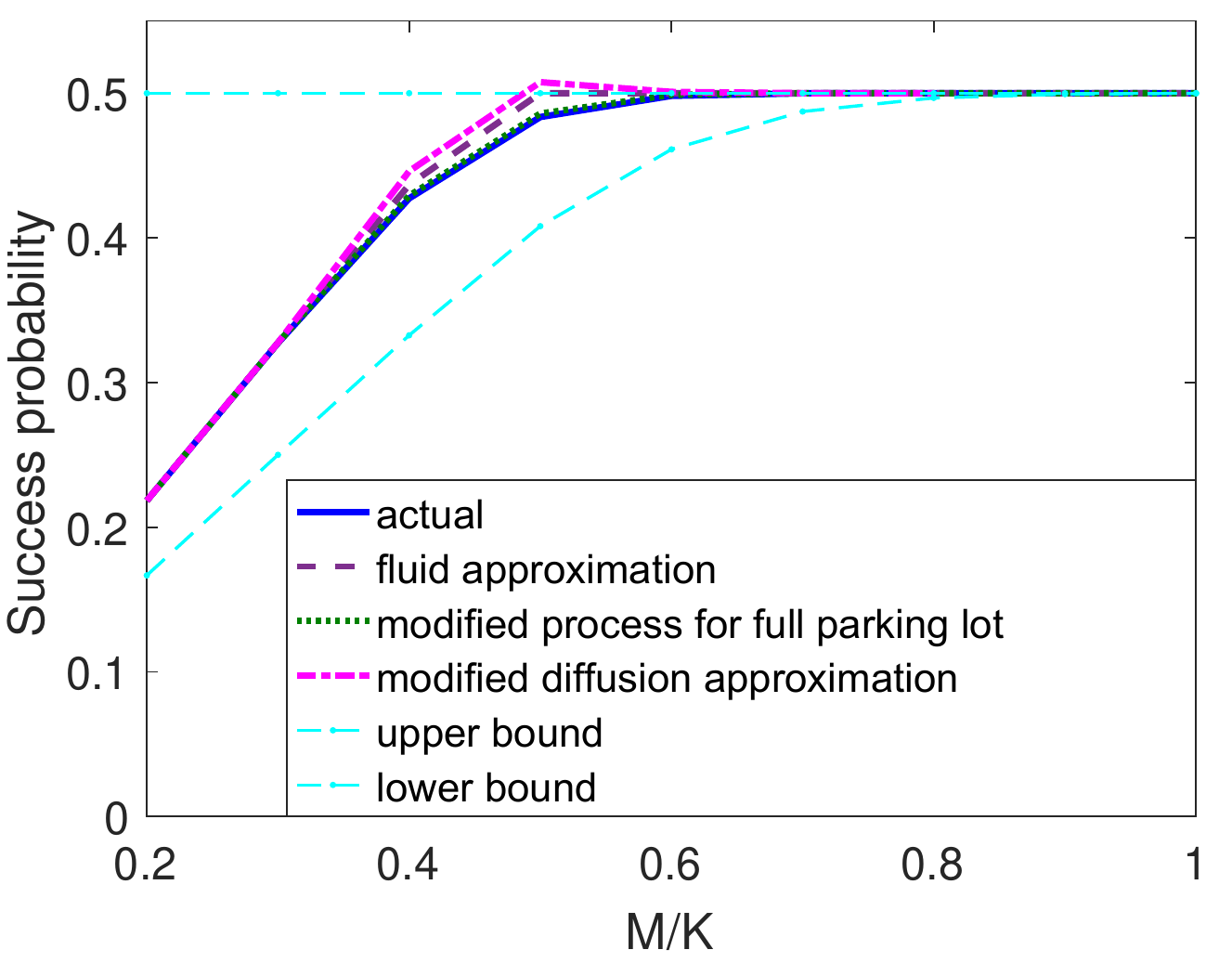}
\caption{$K=30$ and $\lambda=1.2K$.}
\end{minipage}%
\begin{minipage}{.5\textwidth}
  \centering
  \includegraphics[scale=0.5]{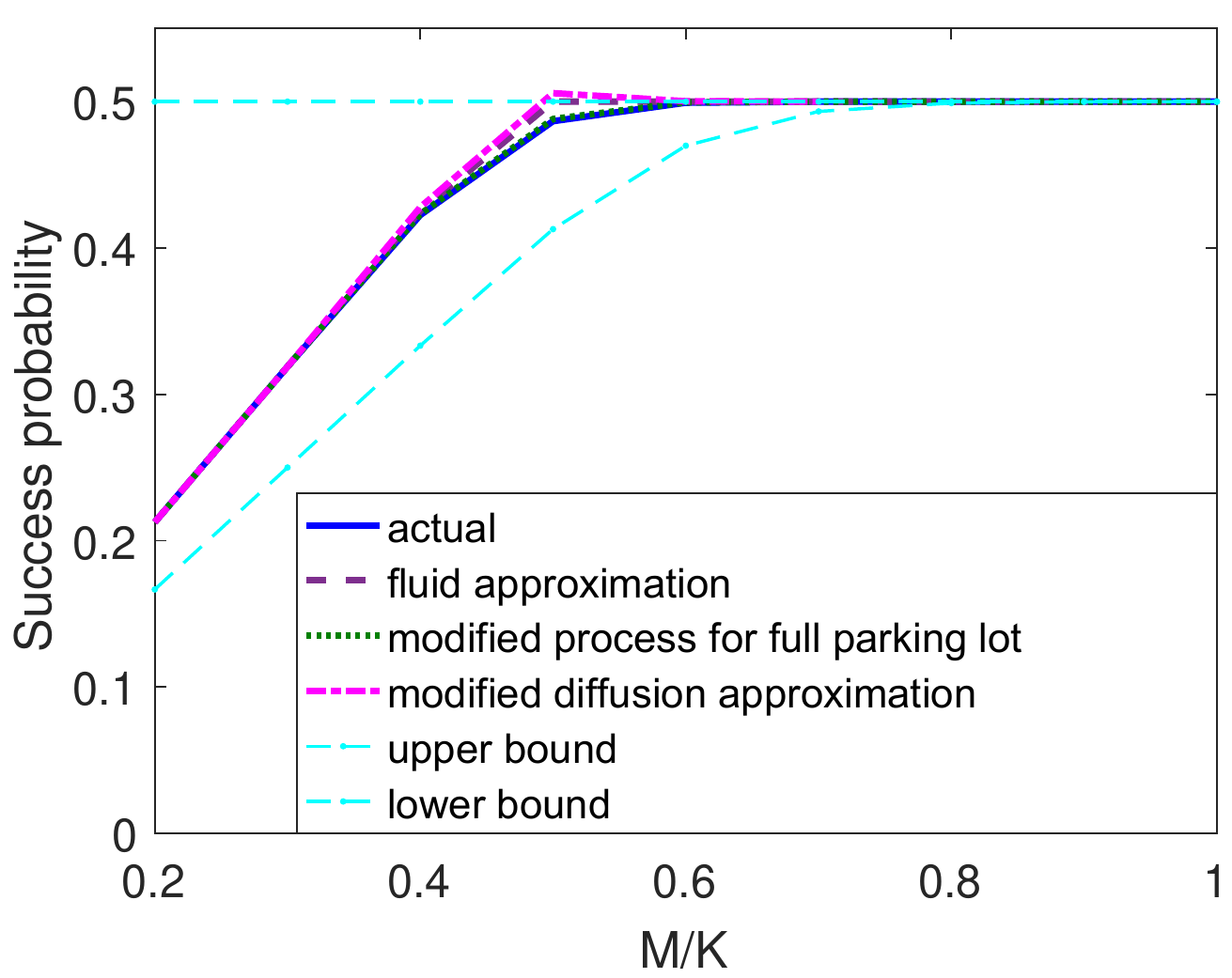}
\caption{$K=50$ and $\lambda=1.2K$.}
\label{fig:last}
\end{minipage}
\end{figure}

\section{Proofs }\label{proofs}
\subsection{Proofs for Section~\ref{sec:explisit results}}
\begin{proof}[Proof of Proposition~\ref{prop:bounds}]
First, we show the upper bound. Note that the probability (in stationarity) that an EV leaves with a full battery can also be given by
\begin{equation*}
P_s=\frac{\E{C^K_M(\infty)}}{\E{Q^K_M(\infty)}}=
\Prob{D > B \max \Big\{ 1, \frac{Z^K_M(\infty)}{M}\Big\}}.
\end{equation*}
Now, using the assumptions that $Z^K_M(\infty)\leq K$ and $M\leq K$, the following inequality holds
 almost surely
\begin{equation*}
B\max \Big\{ 1, \frac{Z^K_M(\infty)}{M}\Big\}\geq
B \max \Big\{ 1, \frac{Z^K_K(\infty)}{K}\Big\} = B.
\end{equation*}
By the last inequality, we have that
\begin{equation*}
\Prob{D > B \max \Big\{ 1, \frac{Z^K_M(\infty)}{M}\Big\}}\leq
\Prob{D > B },
\end{equation*}
which proves the upper bound. Note that the upper bound is nothing else but the minimum of two exponential random variables.

We move now to the proof of the lower bound in \eqref{In:bounds}. We note that it is equivalent to  the following inequality
\begin{equation}\label{in:ELB}
  \Prob{D > B \max \Big\{ 1, \frac{Z^{\infty}_M(\infty)}{M}\Big\}}\leq
\Prob{D > B \max \Big\{ 1, \frac{Z^K_M(\infty)}{M}\Big\}}.
\end{equation}
First, we show that $Z^{\infty}_M(\infty) \geq_{st} Z^{K}_M(\infty)$  by using coupling arguments. Then, \eqref{in:ELB} follows. Fix a sample path $\omega\in \Omega$.  Assume that $Z^{\infty}_M(0) = Z^{K}_M(0)$ and take  identical arrival, charging, and parking times for both systems. Define
$T^*=\inf\{t>0:Q^{K}_M(t)=K \}$. It follows that $Z^{\infty}_M(t) = Z^{K}_M(t)$, for $t\leq T^*$. After time $T^*$, the blocked arrivals in the loss queue will enter in the queue with infinite many parking spaces. That is, $Z^{\infty}_M(t)\geq Z^{K}_M(t) $ for all $t\geq 0$. Removing now the conditioning on the sample path $\omega$, we derive $Z^{\infty}_M(t)\geq_{st} Z^{K}_M(t) $ for all $t\geq 0$, and by the existence of the stationary distribution we have that
$Z^{\infty}_M(\infty)\geq_{st} Z^{K}_M(\infty) $.

It remains to show \eqref{In:lower2}. Let $(Q_{\lambda}(\cdot),Z_{\lambda}(\cdot))$ denote the total number of EVs and the number of uncharged EVs if the arrival rate is $\lambda$. First, using coupling arguments, we prove that
if $\lambda_1\leq \lambda_2$ then
$Q_{\lambda_1}(t) \leq_{st} Q_{\lambda_2}(t)$ and
$Z_{\lambda_1}(t) \leq_{st} Z_{\lambda_2}(t)$ for any $t\geq 0$. Assume the following coupling: if an arrival occurs to the system with arrival rate $\lambda_1$, it also occurs to system with arrival rate $\lambda_2$. Hence, as $\lambda_1\leq \lambda_2$, there are more arrivals in the second system. Further, we assume that all other parameters, i.e. $\mu$, $\nu$, $M$, $K$, are equal in both systems. Assume that both systems start empty and define
$T^{**}=\inf\{t>0:Q_{\lambda_2}(t)=K \}$.
As in the second system there more arrivals, we have that
$Q_{\lambda_1}(t) \leq Q_{\lambda_2}(t)$ and
$Z_{\lambda_1}(t) \leq Z_{\lambda_2}(t)$, for $t\leq T^{**}$. By the Markovian assumptions, we have that the residual charging and parking times are exponential with rare $\mu$ and $\nu$. That is, at any new event after time $T^{**}$, we can resample the charging and parking times and hence the probability of a departure in the system with arrival rate $\lambda_1$ is higher or equal to the probability of a departure in the system with arrival rate $\lambda_2$. In other words, for $t\geq 0$ and for $x>0$,
$\Prob{Q_{\lambda_2}(t) \leq x}\leq \Prob{Q_{\lambda_1}(t) \leq x}$ and
$\Prob{Z_{\lambda_2}(t) \leq x}\leq \Prob{Z_{\lambda_1}(t) \leq x}$. The last relation is equivalent to
$Q_{\lambda_1}(t) \leq_{st} Q_{\lambda_2}(t)$ and
$Z_{\lambda_1}(t) \leq_{st} Z_{\lambda_2}(t)$, for $t\geq 0$.
In the sequel, we see that $Z_{f}(\cdot)$ can be arise as the limit of $Z_{\lambda}(\cdot)$ as $\lambda\rightarrow \infty$, assuming that $Q_{\lambda}(0)\overset{d}\rightarrow K$ and $Z_{\lambda}(0)\overset{d}\rightarrow  Z_{f}(0)$. To see this, first observe that
$Q_{\lambda}(\cdot)\overset{d}\rightarrow K$ as $\lambda \rightarrow \infty$. Now, combining \eqref{eq:numberQ} and \eqref{eq:numberU}, we have that
\begin{align*}
Z_{\lambda}(t)&=Z_{\lambda}(0)-Q_{\lambda}(0)+Q_{\lambda}(t)
+ N_{\nu,2}\left(\int_{0}^{t}
\left(Q_{\lambda}(s)-Z_{\lambda}(s)\right) ds\right)
-N_{\mu}\left(\int_{0}^{t}Z_{\lambda}(s)\Min M ds\right).
\end{align*}
Taking $\lambda\rightarrow \infty$ and using the continuous mapping theorem, we have that
\begin{align*}
Z_{\lambda}(\cdot)\overset{d}\rightarrow Z_{\infty}(\cdot),
\end{align*}
where
\begin{align*}
Z_{\infty}(t)=Z_{f}(0)+ N_{\nu,2}\left(\int_{0}^{t}
\left(K-Z_{\infty}(s)\right) ds\right)
-N_{\mu}\left(\int_{0}^{t}Z_{\infty}(s)\Min M ds\right)\overset{d}=Z_{f}(t)
\end{align*}
and where the last equality follows by \eqref{eq:newZ}.
Furthermore, $Z_{\lambda}(\cdot)$ is non-decreasing. That is,
$Z_{\lambda}(t)\leq_{st} Z_{f}(t)$ for any $t\geq 0$ and by the existence of the stationary distributions we obtain $Z_{\lambda}(\infty)\leq_{st} Z_{f}(\infty)$. By the last inequality, it follows $\E {Z_{\lambda}(\infty)}\leq \E{Z_{f}(\infty)}$ and hence \eqref{In:lower2}.
\end{proof}

\begin{proof}[Proof of Proposition~\ref{Pr: M=K}]
Note that the distribution $p_Q(q)$ corresponds to the stationary distribution of a
one-dimensional Erlang loss system. Furthermore, by \cite[Section~1.3]{kelly2014stochastic} we
know that
\begin{equation*}
p_Q(q)=\frac{1}{q!}\left(\frac{\lambda}{\nu}\right)^qp_Q(0), \mbox{ where }
p_Q(0)= \left(\sum_{i=0}^{K}\frac{1}{i!} \big(\frac{\lambda}{\nu}\big)^i\right)^{-1}.
\end{equation*}
Thus, the probability of an empty system is $\pe(0,0)=p_Q(0)$.

As it is well known that a solution of the balance equations of a Markon process is unique, we
shall show that $\pe(q,z)$, for $z\leq q$ satisfies the flow balance equations
\eqref{eq:balanceEquations}. Then, the proof of the proposition is completed.
First, we note the
relations between $\pe(q+a,z+b)$ and $\pe(q,z)$ for $a,b \in \{-1,0,1\}$. By \eqref{eq:Erlang
loss formula}, we obtain that
\begin{equation}\label{eq:relation in Erlang system}
\begin{split}
\pe(q)= & \frac{1}{q(q-1)!}\frac{\lambda}{\nu}(\frac{\lambda}{\nu})^{q-1}p_Q(0)
 =\frac{1}{q}\frac{\lambda}{\nu}\pe(q-1).
\end{split}
\end{equation}
Now, applying the previous relation in \eqref{eq:stationary dis} we have that
\begin{equation*}
  \begin{split}
      \pe(q-1,z-1)=&
       q \frac{\nu}{\lambda}p_Q(q) \frac{(q-1)!}{(z-1)!(q-z)!}
       \left(\frac{\mu}{\nu+\mu}\right)^{q-z}
       \left(\frac{\nu}{\nu+\mu}\right)^{z-1} \\
        =& z\frac{\nu}{\lambda}\frac{\nu+\mu}{\nu}p_Q(q) \frac{q!}{z!(q-z)!}
       \left(\frac{\mu}{\nu+\mu}\right)^{q-z}
       \left(\frac{\nu}{\nu+\mu}\right)^{z}\\
        =&z\frac{\nu+\mu}{\lambda}\pe(q,z).
  \end{split}
\end{equation*}
Working analogously, we derive the following relations
\begin{align}
  \pe(q-1,z-1) =& z\frac{\nu+\mu}{\lambda}\pe(q,z), \label{eq:relation11}\\
  \pe(q+1,z+1) =& \frac{1}{z+1} \frac{\lambda}{\nu+\mu}\pe(q,z),  \label{eq:relation12}\\
 \pe(q,z+1) =& \frac{q-z}{z+1} \frac{\nu}{\mu}\pe(q,z),  \label{eq:relation13}\\
  \pe(q+1,z) =& \frac{1}{q-z+1} \frac{\lambda}{\nu}\frac{\mu}{\nu+\mu}\pe(q,z).
  \label{eq:relation14} \\ \nonumber
\end{align}
Using the above equations and recalling that $L(z)=z$ when $M=K$, the right-hand side of
\eqref{eq:balanceEquations} for $0<z,q<K$ and $z\neq q$
 can be written as follows
\begin{equation*}
\begin{split}
&\left(\lambda z\frac{\nu+\mu}{\lambda}+
(z+1)\nu \frac{1}{z+1} \frac{\lambda}{\nu+\mu}+
\mu (z+1)\frac{q-z}{z+1} \frac{\nu}{\mu}+
(q-z+1)\nu \frac{1}{q-z+1} \frac{\lambda}{\nu}\frac{\mu}{\nu+\mu}\right) \pe(q,z)\\
&= \left(z(\nu+\mu) + \lambda\frac{\nu}{\nu+\mu} +(q-z)\nu +\lambda \frac{\mu}{\nu+\mu}\right)
\pe(q,z)
=\left(q\nu+\lambda+z \mu\right) \pe(q,z).
\end{split}
\end{equation*}
That is, $\pe(q,z)$ satisfies \eqref{eq:balanceEquations} for $0<z,q<K$ and $z\neq q$.
To show that $\pe(q,z)$ satisfies \eqref{eq:balanceEquations} for $0<q<K$ and $z=0$, we apply
\eqref{eq:stationary dis} and \eqref{eq:relation in Erlang system} in the right-hand side of
\eqref{eq:balanceEquations}.
This leads to
\begin{align*}
 (q+1)\nu \pe(q+1) \left(\frac{\mu}{\nu+\mu}\right)^{q+1}+
\nu \pe(q+1) (q+1)\left(\frac{\mu}{\nu+\mu}\right)^{q}\frac{\nu}{\nu+\mu}+
\mu \pe(q) q \left(\frac{\mu}{\nu+\mu}\right)^{q-1}\frac{\nu}{\nu+\mu}\\
=\left( \lambda \left(\frac{\mu}{\nu+\mu}\right)^{q+1}+
\lambda \left(\frac{\mu}{\nu+\mu}\right)^{q}\frac{\nu}{\nu+\mu}+
\mu q \left(\frac{\mu}{\nu+\mu}\right)^{q-1}\frac{\nu}{\nu+\mu}\right) p_Q(q)\\
 =\left( \lambda+ \mu q \frac{\nu+\mu}{\mu}\frac{\nu}{\nu+\mu}\right)
\left(\frac{\mu}{\nu+\mu}\right)^q p_Q(q)
= (\lambda + q\nu)p_Q(q,0).
\end{align*}
In the same way, we show that the right-hand side of \eqref{eq:balanceEquations} for
$0<z<K$ and $q=K$ becomes
\begin{align*}
\bigg(&\lambda K \frac{\nu}{\lambda} \frac{(K-1)!}{(z-1)!(K-z)!}
\Big(\frac{\mu}{\nu+\mu}\Big)^{K-z}\Big(\frac{\nu}{\nu+\mu}\Big)^{z-1}\\
&\qquad \qquad+\mu (z+1)\frac{K!}{(z+1)!(K-z-1)!}\Big(\frac{\mu}{\nu+\mu}\Big)^{K-z-1}
\Big(\frac{\nu}{\nu+\mu}\Big)^{z+1} \bigg) p_Q(K).
\end{align*}
The last quantity is equal to
\begin{align*}
\big(z \nu \frac{\nu+\mu}{\nu} + (K-z)\mu \frac{\nu}{\nu+\mu}\frac{\nu+\mu}{\mu}\big)
 \frac{K!}{z!(K-z)!}
\Big(\frac{\mu}{\nu+\mu}\Big)^{K-z}\Big(\frac{\nu}{\nu+\mu}\Big)^{z} p_Q(K)
= (K \nu + z \mu) \pe(K,z).
\end{align*}
Using again the relations \eqref{eq:relation11}, \eqref{eq:relation12} and \eqref{eq:relation14}, the
right-hand side of \eqref{eq:balanceEquations} is written for $q<K$ and $q=z$  as follows
\begin{align*}
\left(\lambda q \frac{\nu+\mu}{\lambda}+\nu \frac{\lambda}{\nu+\mu}
+ \nu \frac{\lambda}{\nu}\frac{\mu}{\nu+\mu}\right) \pe(q,q)
=&\left( q (\nu+\mu)+\lambda \frac{\nu}{\nu+\mu}+\lambda \frac{\mu}{\nu+\mu}\right) \pe(q,q)\\
=&(q(\nu+\mu)+\lambda ) \pe(q,q).
\end{align*}
Using again relations \eqref{eq:relation11}$\textup{--}$\eqref{eq:relation14}, it follows
immediately that $\pe(q,z)$ satisfies Equation \eqref{eq:balanceEquations} also for the
remaining cases, i.e., $(q,z)=(0,0)$, $(q,z)=(K,K)$, and last $(q,z)=(K,0)$.
\end{proof}

\begin{proof}[Proof of Proposition~\ref{Prop:erlang A}]
We show that $Z^{\infty}_M(\cdot)$ behaves as modified Erlang-A queue.
Although we adapt
the proof in \cite[Section 6.6.1]{zeltyn2005call}, we briefly describe it here for completeness. First, we write the flow balance
equations for the Markov process $Z^{\infty}_M(\cdot)$ and then we solve them. 	
The balance equations for the Markov process $Z^{\infty}_M(\cdot)$ are given by
\begin{equation}\label{eq:balance2}
\begin{cases}
\lambda \pz(z)+z(\nu +\mu)\pz(z)=\lambda \pz(z-1)+(z+1)(\nu +\mu) \pz(z+1),&
\mbox{if }\  0<z<M,\\
\lambda \pz(z)+(M\mu +z\nu)\pi(z)=\lambda \pz(z-1)+(M \mu +(z+1)\nu) \pz(z+1),&
\mbox{if }\ z\geq M,
\end{cases}
\end{equation}
and for $z=0$ we have that
\begin{equation*}
\lambda \pz(0)=(\nu +\mu) \pz(1).
\end{equation*}
Using the last equation and \eqref{eq:balance2}, we derive inductively the following relations
\begin{equation*}
\lambda \pz(z-1)=z(\nu +\mu)\pz(z),\ \text{if}\ z<M
\end{equation*}
and
\begin{equation*}
(M\mu +z\nu)\pz(z)=\lambda \pz(z-1) ,\ \text{if}\ z \geq M .
\end{equation*}
The balance equations now can be simplified as follows
\begin{equation}\label{eq:alternative_balance}
\begin{cases}
\lambda \pz(z)=(z+1)(\nu +\mu) \pz(z+1),& \mbox{if }  z<M,\\
\lambda \pz(z)=(M \mu +(z+1)\nu) \pz(z+1),& \mbox{if }  z\geq M.
\end{cases}
\end{equation}
Observe that we can directly solve the system \eqref{eq:alternative_balance}. For $z<M$, it is
easy to see that
\begin{equation}\label{eq:solution a}
\pz(z) = \frac{1}{z!}\left(\frac{\lambda}{\nu +\mu}\right)^z \pz(0).
\end{equation}
We show that for $z=M$, the solution of \eqref{eq:alternative_balance} is also given by the last
formula.
By the first equation of \eqref{eq:alternative_balance} for $z=M-1$ and \eqref{eq:solution a},
we obtain the following equation
\begin{equation*}
\pz(M) =
\frac{1}{M}\left(\frac{\lambda}{\nu +\mu}\right)^M \pz(M-1)=
\frac{1}{M!}\left(\frac{\lambda}{\nu +\mu}\right)^M \pz(0).
\end{equation*}
It remains to find the solution in case $z>M$. We do so by induction. Note that by the second
equation of \eqref{eq:alternative_balance} for $z=M$, we have that
\begin{equation*}
\pz(M+1) =
\frac{\lambda}{M\mu +(M+1)\nu}\pz(M)=
\frac{\lambda}{M\mu +(M+1)\nu}
\frac{1}{M!}\left(\frac{\lambda}{\nu +\mu}\right)^M \pz(0).
\end{equation*}
Finally, it is easy to verify that the solution of \eqref{eq:alternative_balance} for $z>M$ is
given by
\begin{equation*}
\pz(z) =
\frac{1}{M!}\left(\frac{\lambda}{\nu +\mu}\right)^M
\prod_{k=M+1}^{z} \frac{\lambda}{M\mu +k\nu}
\pz(0).
\end{equation*}
The probability of an empty system (there are not uncharged vehicles in the parking lot) can be found by the
normalization condition and it is given by \eqref{eq:empty system 2}. Last, we show that the
infinite summation in \eqref{eq:empty system 2} converges. To this end, note that
$L(z)\mu+z\nu\geq z \min\{\nu,\mu\}$. Applying the last observation in \eqref{eq:empty system 2},
we have that
\begin{equation*}
\begin{split}
\sum_{j=0}^{M} \frac{1}{j!}\Big(\frac{\lambda}{\nu +\mu}\Big)^j +
\sum_{j=M+1}^{\infty} \frac{1}{M!}
\Big(\frac{\lambda}{\nu +\mu}\Big)^M\prod_{k=M+1}^{j}\frac{\lambda}{M\mu +k\nu}
\leq & \sum_{j=0}^{\infty} \frac{1}{j!} \Big(\frac{\lambda}{\min\{\nu,\mu\}}\Big)^j\\
=& \exp\{\frac{\lambda}{\min\{\nu,\mu\}} \}.
\end{split}
\end{equation*}
\end{proof}

\begin{proof}[Proof of Proposition~\ref{Prop:overloaded}]
First, we write the balance equations for the one-dimensional birth-death process $\{Z_f(t), t \geq 0\}$. These are given by
\begin{equation*}
\begin{cases}
  \left( \nu(K-z)+\mu z \right)\pf(z) =
  ( \nu(K-z+1)\pf(z-1))+\mu (z+1) \pf(z+1), &\mbox{ if } 0<z<M,\\
\left( \nu(K-z)+\mu M \right)\pf(z)=
 ( \nu(K-z+1)\pf(z-1))+\mu  M \pf(z+1), &\mbox{ if } M\leq z<K,
\end{cases}
\end{equation*}
and on the boundary the following equations hold
\begin{equation*}
  \nu K \pf(0) = \mu \pf(1)  \mbox{ and }
  \mu M \pf(K)=\nu \pf(K-1).
\end{equation*}
Note that the balance equations can be simplified to
\begin{equation}\label{eq:balancefull}
\begin{cases}
   \nu(K-z) \pf(z) = \mu (z+1) \pf(z+1), &\mbox{ if } 0\leq z<M,\\
\mu M \pf(z)=\nu(K-z+1)\pf(z-1), &\mbox{ if } M\leq z\leq K.
\end{cases}
\end{equation}
Applying $z=M-1$ in the first equation of \eqref{eq:balancefull}, we obtain
\begin{equation*}
\pf(M-1)=\frac{\mu M}{\nu (K-M+1)}\pz(M),
\end{equation*}
and recursively we have that
\begin{equation*}
\pf(M-i)=\left(\frac{\mu}{\nu}\right)^i
\frac{\prod_{j=0}^{i-1} (M-j)}{\prod_{j=1}^{i} (K-M+j)}
 \pf(M), \mbox{ if } 0<i \leq M.
\end{equation*}
Change the variable $z=M-i $ in the last equation yields
\begin{equation*}
\pf(z)=\left(\frac{\mu}{\nu}\right)^{M-z}
\frac{\prod_{j=0}^{M-z-1} (M-j)}{\prod_{j=1}^{M-z} (K-M-j)}
 \pf(M),  \mbox{ if } 0 \leq z< M.
\end{equation*}
Working analogously, by the second equation of \eqref{eq:balancefull} we derive
\begin{equation*}
\pf(M+i)=\frac{1}{M^i}
\left(\frac{\mu}{\nu}\right)^i
\prod_{j=0}^{i-1} (K-M-j)\pf(M) , \mbox{ if } 0\leq i \leq K-M,
\end{equation*}
which leads to
\begin{equation*}
\pf(z)=\frac{1}{M^{z-M}}
\left(\frac{\nu}{\mu}\right)^{z-M}\
 \prod_{j=0}^{z-M-1} (K-M-j) \pf(M),
  \mbox{ if } M\leq z \leq K.
\end{equation*}
Last, $\pf(M)$ is determined by the normalization equation $\sum_{z=0}^{K}\pf(z)=1$.
\end{proof}

\subsection{Proofs for Section~\ref{sec:Fluid approximation}}\label{sec:Fluid ap}
\begin{proof}[Proof of Proposition~\ref{prop:fluid limit}]
In this proof, we use martingales arguments. Define the  following filtration
$\mathcal{F}^n_t:=\sigma \left(Z^n(0), Q^n(0), Z^n(s), Q^n(s):0\leq s\leq t \right)$ augmented by including all the null sets for $t\geq 0$ and $n\geq 1$.
Applying the fluid scaling to the dynamical equation \eqref{eq:numberU}, we have that
\begin{align*}
Z^n(t)=Z^n(0)+N_{\lambda} \left(n \int_{0}^{t} \ind{\frac{Q^n(s)}{n}<K} ds\right)
-N_{\mu}\left(n\int_{0}^{t} \frac{Z^n(s)}{n}\Min M ds\right)
-N_{\nu,1}\left(n\int_{0}^{t} \frac{Z^n(s)}{n} ds\right).
\end{align*}
Defining the operator
$\bar{\mathcal{M}}_r^n=\frac{1}{n}\left(N_r(\cdot)-r \cdot\right)$ and
following \cite{pang2007martingale},
we can write
\begin{align*}
\frac{Z^n(t)}{n}=\frac{Z^n(0)}{n}+
\bar{\mathcal{M}}^n_{\lambda}
 \left(n t\right)
-\bar{\mathcal{M}}^n_{\mu}\left(n\int_{0}^{t} \frac{Z^n(s)}{n}\Min M ds\right)
-\bar{\mathcal{M}}^n_{\nu,1}\left(n\int_{0}^{t} \frac{Z^n(s)}{n} ds\right)\\
-\frac{1}{n}\int_{0}^{t} \ind{\frac{Q^n(s)}{n}=K} dN(\lambda n s)
+\lambda t
-\mu \int_{0}^{t} \frac{Z^n(s)}{n}\Min M ds
-\nu\int_{0}^{t} \frac{Z^n(s)}{n} ds.
\end{align*}
The term $\frac{1}{n}\int_{0}^{t} \ind{\frac{Q^n(s)}{n}=K} dN(\lambda n s)$ denotes the number of EVs that are lost due finding the system full under the fluid scaling.  By the discussion in  \cite[Section~6.7]{robert2013}, \cite[Equation~3.44]{kang2015}, and by the assumption $\frac{Q^n(0)}{n} \overset{d} \rightarrow K$,
 it turns out that
\begin{equation*}
\frac{1}{n}\int_{0}^{t} \ind{\frac{Q^n(s)}{n}=K} dN(\lambda n s)
\overset{d}\rightarrow \max\{\lambda-\nu K,0\} t.
\end{equation*}
 Further, observing that $\bar{\mathcal{M}}^n_{r}(\cdot)$ are zero mean martingales with respect to the filtration $\mathcal{F}_t^n$ for any $n\in \mathbb{N}$ (cf.\ \cite{pang2007martingale}) and taking $n \rightarrow \infty$, we derive that
$\frac{Z^n(\cdot)}{n} \overset{d} \rightarrow z(\cdot)$. Moreover, the limit function is characterized by the following functional equation
\begin{equation*}
z(t)=z(0)+\lambda t -\max\{\lambda-\nu K,0\} t
-\mu \int_{0}^{t} z(s)\Min M ds -\nu \int_{0}^{t}z(s) ds,
\end{equation*}
which is equivalent to \eqref{eq:fluid limit}.
\end{proof}

\begin{proof}[Proof of Proposition~\ref{prop:fluid approximation}]
It is not hard to solve the ODE \eqref{eq:fluid limit} explicitly in the two regions, namely
\begin{equation}\label{eq:solutionODE}
z(t)=
\begin{cases}
  \frac{\lambda\Min \nu K} {\nu+\mu}+
  \left(z(0)-\frac{\lambda\Min \nu K}{\nu+\mu}\right) e^{-(\nu+\mu)t}   ,
  & \mbox{ if }  z(t)\leq M,\\
  \frac{\lambda\Min \nu K-\mu M}{\nu}+
  \left(z(0)-\frac{\lambda\Min \nu K-\mu M}{\nu}\right) e^{-\nu t} ,
   & \mbox{ otherwise}.
\end{cases}
\end{equation}
Define $z_1:=\frac{\lambda\Min \nu K} {\nu+\mu}$ and
$z_2:=\frac{\lambda\Min \nu K-\mu M}{\nu}$. First, note that for a given initial state $z(0)\in[0,K]$, we have that $z(t)\leq K$ for any $t\geq 0$. To see this, observe that if $z(t)\leq M$, then by the model assumptions $z(t)\leq M \leq K$. On the other hand, if $z(t)> M$, we have that
\begin{equation*}
z(t)=z_2+ \left(z(0)-z_2\right) e^{-\nu t}
\leq z_2 \leq K,
\end{equation*}
if $z(0)-z_2\leq 0$ and
\begin{equation*}
z(t)=z_2+ \left(z(0)-z_2\right) e^{-\nu t}
\leq z_2 +z(0)-z_2=z(0)\leq K,
\end{equation*}
if $z(0)-z_2\geq  0$. So, for any $t\geq 0$ it follows $z(t)\leq K$ and hence Definition~\ref{def:fluid model} is well defined.

In the sequel, we show that $z(t)$ converges as $t$ goes to infinity  to the following point
\begin{equation}\label{eq:solutionfix}
z=
\begin{cases}
  z_1,
  & \mbox{ if }  z_1\leq M,\\
  z_2 , & \mbox{ otherwise}.
\end{cases}
\end{equation}
Then, we show that $z$ is unique and using the Markovian assumptions, we see that it is equivalent to \eqref{eq:fluid proxy1} and hence $z=z^*$. Also, observe that $z^*$ is an invariant point. Indeed, if we set $z(0)=z^*$, then by \eqref{eq:solutionODE} we have that $z(t)=z^*$ for $t\geq 0$.

First, assume that $z_1\leq M$. If $z(0)\leq M$, then
$z(t)\leq M$ for any $t\geq 0$. To see this, note that if
$z(0)-z_1\leq 0$, then we have that
\begin{equation*}
z_1+
  \left(z(0)-z_1\right) e^{-(\nu+\mu)t}
  \leq z_1 \leq M.
\end{equation*}
On the other hand, if the quantity
$z(0)-z_1$ is positive, we show that there does not  exist $t^*>0$ such that $z(t^*)>M$. Suppose that there exists $t^*$ such that $z(t^*)>M$, we have that
\begin{align*}
 z_1+\left(z(0)-z_1\right) e^{-(\nu+\mu)t^*}>M,
\end{align*}
by assumption $z(0)-z_1>0$, the last inequality leads to
\begin{align*}
 t^* < -\frac{1}{\nu+\mu} \ln \left( \frac{M-z_1}{z(0)-z_1} \right).
\end{align*}
Now, by the assumption $z(0)\leq M$, we obtain $\frac{M-z_1}{z(0)-z_1}\geq 1$ and hence $t^*\leq 0$ which yields a contradiction.
Next, we assume that $z(0)>M$. In this case, we show that there exists $t^*_2$ such that $z(t^*_2)\leq M$. First, observe that if $z_1\leq M$ then $z_2\leq M$. Hence, $z(0)-z_2>M-z_2\geq 0$. Now, we note that
\begin{align*}
z_2+\left(z(0)-z_2\right) e^{-\nu t}\leq M
\end{align*}
leads to
\begin{align*}
 t > -\frac{1}{\nu} \ln \left( \frac{M-z_2}{z(0)-z_2} \right)>0,
\end{align*}
due to $z(0)-z_2>M-z_2$. Let $t^*_2$ be the first time that $z(t)\leq M$ starting from a point $z(0)>M$. Setting now as initial point $z(t^*_2)$, we conclude that if $z_1\leq M$ then $z(t)\leq M$ for $t\geq t^*_2$ and hence
\begin{align*}
z(t)=z_1+
  \left(z(t^*_2)-z_1\right) e^{-(\nu+\mu)t}, \mbox{ for } t\geq t^*_2,
\end{align*}
which yields
\begin{align*}
|z(t)-z_1|\leq
  \left(z(t^*_2)-z_1\right) e^{-(\nu+\mu)t}
  \leq \left(z(0)-z_1\right) e^{-(\nu+\mu)t}, \mbox{ for } t\geq t^*_2.
\end{align*}
This concludes the proof for the case $z_1\leq M$.  The case $z_2>M$ follows the same logic.

Now, we prove that \eqref{eq:fluid proxy1} is equal to $z$. To this end, observe that
\begin{equation*}
B \max\{ 1, \frac{ z^*}{M}\} \overset {d}=B',
\end{equation*}
where $B'$ is an exponential random variable  with
$\E{B'}= \frac{1}{\mu \max{1, \frac{z^*}{M}}}$. We note that \eqref{eq:fluid proxy1} can be written as
\begin{align*}
z^* = (\lambda \Min \nu K) \E{ \min \{D , B \max\{ 1, \frac{ z^*}{M}\}\}}
    &=(\lambda \Min \nu K) \E{ \min \{D , B'\}}\\
    &=(\lambda \Min \nu K)\frac{1}{\nu+\mu \max\{1, \frac{z^*}{M}\}}.
\end{align*}
Solving the last equation yields $z^*=z$.

To conclude the proof of Proposition~\ref{prop:fluid approximation}, it remains to show the uniqueness of the invariant point $z^*$. In other worlds, we show that \eqref{eq:solutionfix} (and hence \eqref{eq:fluid proxy1}) has a unique solution.
It is not hard to see that if
$z_1< M$ then $z_2< M$. So, $z_2$ cannot be the solution of \eqref{eq:solutionfix}. That is, $z_1$ is the unique solution of \eqref{eq:solutionfix}. On the other hand, if $z_1>M$ (i.e., it is not solution of \eqref{eq:solutionfix}), then $z_2>M$. That is, $z_2$ is the unique solution of \eqref{eq:solutionfix}. Last, if $z_1=M$, then we have that $z_2=z_1=M$.
In any case, \eqref{eq:solutionfix} has a unique solution.
\end{proof}

\begin{proof}[Proof of Proposition~\ref{prop:interchange}]
Let $\frac{Z^n(\infty)}{n}$ be the stationary fluid scaled number of uncharged EVs.
We know that $0 \leq Z^n(\infty)\leq K^n$, which yields $\frac{Z^n(\infty)}{n}\leq K$ almost surely. In other worlds, the sequence of random variables $\frac{Z^n(\infty)}{n}$ is stochastically bounded in $\R$ and hence it is tight. Now, we consider the process $\{Z^n(t), t\geq 0\}$ starting at point $Z^n(\infty)$. That is,
$Z^n(t) \overset{d}= Z^n(\infty)$ for any $t\geq 0$.
Since $\frac{Z^n(\infty)}{n}$ is tight for any convergent subsequence, there exists a further subsequence, say $\frac{Z^{\bar{n}}(\infty)}{\bar{n}}$, such that
$\frac{Z^{\bar{n}}(\infty)}{\bar{n}}\overset{d}\rightarrow \bar{z}^*$,
as $\bar{n}\rightarrow \infty$.
We now have  that for any $t\geq 0$,
\begin{equation*}
\frac{Z^{\bar{n}}(t)}{\bar{n}}\overset{d}=
\frac{Z^{\bar{n}}(\infty)}{\bar{n}}\overset{d}\rightarrow \bar{z}^*, \mbox{ as } \bar{n}\rightarrow \infty,
\end{equation*}
and so $\bar{z}^*$ in an invariant point. By the uniqueness of the invariant point we derive $\bar{z}^*=z^*$. This concludes the proof.
\end{proof}

\subsection{Proofs for Section~\ref{sec:DiffAp}}\label{sec:HWregime}
\subsubsection{Proof of Theorem~\ref{Th:diffusion}}\label{sec:HW}
We start the analysis by establishing a continuity result, which can be proved by using results in \cite{pang2007martingale}

\begin{proposition}\label{Th: existence}
Let $t\geq 0$ and $-\infty < \kappa < \infty$. Consider the following system
\begin{equation}\label{eq:system}
\begin{split}
   x_1(t)&=b_1+g_1(t)+\int_{0}^{t} h_1(x_1(s))ds-y(t), \\
   x_2(t)&=b_2+g_2(t)+\int_{0}^{t} h_2(x_2(s))ds-y(t),
\end{split}
\end{equation}
where $b_i$ are positive constants, $h_i: \R \rightarrow \R$ satisfy $h_i(0)=0$ and are
Lipschitz continuous functions for $i=1,2$, and $x_2(t)\leq \kappa$.
In addition, $y(\cdot)$ is a nondecreasing nonnegative function
in $D[0,\infty)$
such that \eqref{eq:system} holds and
$\int_{0}^{\infty} \ind{ x_2(t)<\kappa}dy(t)=0$. Given $b_i \in \R$
and $g_i(\cdot) \in D[0,\infty)$, we have that the system \eqref{eq:system}
has a unique solution $(x_1(\cdot),x_2(\cdot),y(\cdot))$.
Moreover, the functions $(x_1(\cdot),x_2(\cdot),y(\cdot))$
are continuous in $D[0,\infty)^3$ if $D[0,\infty)$ is endowed with the uniform
topology over bounded intervals or the $J_1$ topology.
\end{proposition}

\begin{proof}[Proof of Theorem~\ref{Th: existence}]
First, observe that the function $y(\cdot)$ is independent of the function $x_1(\cdot)$.
We know by \cite[Theorem 7.3]{pang2007martingale} that the second equation
of \eqref{eq:system} has a unique solution $(x_2(\cdot),y(\cdot))$ and that $x_2(\cdot)$,
$y(\cdot)$ are continuous in $D[0,\infty)$ (endowed with the uniform topology over bounded
intervals or the $J_1$ topology). Furthermore, we have that $y(\cdot)$, $g_2(\cdot) \in
D[0,\infty)$
which implies $y(\cdot)+g_2(\cdot) \in D[0,\infty)$. The last observation together
with \cite[Theorem 4.1]{pang2007martingale} implies that the
first equation of \eqref{eq:system} has a unique continuous solution.
That is, the system \eqref{eq:system} has a unique
solution $(x_1(\cdot),x_2(\cdot),y(\cdot))$ and each function is continuous.
\end{proof}
In order to continue our analysis, we need to define
appropriate filtrations. Take the following filtrations, for $n\geq 1$,
\begin{align*}
\mathcal{F}_{t,1}^{n}=\sigma \left(Z^n(0), N_{\lambda^n} (s),
N_{\mu}\left(\int_{0}^{s} L^n(Z^n(z)) dz\right),
N_{\nu,1}\left(\int_{0}^{s} Z^n(z) dz\right): 0 \leq s \leq t \right)
 \end{align*}
and
\begin{equation*}
\mathcal{F}_{t,2}^{n}=\sigma \left(Q^n(0), N_{\lambda^n} (s),
N_{\nu}\left(\int_{0}^{s} Q^n(z) dz\right): 0 \leq s \leq t \right).
 \end{equation*}
In the sequel, we work with the filtrations
 \begin{equation*}
\mathcal{F}_{t}^{n}=\sigma \left(\mathcal{F}_{t,1}^{n}, \mathcal{F}_{t,2}^{n} \right),
 \end{equation*}
augmented by including all the null sets for $t\geq 0$ and $n\geq 1$.

Now, notice that the system dynamics \eqref{eq:numberQ} and \eqref{eq:numberU} can be
rewritten in the following form
\begin{equation*}
\begin{split}
Q^n(t)=Q^n(0)+N_{\lambda^n} (t)
-N_{\nu}\left(\int_{0}^{t}Q^n(s) ds\right)-Y^n(t)
\end{split}
\end{equation*}
and
\begin{align*}
Z^n(t)=Z^n(0)+N_{\lambda^n} (t)
-N_{\mu}\left(\int_{0}^{t} L^n(Z^n(s)) ds\right)
-N_{\nu,1}\left(\int_{0}^{t} Z^n(s) ds\right)
-Y^n(t),
\end{align*}
where $Y^n(t)=\int_{0}^{t} \ind{Q^n(s)=K^n}dN_{\lambda^n} (s)$.
The process $Y^n(t)$ counts all the customers that are lost when
all the servers (chargers) are busy up to time $t$ in the $n^{\text{th}}$ system.
Defining the operator $\mathcal{M}_{r}(\cdot ):=N_{r} (\cdot)-(r\cdot)$, where $``r"$
indicates the rate of the Poisson process $N_{r} (\cdot)$, the system dynamics
take the following form:
\begin{align}\label{Eq:MartingaleQ}
Q^n(t)=Q^n(0)
+\mathcal{M}_{\lambda^n}(t)
-\mathcal{M}_{\nu}\left(\int_{0}^{t}Q^n(s) ds\right)
+\lambda^n t
-\nu \int_{0}^{t}Q^n(s) ds
-Y^n(t)
\end{align}
and
\begin{align}\label{Eq:MartingaleU}
Z^n(t)=Z^n(0)
+\mathcal{M}_{\lambda^n}(t)
-\mathcal{M}_{\mu}\left(\int_{0}^{t} L^n(Z^n(s)) ds\right)
-\mathcal{M}_{\nu,1}\left(\int_{0}^{t} Z^n(s) ds\right)\nonumber\\
+\lambda^n t
-\mu \int_{0}^{t} L^n(Z^n(s)) ds
-\nu \int_{0}^{t} Z^n(s) ds
-Y^n(t).
\end{align}
In order to derive appropriate equations (in the pre-limit) for the diffusion scaled processes,
subtract and add the terms
$n \frac{\nu+\mu}{\nu}$, $n \frac{\nu+\mu}{\nu} t$ in \eqref{Eq:MartingaleQ}
and the terms
$n$, $n \mu t$, $n \nu  t$ in \eqref{Eq:MartingaleU},
and then divide both by $\sqrt{n}$. Recalling that $L^n(Z^n(t))= Z^n(t)\Min M^n$,
$\lambda^n=n(\nu+\mu)$, $M^n=\frac{\lambda^n}{\nu+\mu}+\beta\sqrt{n}$ and
observing that $M^n-n=\beta\sqrt{n}$ and $\lambda^n t- n(\nu+\mu)t=0 $,
we obtain the following equations for the diffusion scaled processes
$\hat Q^n(\cdot)$ and $\hat Z^n(\cdot)$,
\begin{align}\label{Eq:difQ}
\hat Q^n(t)=\hat Q^n(0)
+ \mathcal{\hat M}_{\lambda^n}^n(t)
- \mathcal{\hat M}_{\nu}^n \left(\int_{0}^{t}Q^n(s) ds\right)
-\nu \int_{0}^{t}\hat Q^n(s) ds
-\hat Y^n(t)
\end{align}
and
\begin{align}\label{Eq:difU}
\hat Z^n(t)=\hat Z^n(0)
+ \mathcal{\hat M}_{\lambda^n}^n (t)
-\mathcal{\hat M}_{\mu}^n \left(\int_{0}^{t} L^n(Z^n(s)) ds\right)
- \mathcal{\hat M}_{\nu,1}^n\left(\int_{0}^{t} Z^n(s) ds\right) \nonumber \\
-\mu \int_{0}^{t} (\hat Z^n(s)\Min \beta) ds
-\nu \int_{0}^{t} \hat Z^n(s) ds
-\hat Y^n(t),
\end{align}
where $\mathcal{\hat M}_{r}^n(\cdot):= \frac{ \mathcal{ M}_{r}(\cdot)}{\sqrt{n}}$
and the scaling for the process $\hat Y^n(\cdot)$ is analogous. The following proposition shows
that
the processes $\mathcal{\hat M}_r(\cdot)$ are martingales.

\begin{proposition}\label{Pr:Martingales}
Under the assumptions $\E{Z^n(0)}< \infty $ and $\E{Q^n(0)} < \infty $, we have that the processes
$\mathcal{\hat M}_{\lambda^n}^n (\cdot)$,
$\mathcal{\hat M}_{\mu}^n(\cdot)$, $\mathcal{\hat M}_{\nu,1}^n(\cdot)$, and
$\mathcal{\hat M}_{\nu}^n(\cdot)$ are
square-integrable martingales with respect to the filtration
$\mathcal{F}^{n}:=\{\mathcal{F}_{t}^{n}, t \geq 0\}$.
Their associated predictable quadratic variations, denoted by $<\cdot >$, are
\begin{align}\label{}
  <\mathcal{\hat M}_{\lambda^n}^n(t)> =& \frac{\lambda^n}{n} t=(\nu+\mu)t, \label{eq:pq1}\\
  <\mathcal{\hat M}_{\mu}^n \left(\int_{0}^{t} L^n(Z^n(s)) ds\right)>=
  & \mu\frac{\int_{0}^{t} L^n(Z^n(s)) ds}{n}, \label{eq:pq2}\\
  <\mathcal{\hat M}_{\nu,1}^n\left(\int_{0}^{t} Z^n(s) ds\right)> =&
  \nu\frac{\int_{0}^{t} Z^n(s) ds}{n}, \label{eq:pq3}\\
  <\mathcal{\hat M}_{\nu}^n\left(\int_{0}^{t} Q^n(s) ds\right)> =&
  \nu\frac{\int_{0}^{t} Q^n(s) ds}{n},  \label{eq:pq4}
\end{align}
and $\E{<\mathcal{\hat M}_{r}^n(t)>}<\infty$, for $t\geq 0$ and $n \geq 1$.
\end{proposition}

\begin{proof}
Fix $n\geq 1$. The result for the process $\mathcal{\hat M}_{\lambda^n}^n (\cdot)$ follows
immediately by applying
\cite[Lemma 3.1]{pang2007martingale}.

Now, note that by the system dynamics \eqref{eq:numberQ}, \eqref{eq:numberU}, the fact that $Q^n(t)=Z^n(t)+C^n(t)$, and \eqref{eq:mergeP} we have that
\begin{equation}\label{eq:DecomQ}
\mathcal{\hat M}_{\nu}^n\left(\int_{0}^{t} Q^n(s) ds\right)=
\mathcal{\hat M}_{\nu,1}^n\left(\int_{0}^{t} Z^n(s) ds\right)
+\mathcal{\hat M}_{\nu,2}^n\left(\int_{0}^{t} C^n(s) ds\right),
\end{equation}
where
\begin{equation*}
  \mathcal{\hat M}_{\nu,2}^n\left(\int_{0}^{t} C^n(s) ds\right)
= \frac{
N_{\nu,2}^n\left(\int_{0}^{t} C^n(s) ds\right)-\nu\int_{0}^{t} C^n(s) ds
}
{\sqrt{n}}.
\end{equation*}
Since $N_{\nu,1}(\cdot), N_{\nu,2}(\cdot)$ are independent Poisson processes,
\cite[Lemma 3.1]{pang2007martingale} implies that $\mathcal{\hat M}_{\nu,i}^n(\cdot)$ are
$\mathcal{F}^{n}$-martingales for $i=1,2$. Observe that $\mathcal{\hat M}_{\nu,2}^n(\cdot)$ is
adapted to the filtration $\mathcal{F}^{n}$ as the latter contains all the information about the
processes $Q^n(\cdot)$ and $Z^n(\cdot)$ for fixed $n$. This is enough to determine the process
$C^n(\cdot)$ at any $t\geq 0$.
Using \eqref{eq:DecomQ}, the assumptions
$\E{Z^n(0)}, \E{Q^n(0)} < \infty $, the inequalities
\begin{align}\label{eq:crude}
  Z^n(t) & \leq Z^n(0)+N_{\lambda^n}(t), \\
  Q^n(t) & \leq Q^n(0)+N_{\lambda^n}(t), \nonumber
\end{align}
and adapting the proof in \cite[Lemma 3.4]{pang2007martingale}, we obtain that for fixed
$n\geq 1$, the following moments conditions are satisfied
\begin{align*}
 &\E{\int_{0}^{t} Q^n(s) ds}  < \infty,\qquad
 \E{\int_{0}^{t} Z^n(s) ds} < \infty, \\
 &\E{\int_{0}^{t} C^n(s) ds} \leq   \E{\int_{0}^{t} Q^n(s) ds}< \infty, \\
 &\E{\int_{0}^{t} L^n(Z^n(s)) ds}=
 \E{\int_{0}^{t} Z^n(s)\Min M^n ds} < \infty.
\end{align*}
Also, again by \cite[Lemma 3.4]{pang2007martingale}, we derive that the following moments related
to the Poisson processes are finite.
\begin{align*}
 &\E{N_{\nu}^n\left(\int_{0}^{t} Q^n(s) ds\right)}  < \infty,\qquad
 \E{N_{\nu,1}^n\left(\int_{0}^{t} Z^n(s) ds\right)} < \infty, \\
 &\E{N_{\nu,2}^n\left(\int_{0}^{t} C^n(s) ds\right)}  < \infty, \qquad
 \E{N_{\mu}^n\left(\int_{0}^{t} L^n(Z^n(s)) ds\right)} < \infty.
\end{align*}
Now, the result follows by the conclusion of the proof of
\cite[Theorem~7.1]{pang2007martingale}. To this end, observe that
$\left(\int_{0}^{t} Z^n(s) ds, \int_{0}^{t} C^n(s) ds,
\int_{0}^{t} Z^n(s)\Min M^n ds\right)$ is a $\mathcal{F}^{n}$-stoping time. Thus, applying
\cite[Theorem~8.7]{ethier1986markov}, we derive that
\begin{equation*}
\left(\mathcal{\hat M}_{\lambda^n}^n (t),
\mathcal{\hat M}_{\mu}^n\left(\int_{0}^{t} L^n(Z^n(s)) ds\right),
\mathcal{\hat M}_{\nu,1}^n\left(\int_{0}^{t} Z^n(s) ds\right),
\mathcal{\hat M}_{\nu,2}^n\left(\int_{0}^{t} C^n(s) ds\right)\right)
\end{equation*}
is an $\mathcal{F}^{n}$-martingale. Last, note that by \eqref{eq:DecomQ}, as $\mathcal{\hat M}_{\mu}^n(\cdot)$ is adapted to the
filtration $\mathcal{F}^{n}$, we obtain that
$\mathcal{\hat M}_{\nu}^n\left(\int_{0}^{t} Q^n(s) ds\right)$
is also an $\mathcal{F}^{n}$-martingale.
\end{proof}
In order to apply the martingale central limit theorem, we first need to show that the
corresponding fluid scaled processes converge to a deterministic function (step 3), i.e.,
\begin{equation}\label{eq:fluid limitU}
\bar Z^n(\cdot):=
\frac{Z^n(\cdot)}{n}
\overset{d} \rightarrow
e(\cdot)
\end{equation}
and
\begin{equation}\label{eq:fluid limitQ}
\bar Q^n(\cdot):=
\frac{Q^n(\cdot)}{n}
\overset{d} \rightarrow
\frac{\nu+\mu}
{\mu}e(\cdot),
\end{equation}
where the function $e: [0,\infty)\rightarrow [0,\infty)$ is defined by $e(t)\equiv 1$. The
following proposition presents the fluid limit.

\begin{proposition}\label{Pr:fluid limit}
If
$
\bar Z^n(0):=
\frac{Z^n(0)}{n}
\overset{d} \rightarrow
e(\cdot)
$
and
$
\bar Q^n(0):=
\frac{Q^n(0)}{n}
\overset{d} \rightarrow
\frac{\nu+\mu}
{\mu}e(\cdot),
$
then
\eqref{eq:fluid limitU} and \eqref{eq:fluid limitQ} hold
as $n \rightarrow \infty$.
\end{proposition}

\begin{proof}
We prove the fluid limits using the martingale representations \eqref{Eq:difQ} and
\eqref{Eq:difU}. If the sequences $\{\hat Z^n(\cdot)\}$
and $\{\hat Q^n(\cdot)\}$ are stochastically bounded in $D[0,\infty)$ then by
\cite[Lemma~5.9]{pang2007martingale}, we have that

$$\frac{\hat Z^n(\cdot)}{\sqrt{n}} \overset{d} \rightarrow \eta(\cdot) \mbox{ and }
\frac{\hat Q^n(\cdot)}{\sqrt{n}} \overset{d} \rightarrow \eta(\cdot),
$$
where the function $\eta: [0,\infty)\rightarrow [0,\infty)$ is defined by $\eta(t)\equiv 0$. Note
that the last limits are equivalent to \eqref{eq:fluid limitU} and
\eqref{eq:fluid limitQ}.

The diffusion scaled processes $\hat Z^n(\cdot)$ and $\hat Q^n(\cdot)$ have the martingale
representation \eqref{Eq:difQ} and \eqref{Eq:difU}. In order to prove that they are
stochastically bounded in $D[0,\infty)$, it is enough to show that the corresponding martingales are
stochastically bounded in $D[0,\infty)$; see \cite[Lemma~5.5]{pang2007martingale}. But by
\cite[Lemma~5.8]{pang2007martingale} the martingales are stochastically bounded if the sequence
of their predictable quadratic variations \eqref{eq:pq1}--\eqref{eq:pq4} are stochastically
bounded in $\R$ for each $t\geq 0$.
To prove that the sequences of the predictable quadratic variations of the martingales in
expressions \eqref{Eq:difQ} and \eqref{Eq:difU} are stochastically bounded in $\R$, we use
\eqref{eq:crude}, \cite[Lemma~6.2]{pang2007martingale} and the fact that for any
$t\geq 0$, $\frac{Q^n(t)}{n}\leq \frac{K^n}{n}< \infty$.

The predictable quadratic variation for the arrival process \eqref{eq:pq1} is obviously bounded
as it is a deterministic function. For \eqref{eq:pq1}, we have that
\begin{equation*}
  \mu\frac{\int_{0}^{t} L^n(Z^n(s)) ds}{n}\leq \mu t \frac{M^n}{\sqrt{n}}
  =\mu t(1+\frac{\beta}{\sqrt{n}})\leq \mu t(1+\beta).
\end{equation*}
The result for \eqref{eq:pq3} follows by applying \eqref{eq:crude} and
\cite[Lemma~6.2]{pang2007martingale}.
Last, applying the inequality $\frac{Q^n(t)}{n}\leq \frac{K^n}{n}$ in \eqref{eq:pq4} we obtain
\begin{equation*}
  \nu\frac{\int_{0}^{t} Q^n(s) ds}{n}\leq \nu t \frac{K^n}{n}=
   \nu t(\frac{\nu+\mu}{\nu}+\frac{\kappa}{\sqrt{n}})\leq
   \nu t(\frac{\nu+\mu}{\nu}+\kappa).
\end{equation*}
\end{proof}
Before we move to the final step of the proof of Theorem~\ref{Th:diffusion},
we make a remark for the fluid limit of the number of fully charged EVs in the , which we
need later.
\begin{remark}\label{rm:C} Note that the diffusion scaled process $\{\hat C^n(\cdot)\}$ is also
stochastically bounded in $D[0,\infty)$ and the fluid limit is given by the difference
$\frac{\nu+\mu}{\nu}-1=\frac{\mu}{\nu}$.
\end{remark}

Now, we are ready to put all the pieces together, leading to the last step of the proof of
Theorem~\ref{Th:diffusion}.

\begin{proof}[Proof of Theorem~\ref{Th:diffusion}]
By Proposition~\ref{Pr:fluid limit} and the continuous mapping theorem
\cite[Theorem~1.2]{durrett1996stochastic}, we have that as $n\rightarrow \infty$,
\begin{align*}
  &<\mathcal{\hat M}_{\lambda^n}^n(t)> \rightarrow
  (\nu+\mu)t, \quad
  <\mathcal{\hat M}_{\mu}^n \left(\int_{0}^{t} L^n(Z^n(s)) ds\right)>\rightarrow
   \mu t,\\
  &<\mathcal{\hat M}_{\nu,1}^n\left(\int_{0}^{t} Z^n(s) ds\right)> \rightarrow
   \nu t , \quad
  <\mathcal{\hat M}_{\nu}^n\left(\int_{0}^{t} Q^n(s) ds\right)> \rightarrow
  (\nu+\mu)t.
\end{align*}
Applying the martingale central limit theorem in \cite{ethier1986markov}, we have that
\begin{equation}\label{eq:MCLT}
\begin{split}
\left(\mathcal{\hat M}_{\lambda^n}^n(\cdot),
\mathcal{\hat M}_{\mu}^n (\cdot),
\mathcal{\hat M}_{\nu,1}^n(\cdot),
\mathcal{\hat M}_{\nu}^n(\cdot) \right) \overset{d} \rightarrow
\left( \sqrt{\nu+\mu}W_{\lambda}(\cdot),
\sqrt{\mu}W_{\mu}(\cdot), \sqrt{\nu}W_{\nu,1}(\cdot), \sqrt{\nu+\mu}W_{\nu}(\cdot) \right),
\end{split}
\end{equation}
where $W_{\lambda}(\cdot)$, $W_{\mu}(\cdot)$, $W_{\nu,1}(\cdot)$, and $W_{\nu}(\cdot)$ are (non-independent) standard
Brownian motions. It is essential to observe that by \eqref{eq:DecomQ} and Remark~\ref{rm:C}, we
have that
$$\sqrt{\nu+\mu}W_{\nu}\overset{d} = \sqrt{\nu}W_{\nu,1}+\sqrt{\mu}W_{\nu,2},$$
where now $W_{\lambda}(\cdot)$, $W_{\mu}(\cdot)$, $W_{\nu,i}(\cdot)$, $i=1,2$ are
independent standard Brownian motions.
Furthermore, by the properties of the Brownian motion we obtain
\begin{align}
  \sqrt{\nu+\mu}W_{\lambda}-\sqrt{\mu}W_{\mu}-\sqrt{\nu}W_{\nu,1}
  \overset{d}=&\sqrt{2(\nu+\mu)} W_{\hat Z}, \label{eq:relation1} \\
  \sqrt{\nu+\mu}W_{\lambda}-\sqrt{\nu}W_{\nu,1}-\sqrt{\mu}W_{\nu,2}
   \overset{d}=&\sqrt{2(\nu+\mu)} W_{\hat Q}, \label{eq:relation2}
\end{align}
where $W_{\hat{Z}}(\cdot)$ and $W_{\hat{Q}}(\cdot)$ are non-independent standard Brownian motions. Further, we have that
\begin{equation*}
\E{W_{\hat Z}(t) W_{\hat Q}(t)}= \frac{
\E{(\nu+\mu)W_{\lambda}(t)^2}+
\E{\nu W_{\nu,1}(t)^2}}{2(\nu+\mu)}
=\frac{(2\nu+\mu)}{2(\nu+\mu)}t.
\end{equation*}
In addition, by \cite[Theorem~7.3]{pang2007martingale}, we know that $(Q^n(\cdot), Y^n(\cdot))$
satisfies a one-dimensional reflection mapping; see \cite[Section~6.2]{chen2001fundamentals} for
background of the reflection mapping. That is, $Y^n(\cdot)$ is the unique nondecreasing
nonnegative process such that $Q^n(t)\leq K^n$, \eqref{Eq:difQ} holds and
$$\int_{0}^{\infty} \ind{Q^n(t)<K^n}d Y^n(t)=0, $$
which is equivalent to
$$\int_{0}^{\infty} \ind{\hat Q^n(t)<\kappa}d Y^n(t)=0. $$
Now, combining Theorem~\ref{Th: existence}, Proposition~\ref{Pr:Martingales}, \eqref{eq:MCLT},
\eqref{eq:relation1}, and \eqref{eq:relation2} we derive that
\begin{equation*}
  ( \hat Z^n(\cdot),\hat Q^n(\cdot), \hat Y^n(\cdot) )
  \overset{d} \rightarrow (\hat Z(\cdot), \hat Q(\cdot), \hat Y(\cdot)) \quad \text{in}\quad
  D[0,\infty)^3,
\end{equation*}
where the vector
$(\hat Z(\cdot),\hat Q(\cdot), \hat Y(\cdot))$
 is characterized by
\eqref{eq:SDE}. This concludes the proof of Theorem~\ref{Th:diffusion}.
\end{proof}

\subsection{Proofs for Section~\ref{sec:Difover}}
\begin{proof}[Proof of Proposition~\ref{prop:hugear}]
Adding and subtracting the terms $\frac{\nu}{\nu+\mu}K^n$, $\mu \frac{\nu}{\nu+\mu}K^n t$, $\nu
\frac{\nu}{\nu+\mu}K^nt$, and the means of the Poisson processes in \eqref{eq:newZ}, we have that
\begin{align*}
 \Zfn(t)-\frac{\nu}{\nu+\mu}K^n= \Zfn(0)-\frac{\nu}{\nu+\mu}K^n+
\left(N_{\nu}^f\left( n \int_{0}^{t}(\frac{K^n-\Zfn(s)}{n})ds\right)-
\nu n \int_{0}^{t}(\frac{K^n-\Zfn(s)}{n})ds
\right)\\
-\left(N_{\mu}^f\left(\mu n \int_{0}^{t} \frac{\Zfn(s)}{n}\Min \frac{M^n}{n}  ds\right)-
\mu n \int_{0}^{t} \frac{\Zfn(s)}{n}\Min \frac{M^n}{n}  ds
\right)
+\nu  \int_{0}^{t}(K^n-\Zfn(s)+\frac{\nu}{\nu+\mu}K^n)ds\\-
\mu  \int_{0}^{t} (\Zfn(s)-\frac{\nu}{\nu+\mu}K^n)\Min \sqrt{n}\beta ds-(\nu+\mu)
\frac{\nu}{\nu+\mu}K^n t.
\end{align*}
Recalling that $\hat{\mathcal{M}}_{r}^n(\cdot):=\frac{N_{r}^f (\cdot)-(r\cdot)}{\sqrt{n}}$ and
dividing the last equation by $\sqrt{n}$, we have that
\begin{align*}
 \hat{Z}^n_f(t)= \hat{Z}^n_f(0)+
\hat{\mathcal{M}}_{\nu}^n\left(n \int_{0}^{t}\frac{K^n-\Zfn(s)}{n}ds\right)
-\hat{\mathcal{M}}_{\mu}^n\left( n \int_{0}^{t} \frac{\Zfn(s)}{n}\Min \frac{M^n}{n}  ds\right)\\
-\nu  \int_{0}^{t} \hat{Z}^n_f(s)ds-
\mu  \int_{0}^{t}  \hat{Z}^n_f(s)\Min \beta ds+\frac{1}{\sqrt{n}}
\left(\nu K^n-\nu K^n\right) t.
\end{align*}
Observing that the quantity $\int_{0}^{t}\frac{K^n-\Zfn(s)}{n}ds$ is stochastically bounded and
allowing $n \rightarrow \infty$ in the last equation, we derive
\begin{equation*}
\hat{Z}_f(t)= -\nu \int_{0}^{t} \hat{Z}_f(s)ds -
\mu\int_{0}^{t} \hat{Z}_f(s)\Min \beta ds+
\sqrt{\nu\left(K-\frac{\nu K}{\nu+\mu}\right)}W_1(t)
-\sqrt{\mu\left(\frac{\nu K}{\nu+\mu}\right)}W_2(t),
\end{equation*}
where $W_1(t)$ and $W_2(t)$ are (independent) standard Brownian motions. Last, by the properties
of Brownian motion, we get
\begin{equation*}
\sqrt{\nu\left(K-\frac{\nu K}{\nu+\mu}\right)}W_1(t)
-\sqrt{\mu\left(\frac{\nu K}{\nu+\mu}\right)}W_2(t)
\overset{d}= \sqrt{\frac{2 \nu \mu K }{\nu+\mu}}W(t),
\end{equation*}
where $W(t)$ is a standard Brownian motion.
\end{proof}

\subsubsection{Proof of Theorem~\ref{prop:fullparking}}
The rest of this section gives a proof of Theorem~\ref{prop:fullparking}. We first
show a bound for the process
$E^n(\cdot)$.

\begin{proposition}\label{Prop:boundEmpty}
Let $T>0$. We have that for any $\epsilon>0$ there exists $n_{\epsilon}$ such that
\begin{equation*}
\Prob{\sup_{0\leq t\leq T} E^n(t) \leq L(nT)^{1/4}+
L\log(nT) }>1-\epsilon,
\end{equation*}
for any $n \geq n_{\epsilon}$, where $L$ is a positive constant.
\end{proposition}
Before we proceed to the proof of Proposition~\ref{Prop:boundEmpty}, we show a preliminary
result.
\begin{lemma}\label{lem:bound}
Let $E_{M}(t)$ denote the queue length process in an $M/M/1$ queue at time $t\geq 0$, with
arrival rate $\nu K$ and service rate $\lambda$ such that $\nu K<\lambda$.
For any $T>0$, we have that
\begin{equation*}
\sup\limits_{0\leq t\leq nT} E_M(t) \leq L(nT)^{1/4}+L\log(nT),
\end{equation*}
almost surely as $n \rightarrow \infty$, where $L$ is a positive constant.
\end{lemma}
\begin{proof}
The proof of this lemma is based on results in \cite{chen2001fundamentals}.
By \cite[Theorem~6.16]{chen2001fundamentals}, there exists a reflected (at zero) Brownian motion,
$\tilde{ E}_M(\cdot)$ with drift $\nu K-\lambda<0$ such that
\begin{align*}
 \sup\limits_{0\leq t\leq nT}| E_M(t)-\tilde{ E}_M(t)|= o\left((nT)^{1/4}\right),
\end{align*}
or equivalently
\begin{align}\label{eq:strongA}
 \sup\limits_{0\leq t\leq nT}| E_M(t)-\tilde{ E}_M(t)|\leq  L'(nT)^{1/4},
\end{align}
for all $L'>0$, almost surely as $n \rightarrow \infty$. Further, by
\cite[Theorem~6.3]{chen2001fundamentals},
\begin{align*}
 \sup\limits_{0\leq t\leq nT}|\tilde{ E}_M(t)|= O\left(\log(nT)\right),
\end{align*}
or alternatively there exists $L>0$ such that
\begin{align}\label{eq:RBMProxy}
 \sup_{0\leq t \leq nT} |\tilde{E}_M(t)|\leq L \log(nT),
\end{align}
almost surely as $n \rightarrow \infty$.

Now, using the triangle inequality leads to
\begin{align*}
  \sup\limits_{0\leq t\leq nT} E_M(t) &=
  \sup\limits_{0\leq t\leq nT}| E_M(t)-\tilde{ E}_M(t)+\tilde{E}_M(t)|\\
 &\leq \sup\limits_{0\leq t\leq nT}| E_M(t)-\tilde{ E}_M(t)|
  +\sup\limits_{0\leq t\leq nT} |\tilde{E}_M(t)|.
\end{align*}
Applying \eqref{eq:strongA} (by choosing $L'=L$) and \eqref{eq:RBMProxy} in the last inequality, we
have that
\begin{equation*}
\sup\limits_{0\leq t\leq nT} E_M(t) \leq L (nT)^{1/4}+L\log(nT),
\end{equation*}
almost surely as $n \rightarrow \infty$.
\end{proof}
Now, we are ready to show Proposition~\ref{Prop:boundEmpty}.
\begin{proof}[Proof of Proposition~\ref{Prop:boundEmpty}]
Let $E_{M}^n(t)$ denote the queue length process in an $M/M/1$ queue at time $t\geq 0$, with
arrival rate $n \nu K$ and service rate $n \lambda$.
Using standard coupling arguments, it can be shown that
$\sup\limits_{0\leq t\leq T} E^n(t) \leq \sup\limits_{0\leq t\leq T} E^n_M(t)$ almost surely.
Further, we have that
$\sup\limits_{0\leq t\leq T} E^n_M(t)\overset{d}=
 \sup\limits_{0\leq t\leq nT} E_M(t)$. Hence, for any $q^n>0$,
 \begin{align*}
  \Prob{ \sup\limits_{0\leq t\leq T} E^n(t) \leq q^n }\geq
  \Prob{ \sup\limits_{0\leq t\leq nT} E_M(t) \leq q^n }.
 \end{align*}
Choosing $q^n=L(nT)^{1/4}+L\log(nT)$ and applying
Lemma~\ref{lem:bound}, we have that for any $\epsilon>0$ there exists $n_{\epsilon}$ such that
 \begin{align*}
  \Prob{ \sup\limits_{0\leq t\leq nT} E_M(t) \leq
  L(nT)^{1/4}+L\log(nT) }>1-\epsilon,
 \end{align*}
for $n>n_\epsilon$. This concludes the proof as $\epsilon$ is arbitrary.
\end{proof}

\begin{remark}
Define the sample path set $\mathcal{G}^n\subseteq \Omega$ such that
$$ \mathcal{G}^n:=\left\{\omega\in \Omega:
\sup\limits_{0\leq t\leq T} E^n(t) \leq L(nT)^{1/4}+L\log(nT)  \right\}.$$
By Proposition~\ref{Prop:boundEmpty} follows that
$\Prob{ \mathcal{G}^n}\rightarrow 1$, an $n \rightarrow \infty$. In the sequel, we assume that
$\omega \in \mathcal{G}^n$.
\end{remark}

\begin{proof}[Proof of Theorem~\ref{prop:fullparking}.]
Rewriting \eqref{eq:keyrelation}
for the $n^{th}$ system and assuming without loss of generality $\Prob{E^n(0)=0}=1$, yields
\begin{align*}
Z^n(t)&=Z^n(0)-E^n(t) + N_{\nu,2}\left(\int_{0}^{t}
\left(K^n-E^n(s)-Z^n(s)\right) ds\right)
-N_{\mu}\left(\int_{0}^{t} Z^n(s)\Min M^n ds\right).
\end{align*}
Let $T>0$ and $\omega \in \mathcal{G}^n$. We have that
\begin{align*}
  0\leq \frac{E^n(t)}{\sqrt{n}} \leq
  \sup\limits_{0\leq t\leq T} \frac{E^n(t)}{\sqrt{n}}
  \leq \frac{L(nT)^{1/4}+L\log(nT)}{\sqrt{n}}\rightarrow 0,
\end{align*}
as $n\rightarrow \infty$. The result now follows by adapting the proof of Proposition~\ref{prop:hugear}.
\end{proof}

\subsection{Proofs for Section~\ref{sec:DifsmallparkingT}}\label{sec:largeparking}
\begin{proof}[Proof of Proposition~\ref{Prop: heavy traffic}]
First, we note that it is enough to show the result for $M=1$. Then we replace $\mu$ by $\mu M$; see \cite[Remark~5]{ward03}.

Scale the time by $n$, and add and subtract the means of the Poisson processes in
\eqref{eq:numberQ} and \eqref{eq:numberU} to obtain
\begin{align*}
Q^n(n t)=Q^n(0)+
\left( N_{\lambda^n} (n t) - \lambda^n n t \right)
-\left(
 N_{\nu^n}\left(
 n \int_{0}^{t}Q^n(ns) ds
 \right) - n \nu^n \int_{0}^{t}Q^n(ns)ds
 \right)\\+\lambda^n n t- n \nu^n \int_{0}^{t}Q^n(ns)ds
\end{align*}
and
\begin{align*}
Z^n(n t)=Z^n(0)+\left(N_{\lambda^n} (n t) -\lambda^n n t\right)
-\left(
N_{\mu}\left(
n \int_{0}^{t} \ind{Z^n(ns)>0} ds
\right)-n \mu \int_{0}^{t} \ind{Z^n(ns)>0} ds
\right)  \\
 - \left(
 N_{\nu^n,1}\left(
n\int_{0}^{t} Z^n(ns) ds
\right)-n \nu^n\int_{0}^{t} Z^n(ns) ds
\right)+\lambda^n n t-n \mu \int_{0}^{t} \ind{Z^n(ns)>0} ds
-n \nu^n\int_{0}^{t} Z^n(ns) ds.
\end{align*}
Define $\bar{Q}^n(t)=Q^n(nt)/n$, $\bar{Z}^n(t)=Z^n(nt)/n$, and recall that $\mathcal{M}_{r}(\cdot
):=N_{r} (\cdot)-(r\cdot)$.
Adding and subtracting the terms $\mu n$ and $\lambda^n n t$ in the first equation yields
\begin{align*}
Q^n(n t)- \mu n=Q^n(0)- \mu n
+\mathcal{M}_{\lambda^n}(n t)
-\mathcal{M}_{1}\left(n\int_{0}^{t}\bar {Q}^n(s) ds\right)
-c\mu \sqrt{n} t- \int_{0}^{t} \left(Q^n(ns)-\mu n\right) ds
\end{align*}
and
\begin{align*}
Z^n(n t)=Z^n(0)
+\mathcal{M}_{\lambda^n }(n t)
-\mathcal{M}_{\mu}\left(
n \int_{0}^{t} \ind{\bar{Z}^n(s)>0} ds\right)
-\mathcal{M}_{1,1}\left(
n \int_{0}^{t} \bar{Z}^n(s) ds\right)\\
+\mu(1-\frac{c}{\sqrt{n}})n t
-\mu n \int_{0}^{t} \ind{\bar{Z}^n(s)>0} ds
- \int_{0}^{t} \left(Z^n(n s)\right) ds.
\end{align*}
Dividing the last equations by $\sqrt{n}$ and observing that
\begin{equation*}
\mu(1-\frac{c}{\sqrt{n}})n t
-\mu n \int_{0}^{t} \ind{\bar{Z}^n(s)>0} ds=
-c\mu \sqrt{n} t+\mu n \int_{0}^{t} \ind{\bar{Z}^n(s)=0}ds,
\end{equation*}
we obtain
\begin{align*}
\tilde Q^n(t)=\tilde Q^n(0)
+\hat{ \mathcal{M}}^n_{\lambda^n}(n t)
-\hat{\mathcal{M}}^n_{1}\left(n\int_{0}^{t}\bar {Q}^n(s) ds\right)
-c\mu t- \int_{0}^{t} \tilde{Q}^n(s)  ds
\end{align*}
and
\begin{align*}
\tilde{Z}^n(t)=\tilde{Z}^n(0)
+\hat{\mathcal{M}}^n_{\lambda^n }(nt)
-\hat{\mathcal{M}}^n_{\mu}\left( n
 \int_{0}^{t} \ind{\bar{Z}^n(s)>0} ds\right)
-\hat{\mathcal{M}}^n_{1,1}\left( n\int_{0}^{t} \bar{Z}^n(s) ds\right)\\
-c\mu t+\mu \sqrt{n} \int_{0}^{t} \ind{\bar{Z}^n(s)=0} ds
- \int_{0}^{t} \tilde{Z}^n( s)ds.
\end{align*}
Now, taking the limit $n\rightarrow\infty$ and using the reflection mapping
\cite{chen2001fundamentals}, we derive
\begin{align*}
  d\tilde{Q}(t) &=-(c\mu+\tilde{Q}(t)) dt+\sqrt{2\mu} dW_{\tilde Q}(t) ,\\
    d\tilde{Z}(t)& =-(c\mu+\tilde{Z}(t)) dt+ \sqrt{2\mu} dW_{\tilde Z}(t) +d\tilde{Y}(t),
\end{align*}
where $\int_{0}^{\infty} \ind{\tilde{Z}(t)>0}d\tilde{Y}(t)=0$. Further,
$\E{W_{\tilde Z}(t) W_{\tilde Q}(t)}= \frac{\mu M t}{2 \mu M}=t/2$.
\end{proof}

\begin{proof}[Proof of Proposition~\ref{prop:Difover}]
First, as in Proposition~\ref{Prop: heavy traffic},
we note that it is enough to show the result for $M=1$ and the replace $\mu$ by $\mu M$.
Scale the time by $n$, and  add and subtract the means of the Poisson processes in
\eqref{eq:numberQ} and \eqref{eq:numberU} to obtain
\begin{align*}
Q^n(n t)=Q^n(0)+
\left( N_{\lambda} (n t) - \lambda n t \right)
-\left(
 N_{\nu^n}\left(
 n \int_{0}^{t}Q^n(ns) ds
 \right) - n \nu^n \int_{0}^{t}Q^n(ns)ds
 \right)\\+\lambda n t- n \nu^n \int_{0}^{t}Q^n(ns)ds
\end{align*}
and
\begin{align*}
Z^n(n t)=Z^n(0)+\left(N_{\lambda} (n t) -\lambda n t\right)
-\left(
N_{\mu}\left(
n \int_{0}^{t} \ind{Z^n(ns)>0} ds
\right)-n \mu \int_{0}^{t} \ind{Z^n(ns)>0} ds
\right)  \\
 - \left(
 N_{\nu^n,1}\left(
n\int_{0}^{t} Z^n(ns) ds
\right)-n \nu^n\int_{0}^{t} Z^n(ns) ds
\right)+\lambda n t-n \mu \int_{0}^{t} \ind{Z^n(ns)>0} ds
-n \nu^n\int_{0}^{t} Z^n(ns) ds.
\end{align*}
Adding and subtracting the terms $\lambda n$ and $\lambda n t$ in the first equation, and the
terms $(\lambda-\mu) n$ and $(\lambda-\mu) n t$ in the second equation  yields
\begin{align*}
Q^n(n t)- \lambda n=Q^n(0)- \lambda n
+\mathcal{M}_{\lambda}(n t)
-\mathcal{M}_{1}\left(n\int_{0}^{t}\bar {Q}^n(s) ds\right)
- \int_{0}^{t} \left(Q^n(ns)-\lambda n\right)  ds
\end{align*}
and
\begin{align*}
Z^n(n t)-(\lambda-\mu) n=Z^n(0)-(\lambda-\mu) n
+\mathcal{M}_{\lambda }(n t)
-\mathcal{M}_{\mu}\left(
n \int_{0}^{t} \ind{\bar{Z}^n(s)>0} ds\right)
-\mathcal{M}_{1,1}\left(
n \int_{0}^{t} \bar{Z}^n(s) ds\right)\\
+\lambda n t-(\lambda-\mu) n t
-\mu n \int_{0}^{t} \ind{\bar{Z}^n(s)>0} ds
- \int_{0}^{t} \left(Z^n(n s)-(\lambda-\mu) n\right) ds.
\end{align*}
Dividing the last equations by $\sqrt{n}$ and observing that
\begin{equation*}
\lambda n t-(\lambda-\mu) n t
-\mu n \int_{0}^{t} \ind{\bar{Z}^n(s)>0} ds=
\mu n \int_{0}^{t} \ind{\bar{Z}^n(s)=0}ds,
\end{equation*}
we obtain
\begin{align*}
\tilde Q^n_o(t)=\tilde Q^n_o(0)
+\hat{ \mathcal{M}}^n_{\lambda}(n t)
-\hat{\mathcal{M}}^n_{1}\left(n\int_{0}^{t}\bar {Q}^n(s) ds\right)
- \int_{0}^{t} \tilde{Q}^n_o(s)  ds
\end{align*}
and
\begin{align*}
\tilde{Z}^n_o(t)=\tilde{Z}^n_o(0)
+\hat{\mathcal{M}}^n_{\lambda }(n t)
-\hat{\mathcal{M}}^n_{\mu}\left(n
 \int_{0}^{t} \ind{\bar{Z}^n(s)>0} ds\right)
-\hat{\mathcal{M}}^n_{1,1}\left( n\int_{0}^{t} \bar{Z}^n(s) ds\right)\\
-\mu \sqrt{n} \int_{0}^{t} \ind{\bar{Z}^n_o(s)=0} ds
- \int_{0}^{t} \tilde{Z}^n( s)ds.
\end{align*}
By the overloaded assumption, it follows that
$\sqrt{n} \int_{0}^{t} \ind{\bar{Z}^n(s)=0} ds \rightarrow 0$, as $n \rightarrow \infty$;
\cite{ward03}. Now, taking the limit $n\rightarrow\infty$, we derive
\begin{align*}
d\tilde{Q}_o(t) &=-\tilde{Q}_o(t) dt+\sqrt{\lambda}dW_1(t) - \sqrt{\mu}dW(t),\\
d\tilde{Z}_o(t)& =-\tilde{Z}_o(t) dt+
\sqrt{\lambda}dW_1(t)-\sqrt{\mu}dW_2(t)-\sqrt{\lambda-\mu}dW_3(t),
\end{align*}
and we can write
$\sqrt{\mu}W(t)\overset{d}= \sqrt{\lambda-\mu}W_3(t)+\sqrt{\mu}W_4(t)$,
where $W_i$ are independent standard Brownian motions.
\end{proof}

\section{Concluding Remarks }
This paper proposes to model an electric vehicle charging model by using layered queueing networks.
We develop several bounds and approximations for the number of uncharged EVs in the system and the probability that an EV leaves the charging station with fully charged battery. In the numerical examples, it seems that a modification of the approximation for a full parking lot leads to a good approximation.
Further, the fluid approximation seems to be good in most cases and we believe that the diffusion approximation in the Halfin-Whitt regime will improve the fluid approximation. Unfortunately, the exact (or even numerical) solution of \eqref{eq:BAR} seems very hard.

From an application standpoint, it is important to remove various model assumptions. If parking and charging times are given by the (possibly dependent) generally
distributed random variables $B$ and $D$, we can develop a
measure-valued fluid model by extending \cite{gromoll2006impact}. In addition, we can include in the model time-varying arrival rates, multiple EV types, and multiple parking lots, thus extending results in \cite{remerova14}. Moreover,  the distribution grid (low-voltage network) plays a crucial role and it should be included in the model; see
\cite{aveklouriselectric} for a heuristic approach. For another application in EV-charging including the distribution grid, see  \cite{carvalho2015critical} where simulation results are presented for a Markovian model.

\addcontentsline{toc}{section}{Acknowledgements}
\section*{Acknowledgements}
The research of
Angelos Aveklouris is funded by a TOP grant of the Netherlands Organization
for
Scientific Research (NWO) through project 613.001.301. The research of Maria
Vlasiou is supported by the NWO MEERVOUD grant 632.003.002.
The research of Bert Zwart is partly supported by the NWO VICI grant
639.033.413.

\phantomsection
\addcontentsline{toc}{section}{References}
\begin{footnotesize}

\bibliography{Mybibliography_Angelos_Ab}

\begin{thebibliography}{10}

\bibitem{aldous1986}
D.~Aldous.
\newblock Some interesting processes arising as heavy traffic limits in an
  {M/M/$\infty$} storage process.
\newblock {\em Stochastic processes and their applications}, 22(2):291--313,
  1986.

\bibitem{aveklouriselectric}
A.~Aveklouris, Y.~Nakahira, M.~Vlasiou, and B.~Zwart.
\newblock Electric vehicle charging: a queueing approach.
\newblock {\em SIGMETRICS Performance Evaluation Review}, 45(2):33--35, 2017.

\bibitem{aveklourisSSC}
A.~Aveklouris, M.~Vlasiou, J.~Zhang, and B.~Zwart.
\newblock Heavy-traffic approximations for a layered network with limited
  resources.
\newblock {\em Probability and Mathematical Statistics}, 37(2):497--532, 2017.

\bibitem{aveklourisstochastic}
A.~Aveklouris, M.~Vlasiou, and B.~Zwart.
\newblock A stochastic resource-sharing network for electric vehicle charging.
\newblock {\em Preprint arXiv:1711.05561}.

\bibitem{bayram2013electric}
I.~S. Bayram, G.~Michailidis, M.~Devetsikiotis, and F.~Granelli.
\newblock Electric power allocation in a network of fast charging stations.
\newblock {\em IEEE J. Sel. Areas Commun.}, 31(7):1235--1246, 2013.

\bibitem{billingsley1999convergence}
P.~Billingsley.
\newblock {\em Convergence of probability measures}.
\newblock Wiley, New York, second edition, 1999.

\bibitem{bo2011}
L.~Bo, Y.~Wang, and X.~Yang.
\newblock First passage times of (reflected) ornstein-uhlenbeck processes over
  random jump boundaries.
\newblock {\em Journal of Applied Probability}, 48(3):723--732, 2011.

\bibitem{browne1995}
S.~Browne, W.~Whitt, and J.~Dshalalow.
\newblock Piecewise-linear diffusion processes.
\newblock {\em Advances in queueing: Theory, methods, and open problems},
  4:463--480, 1995.

\bibitem{carvalho2015critical}
R.~Carvalho, L.~Buzna, R.~Gibbens, and F.~Kelly.
\newblock Critical behaviour in charging of electric vehicles.
\newblock {\em New Journal of Physics}, 17(9):095001, 2015.

\bibitem{chen2001fundamentals}
H.~Chen and D.~D. Yao.
\newblock {\em Fundamentals of queueing networks: performance, asymptotics, and
  optimization}, volume~46.
\newblock New York: Springer-Verlag, 2001.

\bibitem{clement2010impact}
K.~Clement-Nyns, E.~Haesen, and J.~Driesen.
\newblock The impact of charging plug-in hybrid electric vehicles on a
  residential distribution grid.
\newblock {\em IEEE Trans. Power Syst.}, 25(1):371--380, 2010.

\bibitem{cohen1975}
J.~Cohen.
\newblock The wiener-hopf technique in applied probability.
\newblock {\em Journal of Applied Probability}, 12(S1):145--156, 1975.

\bibitem{cohen2000}
J.~W. Cohen and O.~J. Boxma.
\newblock {\em Boundary value problems in queueing system analysis}, volume~79.
\newblock Elsevier, 2000.

\bibitem{dai2011}
J.~Dai and A.~Dieker.
\newblock Nonnegativity of solutions to the basic adjoint relationship for some
  diffusion processes.
\newblock {\em Queueing Systems}, 68(3-4):295, 2011.

\bibitem{dai2011ST}
J.~G. Dai, M.~Miyazawa, et~al.
\newblock Reflecting brownian motion in two dimensions: Exact asymptotics for
  the stationary distribution.
\newblock {\em Stochastic Systems}, 1(1):146--208, 2011.

\bibitem{dorsman2013marginal}
J.~L. Dorsman, O.~J. Boxma, and M.~Vlasiou.
\newblock Marginal queue length approximations for a two-layered network with
  correlated queues.
\newblock {\em Queueing Systems}, 75(1):29--63, 2013.

\bibitem{dorsman2015heavy}
J.~L. Dorsman, M.~Vlasiou, and B.~Zwart.
\newblock Heavy-traffic asymptotics for networks of parallel queues with
  markov-modulated service speeds.
\newblock {\em Queueing Systems}, 79(3-4):293--319, 2015.

\bibitem{durrett1996stochastic}
R.~Durrett.
\newblock {\em Stochastic calculus: a practical introduction}, volume~6.
\newblock CRC press, 1996.

\bibitem{ethier1986markov}
S.~N. Ethier and T.~G. Kurtz.
\newblock {\em Markov processes: characterization and convergence}.
\newblock Wiley, New York, 1986.

\bibitem{fleming1994heavy}
P.~Fleming, A.~Stolyar, and B.~Simon.
\newblock Heavy traffic limit for a mobile phone system loss model.
\newblock In {\em Proc. 2nd Int’l Conf. on Telecomm. Syst. Mod. and Analysis.
  Nashville, TN}, 1994.

\bibitem{gromoll2006impact}
H.~C. Gromoll, P.~Robert, B.~Zwart, and R.~Bakker.
\newblock The impact of reneging in processor sharing queues.
\newblock {\em SIGMETRICS Performance Evaluation Review}, 34(1):87--96, 2006.

\bibitem{2025scenario}
G.~Hoogsteen, A.~Molderink, J.~L. Hurink, G.~J. Smit, B.~Kootstra, and
  F.~Schuring.
\newblock Charging electric vehicles, baking pizzas, and melting a fuse in
  lochem.
\newblock {\em CIRED-Open Access Proceedings Journal}, 2017(1):1629--1633,
  2017.

\bibitem{Energy17}
{International Energy Agency}.
\newblock {G}lobal {EV} outlook 2017.
\newblock 2017.

\bibitem{kang2015}
W.~Kang.
\newblock Fluid limits of many-server retrial queues with nonpersistent
  customers.
\newblock {\em Queueing Systems}, 79(2):183--219, 2015.

\bibitem{kang2014}
W.~Kang, K.~Ramanan, et~al.
\newblock Characterization of stationary distributions of reflected diffusions.
\newblock {\em The Annals of Applied Probability}, 24(4):1329--1374, 2014.

\bibitem{karatzas1991}
I.~Karatzas and S.~Shreve.
\newblock {\em Brownian motion and stochastic calculus}, volume 113.
\newblock New York, Springer-Verlag, 1991.

\bibitem{kelly2014stochastic}
F.~Kelly and E.~Yudovina.
\newblock {\em Stochastic networks}, volume~2.
\newblock Cambridge University Press, 2014.

\bibitem{kempker2016optimization}
P.~Kempker, N.~v. Dijk, W.~Scheinhardt, H.~v.~d. Berg, and J.~Hurink.
\newblock Optimization of charging strategies for electric vehicles in
  {PowerMatcher}-driven smart energy grids.
\newblock In {\em Proc. 9th EAI Int. Conf. Perf. Eval. Meth. and Tools}, 2015.

\bibitem{li2012modeling}
G.~Li and X.-P. Zhang.
\newblock Modeling of plug-in hybrid electric vehicle charging demand in
  probabilistic power flow calculations.
\newblock {\em IEEE Trans. Smart Grid}, 3(1):492--499, 2012.

\bibitem{pang2007martingale}
G.~Pang, R.~Talreja, and W.~Whitt.
\newblock Martingale proofs of many-server heavy-traffic limits for markovian
  queues.
\newblock {\em Probability Surveys}, 4:193--267, 2007.

\bibitem{Reiman82}
M.~I. Reiman.
\newblock The heavy traffic diffusion approximation for sojourn times in
  jackson networks.
\newblock {\em Applied probability--computer science: the interface {II}},
  pages 409--421, 1982.

\bibitem{remerova14}
M.~Remerova, J.~Reed, and B.~Zwart.
\newblock Fluid limits for bandwidth-sharing networks with rate constraints.
\newblock {\em Mathematics of Operations Research}, 39(3):746--774, 2014.

\bibitem{robert2013}
P.~Robert.
\newblock {\em Stochastic networks and queues}, volume~52.
\newblock Springer Science \& Business Media, 2013.

\bibitem{rolia1995method}
J.~A. Rolia and K.~C. Sevcik.
\newblock The method of layers.
\newblock {\em IEEE Trans. Softw. Eng.}, 21(8):689--700, 1995.

\bibitem{said2013queuing}
D.~Said, S.~Cherkaoui, and L.~Khoukhi.
\newblock Queuing model for {EVs} charging at public supply stations.
\newblock {\em Wireless Commun. and Mobile Computing Conf.}, 1(5):65--70, 2013.

\bibitem{sortomme2011coordinated}
E.~Sortomme, M.~M. Hindi, S.~J. MacPherson, and S.~Venkata.
\newblock Coordinated charging of plug-in hybrid electric vehicles to minimize
  distribution system losses.
\newblock {\em IEEE Trans. Smart Grid}, 2(1):198--205, 2011.

\bibitem{su2012performance}
W.~Su and M.-Y. Chow.
\newblock Performance evaluation of an {EDA}-based large-scale plug-in hybrid
  electric vehicle charging algorithm.
\newblock {\em IEEE Trans. Smart Grid}, 3(1):308--315, 2012.

\bibitem{Turitsyn2010}
K.~Turitsyn, N.~Sinitsyn, S.~Backhaus, and M.~Chertkov.
\newblock Robust broadcast-communication control of electric vehicle charging.
\newblock In {\em 2010 1st IEEE Int. Conf. Smart Grid Commun.}, pages 203--207,
  2010.

\bibitem{van2001web}
R.~D. van~der Mei, R.~Hariharan, and P.~Reeser.
\newblock Web server performance modeling.
\newblock {\em Telecommunication Systems}, 16(3-4):361--378, 2001.

\bibitem{van2009dynamic}
W.~van~der Weij, S.~Bhulai, and R.~van~der Mei.
\newblock Dynamic thread assignment in web server performance optimization.
\newblock {\em Performance Evaluation}, 66(6):301--310, 2009.

\bibitem{ward03}
A.~R. Ward and P.~W. Glynn.
\newblock A diffusion approximation for a markovian queue with reneging.
\newblock {\em Queueing Systems}, 43(1):103--128, 2003.

\bibitem{whitt2007}
W.~Whitt.
\newblock Proofs of the martingale {FCLT}.
\newblock {\em Probability Surveys}, 4:268--302, 2007.

\bibitem{woodside1995stochastic}
M.~Woodside, J.~E. Neilson, D.~C. Petriu, and S.~Majumdar.
\newblock The stochastic rendezvous network model for performance of
  synchronous client-server-like distributed software.
\newblock {\em IEEE Trans. Comput.}, 44(1):20--34, 1995.

\bibitem{xing2009}
X.~Xing, W.~Zhang, and Y.~Wang.
\newblock The stationary distributions of two classes of reflected
  ornstein--uhlenbeck processes.
\newblock {\em Journal of Applied Probability}, 46(3):709--720, 2009.

\bibitem{Swapping17}
P.~You, Y.~Sun, J.~Pang, S.~Low, and M.~Chen.
\newblock Battery swapping assignment for electric vehicles: A bipartite
  matching approach.
\newblock {\em SIGMETRICS Performance Evaluation Review}, 45(2):85--87, 2017.

\bibitem{yudovina2015socially}
E.~Yudovina and G.~Michailidis.
\newblock Socially optimal charging strategies for electric vehicles.
\newblock {\em IEEE Trans. Autom. Control}, 60(3):837--842, 2015.

\bibitem{zeltyn2005call}
S.~Zeltyn.
\newblock {\em Call centers with impatient customers: exact analysis and
  many-server asymptotics of the M/M/n+ G queue}.
\newblock Technion-Israel Institute of Technology, Faculty of Industrial and
  Management Engineering, 2005.

\end{thebibliography}
\bibliographystyle{abbrv}
\end{footnotesize}
\end{document}